\documentclass[11pt,oneside,reqno]{amsart}
\usepackage[latin9]{inputenc}
\usepackage{geometry}
\usepackage{mathrsfs}
\usepackage{amstext}
\usepackage{amsthm}
\usepackage{amssymb}
\usepackage{graphicx}
\usepackage{setspace}
\usepackage[unicode=true,
 bookmarks=false,
 breaklinks=false,pdfborder={0 0 1},backref=section,colorlinks=false]
 {hyperref}

\makeatletter

\newcommand{\lyxdot}{.}

\numberwithin{equation}{section}
\numberwithin{figure}{section}
\theoremstyle{plain}
\newtheorem{thm}{\protect\theoremname}[section]
\theoremstyle{plain}
\newtheorem*{assumption*}{\protect\assumptionname}
\theoremstyle{remark}
\newtheorem{rem}[thm]{\protect\remarkname}
\theoremstyle{remark}
\newtheorem{notation}[thm]{\protect\notationname}
\theoremstyle{plain}
\newtheorem{lem}[thm]{\protect\lemmaname}
\theoremstyle{definition}
\newtheorem{defn}[thm]{\protect\definitionname}
\theoremstyle{plain}
\newtheorem{prop}[thm]{\protect\propositionname}
\newenvironment{lyxlist}[1]
	{\begin{list}{}
		{\settowidth{\labelwidth}{#1}
		 \setlength{\leftmargin}{\labelwidth}
		 \addtolength{\leftmargin}{\labelsep}
		 }}
	{\end{list}}
\theoremstyle{remark}
\newtheorem*{acknowledgement*}{\protect\acknowledgementname}

\usepackage{amsthm}\usepackage{mathrsfs}\usepackage{nameref}\usepackage{cite}\usepackage{upgreek}
\usepackage{stmaryrd}
\usepackage{etoolbox}

\DeclareFontFamily{U}{matha}{\hyphenchar\font45}
\DeclareFontShape{U}{matha}{m}{n}{
      <5> <6> <7> <8> <9> <10> gen * matha
      <10.95> matha10 <12> <14.4> <17.28> <20.74> <24.88> matha12
      }{}
\DeclareSymbolFont{matha}{U}{matha}{m}{n}
\DeclareFontSubstitution{U}{matha}{m}{n}

\DeclareFontFamily{U}{mathx}{\hyphenchar\font45}
\DeclareFontShape{U}{mathx}{m}{n}{
      <5> <6> <7> <8> <9> <10>
      <10.95> <12> <14.4> <17.28> <20.74> <24.88>
      mathx10
      }{}
\DeclareSymbolFont{mathx}{U}{mathx}{m}{n}
\DeclareFontSubstitution{U}{mathx}{m}{n}

\DeclareMathDelimiter{\vvvert}{0}{matha}{"7E}{mathx}{"17}

\DeclareMathAlphabet{\scal}{U}{dutchcal}{m}{n}

\numberwithin{equation}{section}

\def\th@plain{\thm@notefont{}\itshape}
\def\th@definition{\thm@notefont{}\normalfont}

\makeatother

\providecommand{\acknowledgementname}{Acknowledgement}
\providecommand{\assumptionname}{Assumption}
\providecommand{\definitionname}{Definition}
\providecommand{\lemmaname}{Lemma}
\providecommand{\notationname}{Notation}
\providecommand{\propositionname}{Proposition}
\providecommand{\remarkname}{Remark}
\providecommand{\theoremname}{Theorem}

\begin{document}
\title[Metastability of Ising and Potts Models in Large Volumes]{Metastability of Ising and Potts Models without External Fields in
Large Volumes at Low Temperatures}
\author{Seonwoo Kim and Insuk Seo}
\begin{abstract}
In this article, we investigate the energy landscape and metastable
behavior of the Ising and Potts models on two-dimensional square or
hexagonal lattices in the low temperature regime, especially in the
absence of an external magnetic field. The energy landscape of these
models without an external field is known to have a huge and complex
saddle structure between ground states. In the small volume regime
where the lattice is finite and fixed, the aforementioned complicated
saddle structure has been successfully analyzed in \cite{KS2} for
two or three dimensional square lattices when the inverse temperature
tends to infinity. In this article, we consider the\textit{ large
volume regime} where the size of the lattice grows to infinity. We
first establish an asymptotically sharp threshold such that the ground
states are metastable if and only if the inverse temperature is larger
than the threshold in a suitable sense. Then, we carry out a detailed
analysis of the energy landscape and rigorously establish the Eyring--Kramers
formula when the inverse temperature is sufficiently larger than the
previously mentioned sharp threshold. The proof relies on detailed
characterization of dead-ends appearing in the vicinity of optimal
transitions between ground states and on combinatorial estimation
of the number of configurations lying on a certain energy level.
\end{abstract}

\maketitle

\section{\label{sec1_Intro}Introduction}

Metastability is a ubiquitous phenomenon that occurs when a stochastic
system has multiple locally stable sets. It occurs in a wide class
of models in statistical mechanics such as interacting particle systems
\cite{BL ZRP,BDG}, spin systems \cite{BA-C,B-M,N-S Ising1}, small
random perturbations of dynamical systems \cite{FW}, and models in
numerical simulations such as the kinetic Monte Carlo \cite{DG-L-LP-N}
and stochastic gradient descent method \cite{GMZ}. We refer to the
bibliography of the listed references for more comprehensive literature.
We also refer to monographs such as \cite{BdH,O-V} and references
therein.

\subsubsection*{Ising and Potts model}

In this article, we are interested in the Ising/Potts model defined
on either a square or hexagonal lattice $\Lambda_{L}$ with side length
$L\in\mathbb{N}$ under the periodic boundary condition\footnote{We refer to Figure \ref{fig21} for rigorous definition of the finite
hexagonal lattice with periodic boundary conditions.}. For $q\ge2$, denote by $\{1,\,2,\,\dots,\,q\}$ the set of spins
so that we can get a \textit{spin configuration} by distributing spins
at the vertices of lattice $\Lambda_{L}$. Then, the Ising/Potts model
refers to the Gibbs measure on the space of spin configurations associated
with a certain form of Hamiltonian function (cf. \eqref{e_Ham}) at
inverse temperature $\beta$. In particular, the models with $q=2$
and $q\ge3$ are called the Ising and Potts models, respectively.
Henceforth, we assume that there is no external field acting on our
Ising/Potts model.

\subsubsection*{Ground states and metastability}

For $a\in\{1,\,\dots,\,q\}$, we denote by $\mathbf{a}$ the monochromatic
spin configuration consisting only of spin $a$. Then, we can readily
verify that the set of ground states associated with the Ising/Potts
Hamiltonian is $\mathcal{S}=\{\mathbf{1},\,\dots,\,\mathbf{q}\}$
and hence the Gibbs measure is concentrated on the set $\mathcal{S}$
as the inverse temperature $\beta$ tends to infinity (i.e., as the
temperature goes to $0$). Thus, we can expect that the associated
heat-bath Glauber dynamics exhibits metastability when $\beta$ is
sufficiently large, in the sense that the single-flip Glauber dynamics
starting at a ground state $\mathbf{a}\in\mathcal{S}$ spends a very
long time in a neighborhood of $\mathbf{a}$ before making a transition
to another ground state. Such a \textit{metastable transition} is
one of the primary concerns of the study of metastability and we are
specifically interested in the accurate quantification of the mean
of the metastable transition time. Such a precise estimate of the
mean transition time is called the \textit{Eyring--Kramers formula}
and establishing it requires deep understanding of the energy landscape,
e.g., detailed saddle structure between ground states, associated
with the Hamiltonian. The major difficulty confronted in the current
article lies on the fact that the saddle structure includes a huge
plateau with a large amount of dead-ends. 

We remark that other important problems in the analysis of metastability
include characterizing the typical transition paths and estimating
the spectral gap or mixing time. We shall not pursue these questions
in the current article and leave as future research program.

\subsubsection*{Metastability in small volumes at low temperatures}

The energy landscape of the Ising/Potts model on two-dimensional square
lattices in the small volume regime, i.e, when the side length $L$
of the lattice is large but fixed, was first analyzed in \cite{NZ},
where the energy barrier between ground states was exactly computed
under periodic and open boundary conditions. Based on this result,
the large deviation-type analysis of metastability in the low temperature
regime, i.e., $\beta\rightarrow\infty$ regime, has been carried out
via a robust pathwise approach-type method developed in \cite{NZB}.
This analysis has been extended in \cite{BGN1} where the authors
provide more refined characterization of the optimal transition paths
between ground states.

Finally, in the paper \cite{KS2} by the authors of the present work,
a complete characterization of the entire saddle structure has been
carried out on fixed two or three dimensional square lattices\footnote{Indeed, more generally, rectangular lattices were considered.}
with periodic or open boundary conditions. This level of detailed
understanding of the energy landscape enables us to deduce the Eyring\textit{--}Kramers
formula in the very low temperature regime. By adopting the methodology
developed in that article, metastability of the Blume--Capel model
with zero external field and zero chemical potential has also been
analyzed in \cite{Kim2}.

\subsubsection*{Main results}

In this article, we consider the Ising/Potts model in the \textit{large
volume regime}, i.e., the case when the side length $L$ grows to
infinity. We analyze, at a highly accurate level, the energy landscape
of the Ising/Potts model on the two-dimensional square or hexagonal
lattice of side length $L$ with periodic boundary conditions.

Note that the Gibbs measure is concentrated on the set $\mathcal{S}$
of ground states if $L$ is fixed and $\beta$ is sufficiently large,
since the entropy effect can be neglected in this regime and the energy
is the only dominating factor. However, if we assume that $L$ and
$\beta$ are both sufficiently large, we must consider the entropy
effect and the competition between energy and entropy should be carefully
quantified to determine whether the Gibbs measure is still concentrated
on $\mathcal{S}$. This precise quantification of the entropy-energy
competition is done in Theorem \ref{t_Gibbs1} where we establish
a zero-one law-type result. More precisely, we demonstrate that there
exists a constant\footnote{This constant $\gamma_{0}$ is $1/2$ and $2/3$ for square and hexagonal
lattices, respectively.} $\gamma_{0}$ such that 
\[
\begin{cases}
\text{if }\beta\ge\gamma\log L\text{ for }\gamma>\gamma_{0}\text{ then the Gibbs measure is concentrated on }\mathcal{S}\text{ and}\\
\text{if }\beta\le\gamma\log L\text{ for }\gamma<\gamma_{0}\text{ then the Gibbs measure is concentrated on }\mathcal{S}^{c}
\end{cases}
\]
as $L\rightarrow\infty$, where $\gamma$ denotes a constant independent
of $L$. Therefore, we find an interesting \textit{\emph{phase transition}}
at the critical inverse temperature $\beta_{0}(L)=\gamma_{0}\log L$. 

In view of the previous result, the sharp estimation of the mean of
transition time between ground states provides the Eyring\textit{--}Kramers
formula only when we asymptotically have $\beta\ge\gamma\log L$ for
some $\gamma>\gamma_{0}$. Indeed, we carry out the analysis of the
energy landscape and establish the Eyring\textit{--}Kramers formula
under $\beta\ge\gamma\log L$, $\gamma>\gamma_{1}$ for some constant\footnote{This constant $\gamma_{1}$ is $3$ and $10$ for square and hexagonal
lattices, respectively.} $\gamma_{1}$ which is larger than $\gamma_{0}$ due to technical
reasons.

\subsubsection*{Challenges in the proof}

For both square and hexagonal lattices with side length $L$ under
periodic boundary conditions, it can be shown (cf. \cite{NZ} for
the square lattice and Theorem \ref{t_Ebarrier} of the present work
for the hexagonal lattice) that the energy barrier between ground
states is $2L+2$. In the small volume regime considered in \cite{BGN1,KS2,NZ},
we can neglect all the configurations of energy larger than $2L+2$
since the number of such configurations is determined solely by $L$
(and hence fixed) and therefore, as $\beta\rightarrow\infty$, these
configurations have exponentially negligible mass with respect to
the Gibbs measure, compared to the configurations with energy less
than or equal to $2L+2$ along which typical metastable transitions
take place. However, for the models in the large volume regime, we
are no longer able to neglect these configurations, as the number
of such configurations also grows to infinity and hence the entropy
plays a role. This is the first primary difficulty in the study of
models in large volume when compared to the previous works; indeed,
a subtle combinatorial estimation on the number of configurations
on each energy level is required to overcome this difficulty.

We remark that the saddle structure for the Ising/Potts model without
external field forms a huge and complex plateau. The saddle plateau
consists of canonical configurations providing the main road in the
course of metastable transition, and a large amount of dead-ends attached
there. The analysis of canonical configurations is rather straightforward,
but we need to fully understand the complex structure of the dead-ends
in order to obtain quantitative results such as the Eyring--Kramers
formula. In the two-dimensional square lattice, it is observed in
\cite{KS2} that the dead-ends are attached only at the edge part
of the saddle plateau thanks to its special local geometry, and this
feature allows us to avoid serious difficulties arising from the analysis
of dead-ends. However, for other forms of general lattices, this miracle
does not occur and the dead-ends are attached along the whole saddle
plateau, so that they form a highly complicated maze structure.

We believe that our methodology for proving the Eyring--Kramers formula
is robust against this dead-end structure, and to highlight this robustness
we focus our proof on the hexagonal lattice, in which the dead-ends
structure is indeed complex and emerges in the entire part of the
saddle plateau. The dead-end analysis of the hexagonal lattice relies
on the characterization of all the relevant configurations with energy
around $2L+2$ (cf. Section \ref{Sec6_EB}) and a significant effort
of the current article is devoted to completely understand this dead-end
structure of the hexagonal lattice.

\subsubsection*{Remarks on the Ising/Potts model with non-zero external field}

We conclude the introduction with some remarks on the Ising/Potts
model with \textit{non-zero} external field at very low temperatures.
The metastability of this model has been thoroughly investigated during
the last few decades, and it is interesting that the results are completely
different from the zero external field models considered in \cite{BGN1,KS2,NZ}
and the current article.

For the non-zero external field Ising model in small volume, the saddle
structure of the metastable transition is characterized by appearance
of the specific form of a critical droplet, and hence has a very sharp
saddle structure, in contrast to the fact that the zero external field
model has a huge saddle plateau. Such a characterization has been
carried out in \cite{N-S Ising1,N-S Ising2} for the two-dimensional
square lattice, in \cite{BA-C} for the three-dimensional square lattice,
and in \cite{AJNT} for the two-dimensional hexagonal lattice. In
the last work, it was also been observed that the local geometry of
the hexagonal lattice induces additional difficulty in the analysis
of dead-ends in the vicinity of the critical droplet. These results
also imply the Eyring\textit{--}Kramers formula via potential-theoretic
arguments developed in \cite{B-M}. The same result has been obtained
recently in \cite{BGN2,BGN3} for the Potts model when the external
field acts only on a single spin. An interesting open question is
to analyze the energy landscape and verify the Eyring\textit{--}Kramers
formula when the external field acts on \textit{all} $q$ spins. The
main difficulty of the Potts model compared to the Ising model is
the lack of monotonicity, which was crucially used in \cite{N-S Ising1,N-S Ising2}
to analyze the typical transition path via a grand coupling.

For the model in large volume, it is verified in \cite{BdHS} that
the formation of a critical droplet is still crucial in the transition
in the Ising case. In that article, three regimes of metastable transitions
were studied: formation of a critical droplet, formation of a supercritical
droplet, and evolution of the droplet to a larger size. Typical trajectories
for the metastable transition and the saddle structure have not been
fully characterized yet, and it even remains unknown whether the saddle
configuration contains either one or multiple critical droplets. Hence,
the Eyring--Kramers formula remains an open question for this model. 

\section{\label{sec2_Model}Model}

Before stating the main result of the current article, we rigorously
introduce the model in the current section. 

\subsection{\label{sec21}Spin systems}

\subsubsection*{Lattices}

In this article, we consider spin systems on large, finite two-dimensional
lattices. 

Fix a large positive integer $L\in\mathbb{N}$ and denote by $\Lambda_{L}^{\textup{sq}}$
and $\Lambda_{L}^{\textup{hex}}$ the square and hexagonal lattices
(cf. Figure \ref{fig21}) of size $L$ with periodic boundary conditions.
There is no ambiguity in the definition of $\Lambda_{L}^{\textup{sq}}$,
but further explanation of $\Lambda_{L}^{\textup{hex}}$ is required.
To define $\Lambda_{L}^{\textup{hex}}$, we first select $2L^{2}$
vertices from infinite hexagonal lattice as in Figure \ref{fig21}-(left).
Then, we identify the points at the boundary naturally as illustrated
in the figure. This setting will become intuitively clear when we
introduce the dual lattice in the sequel.

\begin{figure}
\includegraphics[width=14cm]{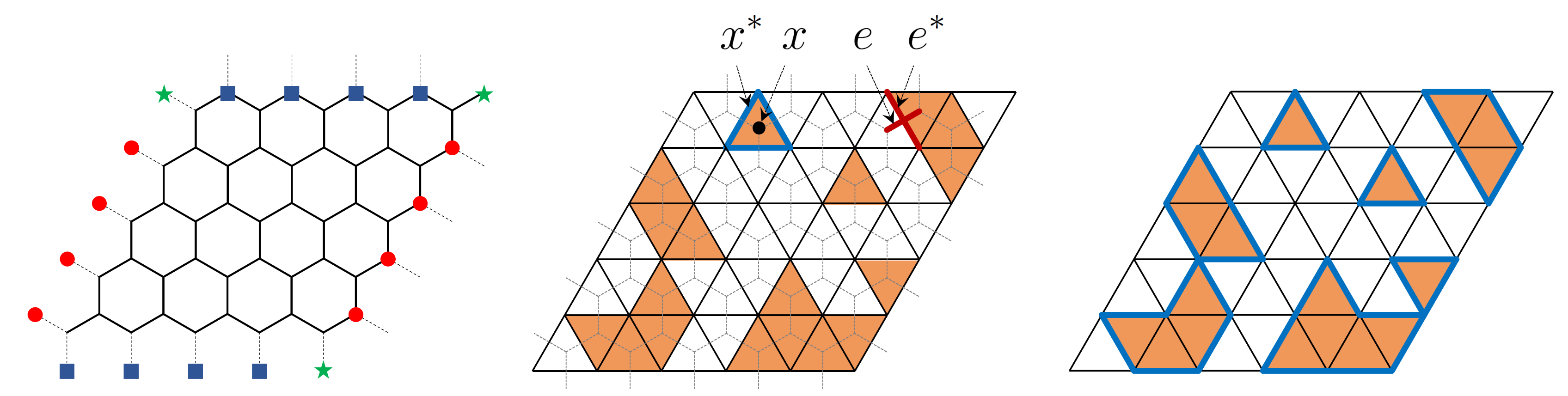}

\caption{\label{fig21}\textbf{(Left)} A hexagonal lattice $\Lambda=\Lambda_{5}^{\textup{hex}}$
with $2\times5^{2}=50$ vertices which are the end points of bold
edges. Under the periodic boundary condition, each vertex at the boundary
highlighted by red circle (resp. blue square) is identified with another
one with red circle (resp. blue square) at the same horizontal level
(resp. same diagonal line with slope $\pi/3$).\textbf{ (Middle) }The
dual lattice $\Lambda^{*}$ of the hexagonal lattice which is a triangular
lattice. The vertex $x$ of $\Lambda$ is identified with the triangular
face $x^{*}$ (the one highlighted by blue bold boundary) in $\Lambda^{*}$.
The edge $e$ of $\Lambda$ is identified with its dual edge $e^{*}$
of $\Lambda^{*}$. In this and the right figure, if the spin at a
certain vertex is $1$ (resp. $2$), we paint the corresponding triangular
face by white (resp. orange). Namely, in this figure we consider the
Ising case $q=2$ and we have $\sigma(x)=2$. \textbf{(Right)} Edges
in $\mathfrak{A}^{*}(\sigma)$ are denoted by blue bold edges, and
hence $\mathfrak{A}^{*}(\sigma)$ is a collection of edges at the
boundaries of the monochromatic clusters. Since the Hamiltonian of
$\sigma$ can be computed as $|\mathfrak{A}^{*}(\sigma)|$, as we
observed in \eqref{e_dualHam1}, $H(\sigma)$ is just the sum of perimeters
of the orange (or white, equivalently) clusters. }
\end{figure}

\subsubsection*{Spin configuration}

We henceforth let $\Lambda=\Lambda_{L}^{\textup{sq}}$ or $\Lambda_{L}^{\textup{hex}}$.
For an integer $q\ge2$, define the set of spins as 
\begin{equation}
\Omega=\Omega_{q}=\{1,\,2,\,\dots,\,q\}\label{e_Omega}
\end{equation}
and we assign a spin from $\Omega$ at each site (vertex) $x$ of
$\Lambda$. The resulting object belonging to the space $\Omega^{\Lambda}$
is called a \textit{(spin) configuration}. We write 
\begin{equation}
\mathcal{X}=\mathcal{X}_{L}=\Omega^{\Lambda}\label{e_Xdef}
\end{equation}
the space of spin configurations\footnote{As in \eqref{e_Xdef}, we omit subscripts or superscripts $L$ highlighting
the dependency of the corresponding object to $L$, as soon as there
is no risk of confusion by doing so.}. We use the notation $\sigma=(\sigma(x))_{x\in\Lambda}$ to denote
a spin configuration, i.e., an element of $\mathcal{X}$, where $\sigma(x)$
represents the spin at site $x\in\Lambda$. 

\subsubsection*{Visualization via dual lattice}

To visualize spin configurations, it is convenient to consider the
dual lattice $\Lambda^{*}$ of $\Lambda$. If $\Lambda=\Lambda_{L}^{\textup{sq}}$,
the dual lattice $\Lambda^{*}$ is again a periodic square lattice
of side length $L$. On the other hand, if $\Lambda=\Lambda_{L}^{\textup{hex}}$,
the dual lattice $\Lambda^{*}$ is a rhombus-shaped periodic triangular
lattice with side length $L$ as in Figure \ref{fig21}-(middle).
Note that the periodic boundary condition of the triangular lattice
inherited from that of the hexagonal lattice simply identifies four
boundaries of the rhombus in a routine way (as in $\mathbb{Z}_{L}^{2}$). 

Since we can identify a site of $\Lambda$ with a face of $\Lambda^{*}$
containing it, we can regard the spins assigned at the sites of $\Lambda$
as those assigned to the faces of $\Lambda^{*}$. Thus, by assigning
different colors to each set of spins, we can readily visualize the
spin configurations on the dual lattice. For instance, in Figure \ref{fig21}-(middle,
right), the triangles with white and orange colors correspond to the
vertices of spins $1$ and $2$, respectively. This visualization
is conceptually more convenient in analyzing the energy of spin configurations,
as we explain in the next subsection. 

\subsection{\label{sec22}Ising and Potts models}

The Ising and Potts models are defined through a suitable probability
distribution on the space $\mathcal{X}$ of spin configurations (cf.
\eqref{e_Xdef}). To define this, let us first define the Ising/Potts
Hamiltonian $H:\mathcal{X}\rightarrow\mathbb{R}$ by 
\begin{equation}
H(\sigma)=\sum_{x\sim y}\mathbf{1}\{\sigma(x)\ne\sigma(y)\}\;\;\;\;;\;\sigma\in\mathcal{X}\;,\label{e_Ham}
\end{equation}
where $x\sim y$ if and only if $x$ and $y$ are connected by an
edge of the lattice $\Lambda$. We emphasize here that this definition
of $H$ indicates that there is no external field acting on the Hamiltonian.
Denote by $\mu_{\beta}(\cdot)$ the Gibbs measure on $\mathcal{X}$
associated with the Hamiltonian $H(\cdot)$ at the inverse temperature
$\beta>0$, i.e., 
\begin{equation}
\mu_{\beta}(\sigma)=\frac{1}{Z_{\beta}}e^{-\beta H(\sigma)}\text{ for }\sigma\in\mathcal{X}\;,\;\;\;\text{where}\;\;\;\;Z_{\beta}=\sum_{\zeta\in\mathcal{X}}e^{-\beta H(\zeta)}\;.\label{e_muZ}
\end{equation}
The random spin configuration associated with the probability measure
$\mu_{\beta}(\cdot)$ is called the \textit{Ising model} if $q=2$
and the \textit{Potts model} if $q\ge3$.

\subsubsection*{Computing the Hamiltonian via dual lattice representation}

We can identify an edge $e$ in $\Lambda$ with the unique edge $e^{*}$
in the dual lattice $\Lambda^{*}$ intersecting with $e$ (cf. Figure
\ref{fig21}-(middle)), and we can identify each vertex $x$ in $\Lambda$
with the unique face $x^{*}$ in the dual lattice $\Lambda^{*}$ containing
$x$. As we have mentioned in Section \ref{sec21}, we regard each
spin in $\Omega=\{1,\,2,\,\dots,\,q\}$ as a color, so that each face
$x^{*}$ in the dual lattice is painted by the color corresponding
to the spin at $x$. Thus, we can identify $\sigma\in\mathcal{X}$
with a $q$-coloring on the faces of the dual lattice $\Lambda^{*}$.
In this coloring representation of $\sigma$, each maximal monochromatic
connected\footnote{Of course, two faces sharing only a vertex are not connected.}
component is called a (monochromatic)\textit{ cluster} of $\sigma$.

We now explain a convenient formulation to understand the Hamiltonian
of a spin configuration $\sigma\in\mathcal{X}$ with the setting explained
above. We refer to Figure \ref{fig21}-(right) for an illustration.
Define $\mathfrak{A}(\sigma)$ as the collection of edges $e=\{x,\,y\}$
in $\Lambda$ such that $\sigma(x)\neq\sigma(y)$. Then, define
\begin{equation}
\mathfrak{A}^{*}(\sigma)=\{e^{*}:e\in\mathfrak{A}(\sigma)\}\;,\label{e_Astar}
\end{equation}
so that by the definition of the Hamiltonian, we have
\begin{equation}
H(\sigma)=|\mathfrak{A}(\sigma)|=|\mathfrak{A}^{*}(\sigma)|\;.\label{e_dualHam1}
\end{equation}
The crucial observation is that the dual edge $e^{*}$ belongs to
$\mathfrak{A}^{*}(\sigma)$ if and only if $e^{*}$ belongs to the
boundary of a cluster of $\sigma$. Hence, as in Figure Figure \ref{fig21}-(right),
we can readily compute the energy $H(\sigma)$ as 
\begin{equation}
H(\sigma)=\frac{1}{2}\sum_{A^{*}:\text{ cluster of }\sigma}(\textup{perimeter of }A^{*})\;,\label{e_dualHam}
\end{equation}
where the factor $\frac{1}{2}$ appears since each dual edge $e^{*}\in\mathfrak{A}^{*}(\sigma)$
belongs to the perimeter of exactly two clusters. 

\subsection{\label{sec23}Heat-bath Glauber dynamics}

We next introduce a heat-bath Glauber dynamics associated with the
Gibbs measure $\mu_{\beta}(\cdot)$. We consider herein the continuous-time
Metropolis--Hastings dynamics $\{\sigma_{\beta}(t)\}_{t\ge0}$ on
$\mathcal{X}$, whose jump rate from $\sigma\in\mathcal{X}$ to $\zeta\in\mathcal{X}$
is given by
\begin{equation}
c_{\beta}(\sigma,\,\zeta)=\begin{cases}
e^{-\beta\max\{H(\zeta)-H(\sigma),\,0\}} & \text{if }\zeta=\sigma^{x,\,a}\ne\sigma\text{ for some }x\in\Lambda\text{ and }a\in\Omega\;,\\
0 & \text{otherwise},
\end{cases}\label{e_cbeta}
\end{equation}
where $\sigma^{x,\,a}\in\mathcal{X}$ denotes the configuration obtained
from $\sigma$ by flipping the spin at site $x$ to $a$. This dynamics
is standard in the study of metastability of the Ising/Potts model
on lattices; see, e.g., \cite{BdH,N-S Ising1,N-S Ising2} and references
therein. For $\sigma,\,\zeta\in\mathcal{X}$, we write
\begin{equation}
\sigma\sim\zeta\;\;\;\;\text{if and only if}\;\;\;\;c_{\beta}(\sigma,\,\zeta)>0\;.\label{e_conn}
\end{equation}
Note that the jump of dynamics $\sigma_{\beta}(\cdot)$ is available
only through a single-spin flip. We will write $\mathbb{P}_{\sigma}^{\beta}$
the law of the Markov process $\sigma_{\beta}(\cdot)$ starting at
$\sigma\in\mathcal{X}$, and write $\mathbb{E}_{\sigma}^{\beta}$
the corresponding expectation.

From the definitions of $\mu_{\beta}(\cdot)$ and $c_{\beta}(\cdot,\,\cdot)$,
we can directly check the following so-called detailed balance condition:
\begin{equation}
\mu_{\beta}(\sigma)c_{\beta}(\sigma,\,\zeta)=\mu_{\beta}(\zeta)c_{\beta}(\zeta,\,\sigma)=\begin{cases}
\min\{\mu_{\beta}(\sigma),\,\mu_{\beta}(\zeta)\} & \text{if }\sigma\sim\zeta\;,\\
0 & \text{otherwise.}
\end{cases}\label{e_detbal}
\end{equation}
Hence, the Markov process $\sigma_{\beta}(\cdot)$ is reversible with
respect to its invariant measure $\mu_{\beta}(\cdot)$. 

\section{\label{sec3_MR}Main Result}

In this section, we explain the main results obtained in this article
for the Ising/Potts model explained in the previous section. We assume
hereafter that $L\ge8$ to avoid unnecessary technical difficulties. 

\subsection{Hamiltonian and energy barrier}

We first explain some results regarding the Hamiltonian $H(\cdot)$
of the Ising/Potts model. 

\subsubsection*{Ground states}

For each $a\in\Omega$, we denote by $\mathbf{a}\in\mathcal{X}$ the
configuration of which all the spins are $a$, i.e., $\mathbf{a}(x)=a$
for all $x\in\Lambda$. Write 
\[
\mathcal{S}=\{\mathbf{1},\,\mathbf{2},\,\dots,\,\mathbf{q}\}\;\;\;\text{and}\;\;\;\mathcal{S}(A)=\{\mathbf{a}\in\mathcal{S}:a\in A\}
\]
for each $A\subseteq\Omega$. Then, it is immediate from the definition
that the Hamiltonian achieves its minimum $0$ exactly at the configurations
belonging to $\mathcal{S}$, and therefore the set $\mathcal{S}$
denotes the collection of all \textit{ground states} of the Ising/Potts
model without external field.

\subsubsection*{Energy barrier}

We first concern on the measurement of the energy barrier between
the ground states. The energy barrier is a fundamental quantity in
the investigation of the saddle structure between ground states. To
define this, let a sequence of configurations $(\omega_{n})_{n=0}^{N}$
in $\mathcal{X}$ be a \textit{path of length $N$} if\footnote{For integers $a$ and $b$, we write $\llbracket a,\,b\rrbracket=[a,\,b]\cap\mathbb{Z}$.}
(cf. \eqref{e_conn}), 
\begin{equation}
\omega_{n}\sim\omega_{n+1}\;\;\;\;\text{for all }n\in\llbracket0,\,N-1\rrbracket\;.\label{e_path}
\end{equation}
The path $(\omega_{n})_{n=0}^{N}$ is said to connect $\sigma$ and
$\zeta$ if $\omega_{0}=\sigma$ and $\omega_{N}=\zeta$ or vice versa.
The \textit{communication height} between two configurations $\sigma,\,\zeta\in\mathcal{X}$
is defined as 
\[
\Phi(\sigma,\,\zeta)=\min_{(\omega_{n})_{n=0}^{N}\text{ connects }\sigma\text{ and }\zeta}\,\max_{n\in\llbracket0,\,N\rrbracket}H(\omega_{n})\;.
\]
Then, the \textit{energy barrier} between ground states is defined
as, for $\mathbf{a},\,\mathbf{b}\in\mathcal{S}$, 
\begin{equation}
\Gamma=\Gamma_{\Lambda}=\Phi(\mathbf{a},\,\mathbf{b})\;,\label{e_Gamma}
\end{equation}
where the value of $\Gamma$ is independent of the selection of $\mathbf{a},\,\mathbf{b}$
by the symmetry of the model.
\begin{thm}
\label{t_Ebarrier}For both $\Lambda=\Lambda_{L}^{\textup{sq}}$ and
$\Lambda_{L}^{\textup{hex}}$, we have $\Gamma=2L+2$. Moreover, there
is no valley of depth larger than $\Gamma$ in the sense that
\[
\min_{\mathbf{s}\in\mathcal{S}}\Phi(\sigma,\,\mathbf{s})-H(\sigma)<\Gamma\;\;\;\;\text{for all }\sigma\in\mathcal{X}\setminus\mathcal{S}\;.
\]
\end{thm}

This theorem has been proven for square lattice in \cite{NZ}, and
we prove this theorem for the hexagonal lattice in Section \ref{Sec6_EB}.

\subsection{\label{sec32}Concentration of Gibbs measure}

We next investigate the Gibbs measure $\mu_{\beta}(\cdot)$. We can
readily observe from definition that if $L$ is fixed and $\beta\rightarrow\infty$,
the Gibbs measure $\mu_{\beta}(\cdot)$ is concentrated on the ground
set $\mathcal{S}$. However, if we consider the large volume regime
for which both $L$ and $\beta$ tends to $\infty$ together, the
non-ground states can have non-negligible masses because of the entropy
effect, that is, there are sufficiently many configurations with high
energy that can dominate the mass of the ground states. By a careful
combinatorial analysis carried out in Section \ref{sec4_Gibbs}, we
can accurately quantify this competition between energy and entropy;
consequently, we establish the zero-one law type result by finding
a sharp threshold determining whether the Gibbs measure $\mu_{\beta}$
is concentrated on $\mathcal{S}$. Before explaining this result,
we explicitly declare the regime that we consider.
\begin{assumption*}
The inverse temperature $\beta=\beta_{L}$ depends on $L$ and we
consider the large-volume, low-temperature regime, in the sense that
$\beta_{L}\rightarrow\infty$ as $L\rightarrow\infty$. 
\end{assumption*}
For sequences $(a_{L})_{L=1}^{\infty}$ and $(b_{L})_{L=1}^{\infty}$,
we write $a_{L}\ll b_{L}$ if $\lim_{L\rightarrow\infty}a_{L}/b_{L}=0$
and write $a_{L}=o_{L}(1)$ if $\lim_{L\rightarrow\infty}a_{L}=0$.
The following theorem will be proven in Section \ref{sec4_Gibbs}.
\begin{thm}
\label{t_Gibbs1}Let us define a constant $\gamma_{0}$ by 
\begin{equation}
\gamma_{0}=\begin{cases}
1/2 & \text{for the square lattice},\\
2/3 & \text{for the hexagonal lattice}.
\end{cases}\label{e_keyconst}
\end{equation}
Then, the following estimates hold.
\begin{enumerate}
\item Suppose that $L^{\gamma_{0}}\ll e^{\beta}$. Then, we have (cf. \eqref{e_muZ})
$Z_{\beta}=q+o_{L}(1)$ and 
\[
\mu_{\beta}(\mathcal{S})=1-o_{L}(1)\;.
\]
\item On the other hand, suppose that $e^{\beta}\ll L^{\gamma_{0}}$. Then,
we have 
\[
\mu_{\beta}(\mathcal{S})=o_{L}(1)\;.
\]
\end{enumerate}
\end{thm}

Henceforth, the constant $\gamma_{0}$ always refers to the one defined
in \eqref{e_keyconst}. This theorem implies that a drastic change
in the valley structure of the Gibbs measure $\mu_{\beta}$ occurs
at $\beta/\log L=\gamma_{0}$. Namely, if $\beta/\log L\ge\gamma$
for some $\gamma>\gamma_{0}$, the ground states themselves form metastable
sets, while if $\beta/\log L\le\gamma$ for $\gamma<\gamma_{0}$,
ground states are no longer metastable. 

The second regime can be investigated further. Define, for each $i\ge0$,
\begin{equation}
\mathcal{X}_{i}=\mathcal{X}_{i,\,L}=\{\sigma\in\mathcal{X}:H(\sigma)=i\}\;,\label{e_Xidef}
\end{equation}
so that $\mathcal{X}_{0}=\mathcal{S}$ denotes the set of ground states.
For any interval $I\subseteq\mathbb{R}$, we write
\[
\mathcal{X}_{I}=\bigcup_{i\in I\cap\mathbb{Z}}\mathcal{X}_{i}\;.
\]
Then, we have the following refinement of case (2) of Theorem \ref{t_Gibbs1}
which will be proven in Section \ref{sec4_Gibbs} as well. 
\begin{thm}
\label{t_Gibbs2}Suppose that $e^{\beta}\ll L^{\gamma_{0}}$ and fix
a constant $\alpha\in(0,\,1)$. Then, the following statements hold.
\begin{enumerate}
\item Suppose that $L^{\gamma_{0}(1-\alpha)}\ll e^{\beta}$. Then, for every
$c>0$, we have
\[
\mu_{\beta}\big(\,\mathcal{X}_{[0,\,cL^{2\alpha}]}\,\big)=1-o_{L}(1)\;.
\]
\item Suppose that $L^{\gamma_{0}(1-\alpha)}\gg e^{\beta}$. Then, for every
$c>0$, we have
\[
\mu_{\beta}\big(\,\mathcal{X}_{[0,\,cL^{2\alpha}]}\,\big)=o_{L}(1)\;.
\]
\end{enumerate}
\end{thm}

We can infer from Theorem \ref{t_Ebarrier} that the energetic valley
$\mathcal{V}_{\mathbf{a}}$ containing each $\mathbf{a}\in\mathcal{S}$
should be the connected component of 
\[
\{\sigma\in\mathcal{X}:H(\sigma)<2L+2\}
\]
containing $\mathbf{a}$, where the connectedness of a set $\mathcal{A}\subseteq\mathcal{X}$
here refers to the path-connectedness (cf. \eqref{e_path}). Since
Theorem \ref{t_Gibbs2} implies that
\[
\mu_{\beta}\big(\,\mathcal{X}_{[0,\,2L+1]}\,\big)=\begin{cases}
1-o_{L}(1) & \text{if }L^{\gamma_{0}/2}\ll e^{\beta}\;,\\
o_{L}(1) & \text{if }L^{\gamma_{0}/2}\gg e^{\beta}\;,
\end{cases}
\]
we can conclude that the energetic valleys $\mathcal{V}_{\mathbf{a}}$
are indeed metastable valleys if $\beta/\log L\ge\gamma$ for some
$\gamma>\frac{\gamma_{0}}{2}$.\textbf{ }In contrast, if $\beta/\log L\le\gamma$
for some $\gamma<\frac{\gamma_{0}}{2}$, the Gibbs measure is concentrated
on the \emph{complement} of these energetic valleys. Hence, in the
latter regime (as long as $\beta$ is bigger than the critical temperature
$\beta_{c}(q)=\log(1+\sqrt{q})$ of the Ising/Potts model \cite{B-DC}),
we deduce that the metastable set must lie upon the configurations
with higher energy; this is the onset in which the entropy starts
to play a significant role. 

\subsection{\label{sec33}Eyring--Kramers formula}

We next concern on the dynamical metastable behavior exhibited by
the Metropolis dynamics $\sigma_{\beta}(\cdot)$ defined in Section
\ref{sec23}. If the invariant measure $\mu_{\beta}(\cdot)$ is concentrated
on the set $\mathcal{S}$, we can expect that the process $\sigma_{\beta}(\cdot)$
starting at some $\mathbf{a}\in\mathcal{S}$ spends a sufficiently
long time around $\mathbf{a}$ before making a transition to another
ground state. This type of behavior is the signature of metastability
of the process $\sigma_{\beta}(\cdot)$, and we are interested in
its quantification. To explain this in more detail, we first define
the hitting time of the set $\mathcal{A}\subseteq\mathcal{X}$ as
\[
\tau_{\mathcal{A}}=\inf\{t\ge0:\sigma_{\beta}(t)\in\mathcal{A}\}\;,
\]
and simply write $\tau_{\{\sigma\}}=\tau_{\sigma}$. Then, we are
primarily interested in the \textit{mean transition time} of the form
$\mathbb{E}_{\mathbf{a}}^{\beta}[\tau_{\mathcal{S}\setminus\{\mathbf{a}\}}]$
or $\mathbb{E}_{\mathbf{a}}^{\beta}[\tau_{\mathbf{b}}]$ denoting
the expectation of the metastable transition from a ground state to
another one, where $\mathbb{E}_{\mathbf{a}}^{\beta}$ is defined right
after \eqref{e_conn}. These quantities are significant in the study
of the metastable behavior because they are key notions explaining
the amount of time required to observe a metastable transition and
are closely related to the mixing time or spectral gap of the dynamics.
The precise estimation of the mean transition time is called the \textit{Eyring--Kramers
formula}. The next main result of the current article is the following
Eyring--Kramers formula for the Metropolis dynamics. Define a constant
$\kappa_{0}$ by 
\begin{equation}
\kappa_{0}=\begin{cases}
1/8 & \text{for the square lattice},\\
1/12 & \text{for the hexagonal lattice}.
\end{cases}\label{e_kappa0}
\end{equation}

\begin{thm}[Eyring--Kramers formula]
\label{t_EK} Suppose that $\beta=\beta_{L}$ satisfies $L^{3}\ll e^{\beta}$
for the square lattice and $L^{10}\ll e^{\beta}$ for the hexagonal
lattice. Then, for all $\mathbf{a},\,\mathbf{b}\in\mathcal{S}$, we
have
\begin{equation}
\mathbb{E}_{\mathbf{a}}^{\beta}\big[\,\tau_{\mathcal{S}\setminus\{\mathbf{a}\}}\,\big]=\frac{\kappa_{0}+o_{L}(1)}{q-1}e^{\Gamma\beta}\;\;\;\;\text{and}\;\;\;\;\mathbb{E}_{\mathbf{a}}^{\beta}\big[\,\tau_{\mathbf{b}}\,\big]=(\kappa_{0}+o_{L}(1))e^{\Gamma\beta}\;,\label{e_EK}
\end{equation}
where $\Gamma=2L+2$ is the energy barrier obtained in Theorem \ref{t_Ebarrier}. 
\end{thm}

The proof of this theorem is given in Sections \ref{sec8_PT} through
\ref{sec10_LB} based on the comprehensive analysis of the saddle
structure carried out in Sections \ref{Sec6_EB} and \ref{sec7_Saddle}.
\begin{rem}
We conjecture that this result holds for all $\beta=\beta_{L}$ satisfying
$L^{\gamma_{0}/2}\ll e^{\beta}$, under which the invariant measure
is concentrated on the energetic valleys around ground states. The
sub-optimality of the lower bound (of constant order) on $\beta$
owes to several technical issues arising in the proof (cf. Sections
\ref{sec9_UB} and \ref{sec10_LB}), and we guess that additional
innovative ideas are required to get the optimal bound.
\end{rem}

\begin{rem}
The condition on $\beta$ is relatively tight ($L^{3}\ll e^{\beta}$)
for the square lattice, whereas the condition for the hexagonal lattice
is slightly loose ($L^{10}\ll e^{\beta}$). This is because the dead-end
analysis is much more complicated for the hexagonal lattice owing
to its complicated local geometry. It will be highlighted in Sections
\ref{sec62} and \ref{sec63}.
\end{rem}

\begin{rem}
One can also obtain the Markov chain convergence of the so-called
trace process (cf. \cite{BL1}) of the accelerated process $\{\sigma_{\beta}(e^{\Gamma\beta}t)\}_{t\ge0}$
on the set $\mathcal{S}$ to the Markov process on $\mathcal{S}$
with uniform rate $r(\mathbf{a},\,\mathbf{b})=\frac{1}{\kappa_{0}}$
for all $\mathbf{a},\,\mathbf{b}\in\mathcal{S}$. Such a Markov chain
model reduction of the metastable behavior is an alternative method
of investigating the metastability (cf. \cite{BL1,BL2,Lan}). The
proof of this result using Theorem \ref{t_cap} is identical to that
of \cite[Theorem 2.11]{KS2} and is not repeated here.
\end{rem}

\subsection{Outlook of remainder of article}

\textit{In the remainder of the article, we explain the proof of the
theorems explained above in detail only for the hexagonal lattice},
since the proof for the square lattice is similar to that for the
hexagonal lattice and in fact much simpler; the geometry of the hexagonal
lattice is far more complex and needs careful consideration with additional
complicated arguments. Moreover, the analysis of square lattice can
be helped a lot by the computations carried out in \cite{KS2} which
considered the small-volume regime ($L$ is fixed and $\beta$ tends
to infinity). 

The remainder of the article is organized as follows. In Section \ref{sec4_Gibbs},
we analyze the Gibbs measure $\mu_{\beta}(\cdot)$ to prove Theorems
\ref{t_Gibbs1} and \ref{t_Gibbs2}. In Section \ref{Sec5_Pre}, we
provide some preliminary observations to investigate the energy landscape
in a more detailed manner. We then analyze the energy landscape of
the Hamiltonian in detail in Sections \ref{Sec6_EB} and \ref{sec7_Saddle}.
As a by-product of our deep analysis, Theorem \ref{t_Ebarrier} will
be proven at the end of Section \ref{Sec6_EB}. Then, we finally prove
the Eyring-Kramers formula, i.e., Theorem \ref{t_EK} in remaining
sections.

\section{\label{sec4_Gibbs}Sharp Threshold for Gibbs Measure}

In this section, we prove Theorems \ref{t_Gibbs1} and \ref{t_Gibbs2}.
We remark that we will now implicitly assume that the underlying lattice
is the hexagonal one, unless otherwise specified. We shall briefly
discuss the square lattice in Section \ref{sec45}.

\subsection{Lemma on graph decomposition}

We begin with a lemma on graph decomposition which is crucially used
in estimating the number of configurations having a specific energy. 
\begin{notation}
\label{n_graph}For a graph $G=(V,\,E)$ and a set $E_{0}\subseteq E$
of edges, we denote by $G[E_{0}]=(V[E_{0}],\,E_{0})$ the subgraph
induced by the edge set $E_{0}$ where the vertex set $V[E_{0}]$
is the collection of end points of the edges in $E_{0}$. The edge
set $E_{0}\subseteq E$ is said to be connected if the induced graph
$G[E_{0}]$ is a connected graph.
\end{notation}

\begin{lem}
\label{l_graph}$G=(V,\,E)$ be a graph such that every connected
component has at least three edges. Then, we can decompose 
\[
E=E_{1}\cup\cdots\cup E_{n}
\]
such that $E_{i}$ is connected and $|E_{i}|\in\{3,\,4,\,5,\,6\}$
for all $i=1,\,\dots,\,n$.
\end{lem}

\begin{proof}
It suffices to prove the lemma for a connected graph $G$ with at
least three edges, since we can apply this result to each connected
component to complete the proof for general case. Hence, we from now
on assume that $G$ is a connected graph with at least three edges.
Then, the proof is proceeded by induction on the cardinality $|E|$.

First, there is nothing to prove if $|E|\le6$ since we can take $n=1$
and $E_{1}=E$. Next, let us fix $k\ge7$ and assume that the lemma
holds if $3\le|E|\le k-1$. Let $G=(V,\,E)$ be a connected graph
with $|E|=k$. We will find $E'\subseteq E$ such that 
\begin{equation}
|E'|,\,|E\setminus E'|\ge3\text{ and both }E'\text{ and }E\setminus E'\text{ are connected}.\label{e_condec}
\end{equation}
Once finding such an $E'$, it suffices to apply the induction hypothesis
to the sets $E'$ and $E\setminus E'$ to complete the proof.

\begin{figure}[h]
\includegraphics[width=12.5cm]{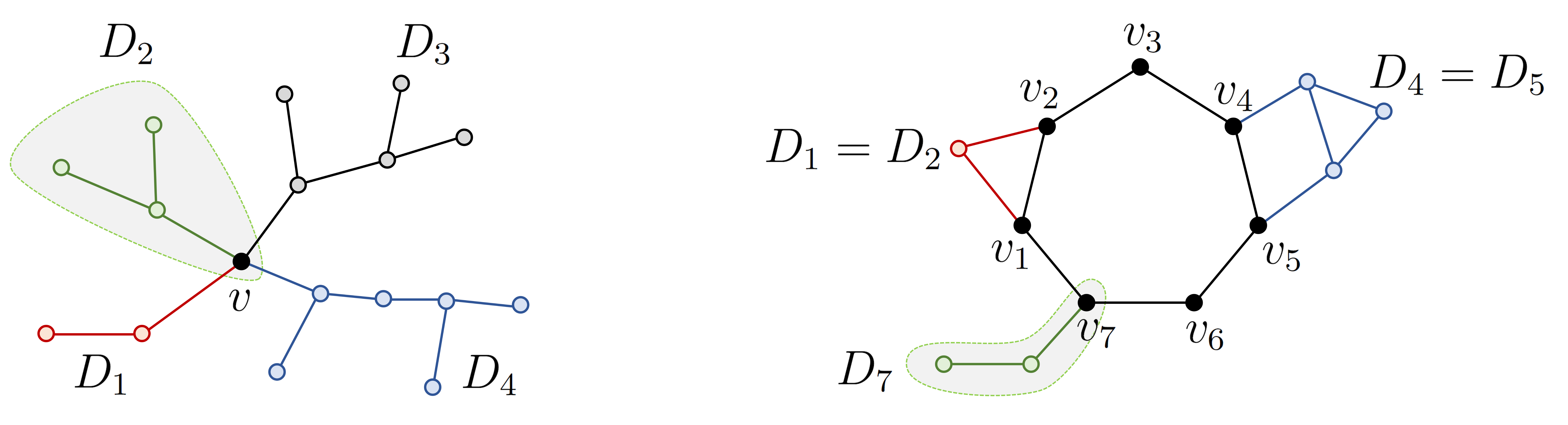}\caption{\label{fig41}The left and right figures illustrate \textbf{(Case
1)} and \textbf{(Case 2)} in the proof of Lemma \ref{l_graph}, respectively.
Note that we have $D_{3}=D_{6}=\emptyset$ at the right figure.}
\end{figure}

\noindent \textbf{(Case 1: $G$ does not have a cycle, i.e., $G$
is a tree) }If every vertex of \textbf{$G$} has degree at most $2$,
then \textbf{$G$} is a line graph, and thus we can easily divide
$E$ into two connected subsets $E'$ and $E\setminus E'$ satisfying
\eqref{e_condec}.

Next, we suppose that a vertex $v\in V$ has degree at least $3$.
Since $G$ is a tree, we can decompose $E$ into connected $D_{1},\,D_{2},\,\dots,\,D_{m}$
with $m=\deg(v)\ge3$, such that the edges in $D_{i}$ and $D_{j}$
($i\neq j)$ possibly intersect only at $v$ (cf. Figure \ref{fig41}-(left)).
We impose the condition $|D_{1}|\le\cdots\le|D_{m}|$ for convenience.

If $|D_{k}|\ge3$ for $k=1$ or $2$, it suffices to take $E'=D_{k}$.
If $|D_{1}|,\,|D_{2}|\le2$ but $|D_{1}|+|D_{2}|\ge3$, we take $E'=D_{1}\cup D_{2}$.
Finally, if $|D_{1}|=|D_{2}|=1$, we take
\[
E'=\begin{cases}
D_{1}\cup D_{2}\cup\{\text{the edge in }D_{3}\text{ having }v\text{ as an end point}\} & \text{if }m=3\;,\\
D_{1}\cup D_{2}\cup D_{3} & \text{if }m\ge4\;.
\end{cases}
\]

\noindent \textbf{(Case 2: $G$ has a cycle)} Suppose that $(v_{1},\,v_{2},\,\dots,\,v_{n})$
is a cycle in $G$ in the sense that $\{v_{i},\,v_{i+1}\}\in E$ for
all $i\in\llbracket1,\,n\rrbracket$ with the convention $v_{n+1}=v_{1}$.
We denote by $E_{0}$ the edges belonging to this cycle, i.e., 
\[
E_{0}=\big\{\,\{v_{i},\,v_{i+1}\}:i\in\llbracket1,\,n\rrbracket\,\big\}\;.
\]
If $E=E_{0}$, i.e., $G$ is a ring graph, we can easily divide $E$
into two connected subsets $E'$ and $E\setminus E'$ satisfying \eqref{e_condec}
and hence suppose that $E\setminus E_{0}\ne\emptyset$. For each $i\in\llbracket1,\,n\rrbracket$,
we denote by $D_{i}$ the connected component of $E\setminus E_{0}$
containing the vertex $v_{i}$ so that we have as in Figure \ref{fig41}-(right)
so that 
\[
E=E_{0}\cup\Big(\,\bigcup_{i=1}^{n}D_{i}\,\Big)\;.
\]
Note that we may have $D_{i}=D_{j}$ for some $i\neq j$. Since we
assumed $E\setminus E_{0}\ne\emptyset$, we can assume without loss
of generality that $D_{1}\ne\emptyset$. If $|D_{1}|\ge3$, we take
$E'=D_{1}$. Otherwise, we take $E'=D_{1}\cup\{\{v_{1},\,v_{2}\},\,\{v_{1},\,v_{n}\}\}$.

This completes the proof of \eqref{e_condec} and we are done. 
\end{proof}
\begin{rem}
We remark that the set $\{3,\,4,\,5,\,6\}$ appeared in the previous
lemma cannot be replaced with $\{3,\,4,\,5\}$. For example, in \textbf{(Case
1)} of the proof (cf. Figure \ref{fig41}-(left)), the graph with
$m=3$ and $|E_{1}|=|E_{2}|=|E_{3}|=2$ provides such a counterexample.
\end{rem}

\subsection{Counting of configurations with fixed energy}

The crucial lemma in the analysis of the Gibbs measure is the following
upper and lower bounds for the number of configurations belonging
to the set $\mathcal{X}_{i}$, which denotes the collection of configurations
with energy $i$ (cf. \eqref{e_Xidef}).
\begin{lem}
\label{l_Xi}There exists $\theta>1$ such that the following estimates
hold. 
\begin{enumerate}
\item (Upper bound) For all $i\in\mathbb{N}$, we have
\[
|\mathcal{X}_{i}|\le q^{i+1}\times\sum_{\substack{n_{3},\,n_{4},\,n_{5},\,n_{6}\ge0:\\
3n_{3}+4n_{4}+5n_{5}+6n_{6}=i
}
}{\theta L^{2} \choose n_{3}}{\theta L^{2} \choose n_{4}}{\theta L^{2} \choose n_{5}}{\theta L^{2} \choose n_{6}}\;.
\]
\item (Lower bound) For all $1\le j<\lfloor\frac{L^{2}}{2}\rfloor$, we
have
\[
|\mathcal{X}_{3j}|\ge4^{j}{\lfloor\frac{L^{2}}{2}\rfloor \choose j}\;.
\]
\end{enumerate}
\end{lem}

\begin{proof}
(1) As the assertion is obvious for $i=0$ where $|\mathcal{X}_{0}|=q$,
we assume $i\ne0$ so that $i\ge3$ (since $\mathcal{X}_{1}$ and
$\mathcal{X}_{2}$ are empty). Denote by $E(\Lambda^{*})$ the collection
of edges of the dual lattice $\Lambda^{*}$ and let $\mathcal{E}_{i}$
be the collection of $E_{0}\subseteq E(\Lambda^{*})$ such that $|E_{0}|=i$.
Then, by the arguments given in Section \ref{sec22}, we can regard
$\mathfrak{A}^{*}(\cdot)$ defined in \eqref{e_Astar} as a map from
$\mathcal{X}_{i}$ to $\mathcal{E}_{i}$. 

For $\sigma\in\mathcal{X}$, it is immediate that the graph $G[\mathfrak{A}^{*}(\sigma)]$
(cf. Notation \ref{n_graph}) has no vertex of degree $1$, since
if there exists such a vertex, then there is no possible coloring
on the six faces of $\Lambda^{*}$ surrounding the vertex which realizes
$\mathfrak{A}^{*}(\sigma)$. Therefore, each vertex of $G[\mathfrak{A}^{*}(\sigma)]$
has degree at least two. This implies that each connected component
of $G[\mathfrak{A}^{*}(\sigma)]$ has a cycle and hence has at least
three edges. Thus, by Lemma \ref{l_graph}, we can decompose an element
of $\mathfrak{A}^{*}(\mathcal{X}_{i})$ by connected components of
sizes $3$, $4$, $5$, or $6$. Note that there exists a fixed integer
$\theta>1$ such that there are at most $\theta L^{2}$ connected
subgraphs of $\Lambda^{*}$ with at most $6$ edges (for all $L$).
Combining the observations above concludes that
\begin{equation}
|\mathfrak{A}^{*}(\mathcal{X}_{i})|\le\sum_{\substack{n_{3},\,n_{4},\,n_{5},\,n_{6}\ge0:\\
3n_{3}+4n_{4}+5n_{5}+6n_{6}=i
}
}{\theta L^{2} \choose n_{3}}{\theta L^{2} \choose n_{4}}{\theta L^{2} \choose n_{5}}{\theta L^{2} \choose n_{6}}\;.\label{e_bdAx}
\end{equation}
Next, we will show that
\begin{equation}
|(\mathfrak{A}^{*})^{-1}(\eta)|\le q^{i+1}\;\;\;\;\text{for all }\eta\in\mathfrak{A}^{*}(\mathcal{X}_{i})\;.\label{e_bdAx2}
\end{equation}
Indeed, since $\eta\in\mathfrak{A}^{*}(\mathcal{X}_{i})$ has $i$
edges, it divides (the faces of) $\Lambda^{*}$ into at most $i+1$
connected components, where each component must be a monochromatic
cluster in each $\sigma\in(\mathfrak{A}^{*})^{-1}(\eta)$. Therefore,
there are at most $q^{i+1}$ (indeed, $q\times(q-1)^{i}$) ways to
paint these monochromatic clusters and we get \eqref{e_bdAx2}. Part
(1) follows directly from \eqref{e_bdAx} and \eqref{e_bdAx2}.

\noindent (2) If we take an independent set\footnote{Here, a set is called \textit{independent} if it consists of lattice
vertices among which any two vertices are not connected by a lattice
edge.} $A$ of size $j$ from $\Lambda$ (i.e., we take $j$ mutually disconnected
triangle faces in $\Lambda^{*}$), and assign spins $1$ and $2$
on $A$ and $\Lambda\setminus A$, respectively, then the energy of
the corresponding configuration is $3j$ by \eqref{e_dualHam}. If
we select such $j$ vertices one by one, then each selection of a
vertex reduces at most four possibilities of the next choice (namely,
the selected one and the three adjacent vertices). Since the selection
does not depend on the order, there are at least
\[
\frac{2L^{2}(2L^{2}-4)\cdots(2L^{2}-4j+4)}{j!}\ge4^{j}{\lfloor\frac{L^{2}}{2}\rfloor \choose j}
\]
ways of selecting such an independent set of size $j$. This concludes
the proof of part (2).
\end{proof}

\subsection{Lemma on concentration}

In this subsection, we establish a counting lemma which is useful
in the proof of Theorems \ref{t_Gibbs1} and \ref{t_Gibbs2}. Remark
that we regard $\beta=\beta_{L}$ to be dependent of $L$. 
\begin{lem}
\label{l_muprop}Suppose that $e^{\beta}\ll L^{2/3}$ and moreover
two sequences $(g_{1}(L))_{L\in\mathbb{N}}$ and $(g_{2}(L))_{L\in\mathbb{N}}$
satisfy
\[
1\ll g_{1}(L)\ll L^{2}e^{-3\beta}\ll g_{2}(L)\;.
\]
Then, we have
\[
\mu_{\beta}\big(\,\mathcal{X}_{[g_{1}(L),\,g_{2}(L)]}\,\big)=1-o_{L}(1)\;.
\]
\end{lem}

\begin{proof}
It is enough to show that 
\begin{equation}
\mu_{\beta}\Big(\,\bigcup_{i<g_{1}(L)}\mathcal{X}_{i}\,\Big)=o_{L}(1)\;\;\;\;\text{and}\;\;\;\;\mu_{\beta}\Big(\,\bigcup_{i>g_{2}(L)}\mathcal{X}_{i}\,\Big)=o_{L}(1)\;.\label{e_muprop}
\end{equation}
To prove the first one, it suffices to prove that
\begin{equation}
\sum_{i<g_{1}(L)}|\mathcal{X}_{i}|e^{-\beta i}\ll\sum_{i\ge g_{1}(L)}|\mathcal{X}_{i}|e^{-\beta i}\;.\label{e_pf1}
\end{equation}
By part (1) of Lemma \ref{l_Xi}, we have
\[
\sum_{i<g_{1}(L)}|\mathcal{X}_{i}|e^{-\beta i}\le q\times\sum_{\substack{n_{3},\,n_{4},\,n_{5},\,n_{6}\ge0:\\
3n_{3}+4n_{4}+5n_{5}+6n_{6}<g_{1}(L)
}
}{\theta L^{2} \choose n_{3}}{\theta L^{2} \choose n_{4}}{\theta L^{2} \choose n_{5}}{\theta L^{2} \choose n_{6}}(qe^{-\beta})^{3n_{3}+4n_{4}+5n_{5}+6n_{6}}\;.
\]
Let $L$ be large enough so that $qe^{-\beta}<1$. Then, the summation
at the right-hand side is bounded from above by
\begin{align}
 & \sum_{\substack{n_{3},\,n_{4},\,n_{5},\,n_{6}\ge0:\\
3n_{3}+3n_{4}+3n_{5}+3n_{6}<g_{1}(L)
}
}{\theta L^{2} \choose n_{3}}{\theta L^{2} \choose n_{4}}{\theta L^{2} \choose n_{5}}{\theta L^{2} \choose n_{6}}(qe^{-\beta})^{3n_{3}+3n_{4}+3n_{5}+3n_{6}}\nonumber \\
 & =\sum_{i<\frac{g_{1}(L)}{3}}\sum_{\substack{n_{3},\,n_{4},\,n_{5},\,n_{6}\ge0:\\
n_{3}+n_{4}+n_{5}+n_{6}=i
}
}{\theta L^{2} \choose n_{3}}{\theta L^{2} \choose n_{4}}{\theta L^{2} \choose n_{5}}{\theta L^{2} \choose n_{6}}(qe^{-\beta})^{3i}=\sum_{i<\frac{g_{1}(L)}{3}}{4\theta L^{2} \choose i}(qe^{-\beta})^{3i}\;,\label{e_pf1.5}
\end{align}
where at the last equality we used a combinatorial identity of the
form 
\begin{equation}
\sum_{x+y+z+w=k}{a \choose x}{b \choose y}{c \choose z}{d \choose w}={a+b+c+d \choose k}\;.\label{e_combid}
\end{equation}
We can further bound the last summation in \eqref{e_pf1.5} from above
by 
\[
\sum_{i<\frac{g_{1}(L)}{3}}\frac{(4\theta L^{2})^{i}}{i!}(qe^{-\beta})^{3i}\le\frac{g_{1}(L)+3}{3}\cdot\frac{(4\theta L^{2})^{\frac{g_{1}(L)}{3}}\cdot(qe^{-\beta})^{g_{1}(L)}}{\lfloor\frac{g_{1}(L)}{3}\rfloor!}\le g_{1}(L)\cdot\Big(\frac{CL^{2}e^{-3\beta}}{g_{1}(L)}\Big)^{\frac{g_{1}(L)}{3}}
\]
using $g_{1}(L)\ll L^{2}e^{-3\beta}$ and an elementary bound $n!\ge n^{n}/e^{n}$.
Summing up, we get 
\begin{equation}
\sum_{i<g_{1}(L)}|\mathcal{X}_{i}|e^{-\beta i}\le qg_{1}(L)\cdot\Big(\frac{CL^{2}e^{-3\beta}}{g_{1}(L)}\Big)^{\frac{g_{1}(L)}{3}}\;.\label{e_pf2}
\end{equation}

Next, let $\widetilde{g}_{1}(L)=\lfloor g_{1}(L)^{2/3}(L^{2}e^{-3\beta})^{1/3}\rfloor$
so that we have $g_{1}(L)\ll\widetilde{g}_{1}(L)\ll L^{2}e^{-3\beta}$.
Then, by part (2) of Lemma \ref{l_Xi}, we have
\[
\sum_{i\ge g_{1}(L)}|\mathcal{X}_{i}|e^{-\beta i}\ge|\mathcal{X}_{3\widetilde{g}_{1}(L)}|e^{-3\beta\widetilde{g}_{1}(L)}\ge4^{\widetilde{g}_{1}(L)}{\lfloor\frac{L^{2}}{2}\rfloor \choose \widetilde{g}_{1}(L)}\times e^{-3\beta\widetilde{g}_{1}(L)}\;.
\]
By Stirling's formula and $\widetilde{g}_{1}(L)\ll L^{2}e^{-3\beta}\ll L^{2}$,
this is bounded from below by, for all large enough $L$, 
\begin{equation}
\frac{1}{2}(4e)^{\widetilde{g}_{1}(L)}\frac{(\frac{L^{2}}{3})^{\widetilde{g}_{1}(L)}}{\widetilde{g}_{1}(L)^{\widetilde{g}_{1}(L)}\sqrt{2\pi\widetilde{g}_{1}(L)}}\cdot e^{-3\beta\widetilde{g}_{1}(L)}\ge\Big(\frac{L^{2}e^{-3\beta}}{\widetilde{g}_{1}(L)}\Big)^{\widetilde{g}_{1}(L)}\gg\Big(\frac{L^{2}e^{-3\beta}}{\widetilde{g}_{1}(L)}\Big)^{g_{1}(L)}\;.\label{e_pf3}
\end{equation}
Therefore by \eqref{e_pf2} and \eqref{e_pf3}, we can reduce the
proof of \eqref{e_pf1} into 
\[
\frac{L^{2}e^{-3\beta}}{\widetilde{g}_{1}(L)}\gg\Big(\frac{L^{2}e^{-3\beta}}{g_{1}(L)}\Big)^{1/3}\;.
\]
This follows from the definition of $\widetilde{g}_{1}(L)$ and the
fact that $g_{1}(L)\ll L^{2}e^{-3\beta}$. This proves the first statement
in \eqref{e_muprop}. 

Next, to prove the second estimate of \eqref{e_muprop}, it suffices
to prove
\[
\sum_{i>g_{2}(L)}|\mathcal{X}_{i}|e^{-\beta i}\ll1
\]
since the partition function $Z_{\beta}$ has a trivial lower bound
$Z_{\beta}\ge q$ (by only considering the ground states). By a similar
computation leading to \eqref{e_pf2}, we get
\begin{equation}
\sum_{i>g_{2}(L)}|\mathcal{X}_{i}|e^{-\beta i}\le q\sum_{i>\frac{g_{2}(L)}{6}}\frac{(4\theta L^{2})^{i}}{i!}(qe^{-\beta})^{3i}\;.\label{e_pf4}
\end{equation}
Here, Taylor's theorem on the function $x\mapsto e^{x}$ implies that
for $x>0$ and $M\in\mathbb{N}$,
\begin{equation}
\sum_{i>M}\frac{x^{i}}{i!}\le\max_{t\in[0,\,x]}|e^{t}|\times\frac{x^{M+1}}{(M+1)!}=\frac{e^{x}x^{M+1}}{(M+1)!}\;.\label{e_pf5}
\end{equation}
Therefore, the right-hand side of \eqref{e_pf4} is bounded from above
by
\begin{align*}
e^{CL^{2}e^{-3\beta}}\times\frac{(CL^{2}e^{-3\beta})^{\frac{g_{2}(L)}{6}}}{(\frac{g_{2}(L)}{6})^{\frac{g_{2}(L)}{6}}} & \le\Big[\,6C(e^{6C})^{\frac{L^{2}e^{-3\beta}}{g_{2}(L)}}\times\frac{L^{2}e^{-3\beta}}{g_{2}(L)}\,\Big]^{\frac{g_{2}(L)}{6}}\;.
\end{align*}
As $L^{2}e^{-3\beta}\ll g_{2}(L)$, this expression vanishes as $L\rightarrow\infty$.
This concludes the proof.
\end{proof}

\subsection{Proof of Theorems \ref{t_Gibbs1} and \ref{t_Gibbs2}}

Now, we are ready to prove Theorems \ref{t_Gibbs1} and \ref{t_Gibbs2}.
Remark that the constant $\gamma_{0}$ is $\frac{2}{3}$ since we
consider the hexagonal lattice.
\begin{proof}[Proof of Theorem \ref{t_Gibbs1}]
 (1) It suffices to prove that, for some constant $C>0$, 
\begin{equation}
\sum_{\sigma\in\mathcal{X}\setminus\mathcal{S}}e^{-\beta H(\sigma)}=\sum_{i=3}^{3L^{2}}|\mathcal{X}_{i}|e^{-\beta i}\ll1\;,\label{ebdd1}
\end{equation}
where the identity follows from the observation that the minimum non-zero
value of the Hamiltonian is $3$ and the maximum is $3L^{2}$. By
part (1) of Lemma \ref{l_Xi}, we have (for $qe^{-\beta}<1$) 
\begin{align*}
\sum_{i=3}^{3L^{2}}|\mathcal{X}_{i}|e^{-\beta i} & \le q\times\sum_{i=3}^{3L^{2}}\sum_{\substack{n_{3},\,n_{4},\,n_{5},\,n_{6}\ge0:\\
3n_{3}+4n_{4}+5n_{5}+6n_{6}=i
}
}{\theta L^{2} \choose n_{3}}{\theta L^{2} \choose n_{4}}{\theta L^{2} \choose n_{5}}{\theta L^{2} \choose n_{6}}(qe^{-\beta})^{3n_{3}+4n_{4}+5n_{5}+6n_{6}}\\
 & \le q\times\sum_{\substack{n_{3},\,n_{4},\,n_{5},\,n_{6}\ge0:\\
n_{3}+n_{4}+n_{5}+n_{6}\ge1
}
}{\theta L^{2} \choose n_{3}}{\theta L^{2} \choose n_{4}}{\theta L^{2} \choose n_{5}}{\theta L^{2} \choose n_{6}}(qe^{-\beta})^{3n_{3}+3n_{4}+3n_{5}+3n_{6}}\;.
\end{align*}
Summing up and applying \eqref{e_combid}, we get
\[
\sum_{\sigma\in\mathcal{X}\setminus\mathcal{S}}e^{-\beta H(\sigma)}\le q\sum_{i=1}^{\infty}{4\theta L^{2} \choose i}(qe^{-\beta})^{3i}\le q\sum_{i=1}^{\infty}\frac{(4\theta q^{3}L^{2}e^{-3\beta})^{i}}{i!}\;.
\]
Again applying Taylor's theorem on the function $x\mapsto e^{x}$
(cf. \eqref{e_pf5}) for $x=4\theta q^{3}L^{2}e^{-3\beta}$, the last
summation is bounded by
\[
e^{4\theta q^{3}L^{2}e^{-3\beta}}\times(4\theta q^{3}L^{2}e^{-3\beta})\;.
\]
This completes the proof of \eqref{ebdd1} since we have $L^{2}e^{-3\beta}\ll1$
by assumption.\medskip{}

\noindent (2) We have $e^{\beta}\ll L^{2/3}$ and therefore as in
Lemma \ref{l_muprop} we can take two sequences $(g_{1}(L))_{L\in\mathbb{N}}$
and $(g_{2}(L))_{L\in\mathbb{N}}$ satisfying 
\[
1\ll g_{1}(L)\ll L^{2}e^{-3\beta}\ll g_{2}(L)\;.
\]
Then, by Lemma \ref{l_muprop}, the measure $\mu_{\beta}$ is concentrated
on $\mathcal{X}_{[g_{1}(L),\,g_{2}(L)]}$ and therefore $\mu_{\beta}(\mathcal{S})=o_{L}(1)$.
\end{proof}
\begin{proof}[Proof of Theorem \ref{t_Gibbs2}]
 (1) Since $L^{\frac{2}{3}(1-\alpha)}\ll e^{\beta}$, we have $L^{2}e^{-3\beta}\ll cL^{2\alpha}$
for any $c>0$. Thus, we can complete the proof by recalling Lemma
\ref{l_muprop} with any $g_{1}(L)$ such that $1\ll g_{1}(L)\ll L^{2}e^{-3\beta}$
(which is possible since we assumed that $e^{\beta}\ll L^{\gamma_{0}}$)
and $g_{2}(L)=cL^{2\alpha}$.\medskip{}

\noindent (2) We can take $g_{1}(L)=cL^{2\alpha}$ (and any $g_{2}(L)$
such that $L^{2}e^{-3\beta}\ll g_{2}(L)$) to get $\mu_{\beta}(\mathcal{X}_{[cL^{2\alpha},\,g_{2}(L)]})=1-o_{L}(1)$.
This completes the proof.
\end{proof}

\subsection{\label{sec45}Remarks on square lattice case}

For the square lattice case, a slightly different version of Lemma
\ref{l_graph} is required. More precisely, we need a version which
is obtained from Lemma \ref{l_graph} by replacing the set $\{3,\,4,\,5,\,6\}$
with $\{4,\,5,\,\dots,\,9\}$. This modification comes from the fact
that the minimal cycle in the dual graph $\Lambda^{*}$ has three
edges in the hexagonal lattice but has \textit{four} edges in the
square lattice case (cf. proof of Lemma \ref{l_Xi}). The proof of
this lemma is similar to that of Lemma \ref{l_graph} and we will
not repeat the proof. As a consequence of this modification, the upper
and lower bounds appeared in Lemma \ref{l_Xi} should be replaced
with 
\[
|\mathcal{X}_{i}|\le q^{i+1}\times\sum_{\substack{n_{4},\,n_{5},\,\dots,\,n_{9}\ge0:\\
4n_{4}+5n_{5}+\cdots+9n_{9}=i
}
}{\theta L^{2} \choose n_{4}}{\theta L^{2} \choose n_{5}}\cdots{\theta L^{2} \choose n_{9}}
\]
and $|\mathcal{X}_{4j}|\ge5^{j}{\lfloor\frac{L^{2}}{5}\rfloor \choose j}$
for $1\le j<\lfloor\frac{L^{2}}{5}\rfloor$, respectively. The constant
$\gamma_{0}$ for the square lattice is different to that for the
hexagonal one because of this modification.

\section{\label{Sec5_Pre}Preliminaries for Energy Landscape}

In this section, we introduce several preliminary notation and results
which are useful in the subsequent analysis of the energy landscape. 

\subsection{\label{sec51}Strip, bridge and cross. }

In this subsection, we provide some crucial notation regarding the
structure of the dual lattice $\Lambda^{*}$. We refer to Figure \ref{fig51}
for illustrations of the notation defined below and we consistently
refer to this figure.

\begin{figure}[h]
\includegraphics[width=11cm]{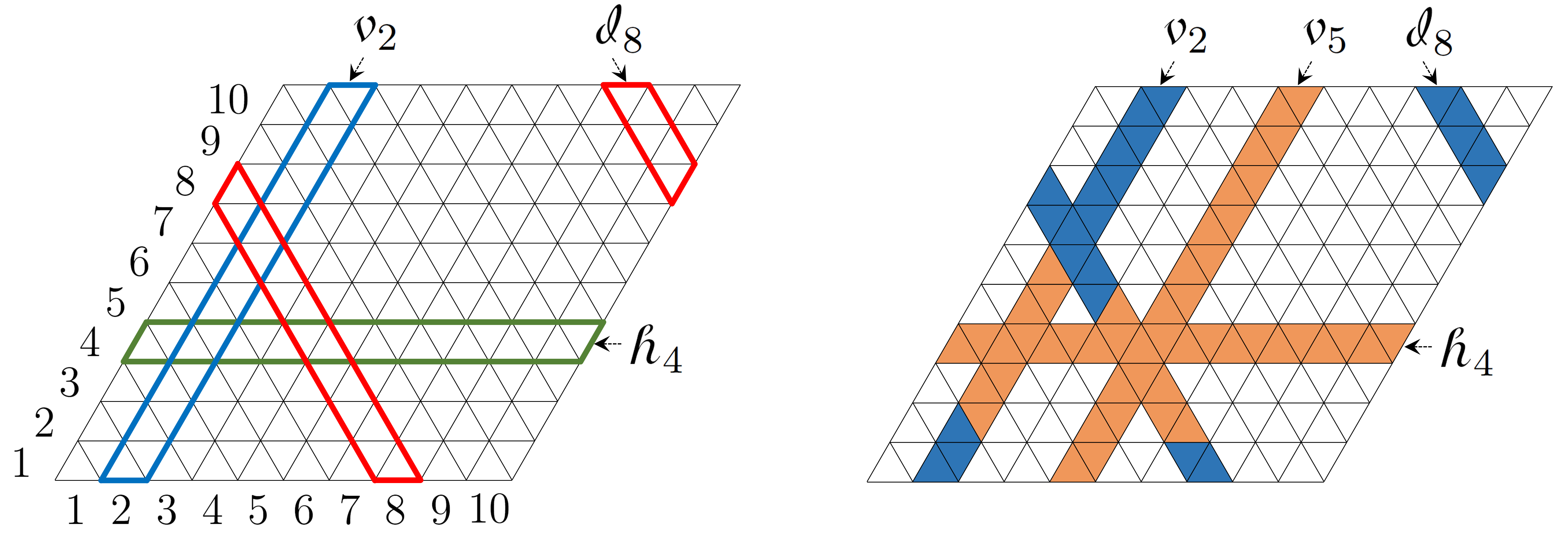}\caption{\label{fig51}\textbf{(Left)} Strips $\scal{h}_{4}$, $\scal{v}_{2}$,
and $\scal{d}_{8}$.\textbf{ (Right)} \textit{Here and in the following
figures, white, orange, and blue colors represent spins $a,\,b,\,c$,
respectively.}\textbf{ }Strips $\scal{h}_{4}$ and $\scal{v}_{5}$
are $b$-bridges and thus they form a $b$-cross. Strips $\scal{v}_{2}$
and $\scal{d}_{8}$ are $\{b,\,c\}$-semibridges.}
\end{figure}

\begin{defn}[Strip, bridge, cross and semibridge]
\label{d_strip} We define the crucial concepts here.
\begin{enumerate}
\item We denote by a \textit{strip} the $2L$ consecutive triangles in $\Lambda^{*}$
as illustrated in Figure \ref{fig51}-(left).\textbf{ }We may regard
each strip as a discrete torus $\mathbb{T}_{2L}$ via the obvious
manner.
\item There are three possible directions for strips. We call these three
directions as \textit{horizontal}, \textit{vertical}, and \textit{diagonal},
and these are highlighted by black, blue, and red lines in Figure
\ref{fig51}-(left), respectively. For each $\ell\in\mathbb{T}_{L}=\{1,\,2,\,\dots,\,L\}$,
the $\ell$-th strip of horizontal, vertical, and diagonal directions
are denoted by $\scal{h}_{\ell}$, $\scal{v}_{\ell}$, and $\scal{d}_{\ell}$,
respectively, as in Figure \ref{fig51}-(left).
\item A strip $\scal{s}$ is called a \textit{bridge} of $\sigma\in\mathcal{X}$
if all the spins of $\sigma$ in $\scal{s}$ are the same. If this
spin is $a$, we call $\scal{s}$ an \textit{$a$-bridge} of $\sigma$.
Furthermore, we can specify the direction of a bridge by calling it
a \textit{horizontal, vertical, or diagonal bridge} of $\sigma$.
Finally, the union of two bridges of different directions (of spin
$a$) is called a \textit{cross} (an \textit{$a$-cross}). We refer
to Figure \ref{fig51}-(right).
\item A strip $\scal{s}$ is called a \textit{semibridge }of $\sigma\in\mathcal{X}$,
if the strip $\scal{s}$ in $\sigma$ consists of exactly two spins,
and moreover the sites in $\scal{s}$ with either of these spins are
consecutive. If a semibridge consists of two spins $a$ and $b$,
we say that it is an $\{a,\,b\}$-semibridge. We refer to Figure \ref{fig51}-(right). 
\end{enumerate}
\end{defn}

\subsection{\label{sec52}Low-dimensional decomposition of energy}

For each strip $\scal{s}$, the energy of a configuration $\sigma$
on the strip $\scal{s}$ is defined as
\[
\Delta H_{\scal{s}}(\sigma)=\sum_{x,\,y\in\scal{s}:\,x\sim y}\mathbf{1}\{\sigma(x)\ne\sigma(y)\}
\]
so that by the definition of the Hamiltonian $H$, we have the following
decomposition 
\begin{equation}
H(\sigma)=\frac{1}{2}\sum_{\ell\in\mathbb{T}_{L}}\big[\,\Delta H_{\scal{h}_{\ell}}(\sigma)+\Delta H_{\scal{v}_{\ell}}(\sigma)+\Delta H_{\scal{d}_{\ell}}(\sigma)\,\big]\;,\label{e_Edec}
\end{equation}
where the term $1/2$ appears since each edge is counted twice. The
following simple fact is worth mentioning explicitly.
\begin{lem}
\label{l_strip}Suppose that a strip $\scal{s}$ is not a bridge of
$\sigma$. Then, we have 
\[
\Delta H_{\scal{s}}(\sigma)\ge2\;,
\]
and furthermore $\Delta H_{\scal{s}}(\sigma)=2$ if and only if $\scal{s}$
is a semibridge of $\sigma$.
\end{lem}

\begin{proof}
The proof is straightforward by identifying a strip $\scal{s}$ with
$\mathbb{T}_{2L}$ as in Definition \ref{d_strip}-(1).
\end{proof}
The next lemma provides an elementary lower bound on the number of
bridges based on the energy of configurations. Let us denote by $B_{a}(\sigma)$
the number of $a$-bridges in $\sigma\in\mathcal{X}$. 
\begin{lem}
\label{l_Elb}For $\sigma\in\mathcal{X}$, there are at least $3L-H(\sigma)$
bridges. Moreover, if $\sigma$ has exactly $3L-H(\sigma)$ bridges
then all strips are either bridges or semibridges.
\end{lem}

\begin{proof}
By \eqref{e_Edec} and Lemma \ref{l_strip}, we have
\begin{equation}
H(\sigma)\ge\frac{1}{2}\times2\times\Big[\,3L-\sum_{a\in\Omega}B_{a}(\sigma)\,\Big]=3L-\sum_{a\in\Omega}B_{a}(\sigma)\;.\label{e_Elb}
\end{equation}
This proves that there are at least $3L-H(\sigma)$ bridges. Moreover,
by Lemma \ref{l_strip}, a strip which is not a bridge should be a
semibridge in order to have the equality in the bound \eqref{e_Elb}.
This completes the proof.
\end{proof}

\subsection{\label{sec53}Neighborhoods}

Recall the notion of paths from \eqref{e_path}. We say that a path
$(\omega_{n})_{n=0}^{N}$ is in $\mathcal{A}\subseteq\mathcal{X}$
if $\omega_{n}\in\mathcal{A}$ for all $n\in\llbracket0,\,N\rrbracket$.
For $t\in\mathbb{R}$, we say that a path $(\omega_{n})_{n=0}^{N}$
is a \textit{$t$-path} if $H(\omega_{n})\le t$ for all $n\in\llbracket0,\,N\rrbracket$.
\begin{defn}
\label{d_nbd}We define two types of neighborhoods.
\begin{enumerate}
\item For $\sigma\in\mathcal{X}$, the neighborhoods $\mathcal{N}(\sigma)$
and $\widehat{\mathcal{N}}(\sigma)$ are defined as
\begin{align*}
\mathcal{N}(\sigma) & =\{\zeta\in\mathcal{X}:\exists\text{a }(2L+1)\text{-path connecting }\sigma\text{ and }\zeta\}\;,\\
\widehat{\mathcal{N}}(\sigma) & =\{\zeta\in\mathcal{X}:\exists\text{a }(2L+2)\text{-path connecting }\sigma\text{ and }\zeta\}\;.
\end{align*}
Then for $\mathcal{A}\subseteq\mathcal{X}$, we define 
\[
\mathcal{N}(\mathcal{A})=\bigcup_{\sigma\in\mathcal{A}}\mathcal{N}(\sigma)\;\;\;\;\text{and}\;\;\;\;\widehat{\mathcal{N}}(\mathcal{A})=\bigcup_{\sigma\in\mathcal{A}}\mathcal{\widehat{N}}(\sigma)\;.
\]
We sometimes refer these as $\mathcal{N}$- and $\widehat{\mathcal{N}}$-neighborhoods,
respectively.
\item Let $\mathcal{B}\subseteq\mathcal{X}$. For $\sigma\in\mathcal{X}\setminus\mathcal{B}$,
we define
\[
\widehat{\mathcal{N}}(\sigma;\,\mathcal{B})=\{\zeta\in\mathcal{X}:\exists\text{a }(2L+2)\text{-path in }\mathcal{X}\setminus\mathcal{B}\text{ connecting }\sigma\text{ and }\zeta\}\;.
\]
Then for $\mathcal{A}\subseteq\mathcal{X}$ disjoint with $\mathcal{B}$,
we define 
\[
\widehat{\mathcal{N}}(\mathcal{A};\,\mathcal{B})=\bigcup_{\sigma\in\mathcal{A}}\widehat{\mathcal{N}}(\sigma;\,\mathcal{B})\;.
\]
\end{enumerate}
We remark that the numbers $2L+1$ and $2L+2$ appear in the definition
since it will be shown that $2L+2$ is the energy barrier $\Gamma$.
\end{defn}

\section{\label{Sec6_EB}Energy Barrier}

This section provides the first level of investigation of the energy
landscape which suffices to prove Theorem \ref{t_Ebarrier}, that
is, the energy barrier between ground states is $2L+2$. A deeper
analysis of the energy landscape required to prove the Eyring--Kramers
formula will be carried out in Section \ref{sec7_Saddle}. 

We collect here notation heavily used in the remainder of the article.
\begin{notation}
\label{not61}Here, the alphabets $\scal{h},\,\scal{v},\,\scal{d}$
stand for \textit{horizontal}, \textit{vertical}, and \textit{diagonal},
respectively.
\begin{enumerate}
\item We say that $(A,\,B)$ is a \textit{proper partition} (of $\Omega$)
if $A,\,B\ne\emptyset$, $A\cup B=\Omega$, and $A\cap B=\emptyset$.
\item Let $L\ge2$ and denote by $\mathfrak{S}_{L}$ the collection of connected
subsets of $\mathbb{T}_{L}$. For example, we have $\emptyset,\,\{2\},\,\{2,\,3,\,4,\,5\},\,\{6,\,1,\,2\}\in\mathfrak{S}_{6}$
(since $6$ and $1$ are neighboring in $\mathbb{T}_{6}$). 
\begin{enumerate}
\item For $P,\,P'\in\mathfrak{S}_{L}$, we write $P\prec P'$ if $P\subseteq P'$
and $|P'|=|P|+1$.
\item For each $P\in\mathfrak{S}_{L}$, we write
\[
\scal{h}(P)=\bigcup_{\ell\in P}\scal{h}_{\ell}\;,\;\;\;\;\scal{v}(P)=\bigcup_{\ell\in P}\scal{v}_{\ell}\;,\;\;\;\;\text{and}\;\;\;\;\scal{d}(P)=\bigcup_{\ell\in P}\scal{d}_{\ell}\;.
\]
\end{enumerate}
\item We regard the dual lattice $\Lambda^{*}$ as the collection of triangles
(corresponding to the sites, or vertices of $\Lambda$) and hence
we say that $U$ is a subset of $\Lambda^{*}$ (i.e., $U\subseteq\Lambda^{*}$)
if $U$ is a collection of triangles in $\Lambda^{*}$. For example,
a strip is a subset of $\Lambda^{*}$ consisting of $2L$ triangles.
\item For each $U\subseteq\Lambda^{*}$ and $a,\,b\in\Omega$, we write
$\xi_{U}^{a,\,b}\in\mathcal{X}$ the configuration whose spins are
$b$ on the sites corresponding to the triangles in $U$ and $a$
on the remainder.
\end{enumerate}
\end{notation}

\subsection{\label{sec61}Canonical configurations}

In this subsection, we define the canonical configurations between
ground states. These canonical configurations provide the backbone
of the saddle structure. We shall see in the sequel that the saddle
structure is completed by attaching dead-end structures or bypasses
at this backbone. We define canonical configurations in several steps.
The first step is devoted to define the regular configurations which
are indeed special forms of canonical configurations. 

\begin{figure}
\includegraphics[width=13cm]{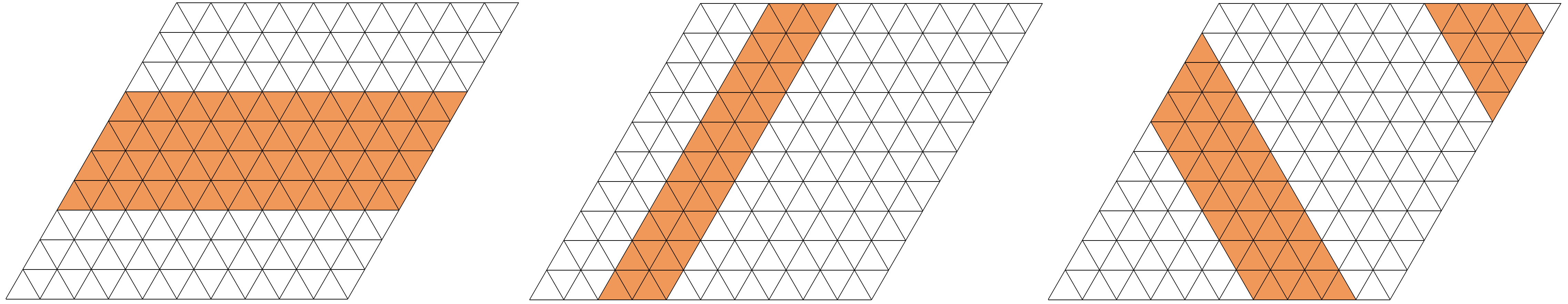}\caption{\label{fig61}Example of regular configurations: $\xi_{\scal{h}(\llbracket4,\,7\rrbracket)}^{a,\,b}$
(left), $\xi_{\scal{v}(\llbracket3,\,4\rrbracket)}^{a,\,b}$ (middle),
and $\xi_{\scal{d}(\llbracket7,\,9\rrbracket)}^{a,\,b}$ (right).}
\end{figure}

\begin{defn}[Regular configurations]
\label{d_reg} Fix $a,\,b\in\Omega$. We recall Notation \ref{not61}.
\begin{itemize}
\item A configuration of the form $\xi_{\scal{h}(P)}^{a,\,b}$, $\xi_{\scal{v}(P)}^{a,\,b}$,
or $\xi_{\scal{d}(P)}^{a,\,b}$ for some $P\in\mathfrak{S}_{L}$ is
called a \textit{horizontal, vertical, or diagonal regular configuration}
between $\mathbf{a}$ and $\mathbf{b}$, respectively. We refer to
Figure \ref{fig61} for illustrations.
\item For $n\in\llbracket0,\,L\rrbracket$, we define
\[
\mathcal{R}_{n}^{a,\,b}=\bigcup_{P\in\mathfrak{S}_{L}:\,|P|=n}\big\{\,\xi_{\scal{h}(P)}^{a,\,b},\,\xi_{\scal{v}(P)}^{a,\,b},\,\xi_{\scal{d}(P)}^{a,\,b}\,\big\}\;\;\;\;\text{and}\;\;\;\;\mathcal{R}^{a,\,b}=\bigcup_{n=0}^{L}\mathcal{R}_{n}^{a,\,b}\;.
\]
Then for a proper partition $(A,\,B)$, we write 
\[
\mathcal{R}_{n}^{A,\,B}=\bigcup_{a\in A}\bigcup_{b\in B}\mathcal{R}_{n}^{a,\,b}\;\;\;\;\text{and}\;\;\;\;\mathcal{R}^{A,\,B}=\bigcup_{a\in A}\bigcup_{b\in B}\mathcal{R}^{a,\,b}\;.
\]
\end{itemize}
\end{defn}

Canonical configurations are now defined as the ones obtained by adding
suitable protuberances at a monochromatic cluster of a regular configuration.
To carry this out rigorously, we first define canonical sets. 
\begin{defn}[One-dimensional canonical sets]
\label{d_ocs} We say that $U\subseteq\scal{s}$ for some strip $\scal{s}$
is an \textit{one-dimensional canonical set} if $U\ne\emptyset,\,\scal{s}$
and either $U$ is connected (we remark again that two triangles sharing
only a vertex are not connected) as in the two left figures below,
or $|U|$ is even and $U$ can be decomposed into two disjoint, connected
components $U_{1}$ and $U_{2}$ such that $|U_{2}|=1$ and that $U_{1}$
and $U_{2}$ share a vertex in $\Lambda^{*}$ as in the rightmost
figure below. 
\end{defn}

\noindent \begin{center}
\includegraphics[width=11cm]{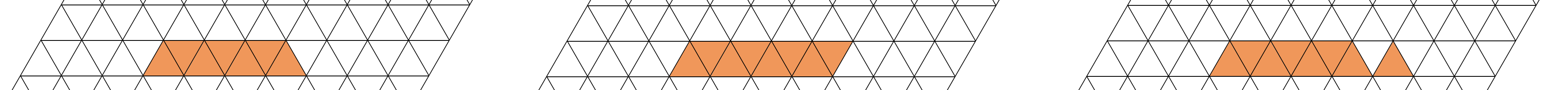}
\par\end{center}

We now define the general canonical sets.

\begin{figure}
\includegraphics[width=15cm]{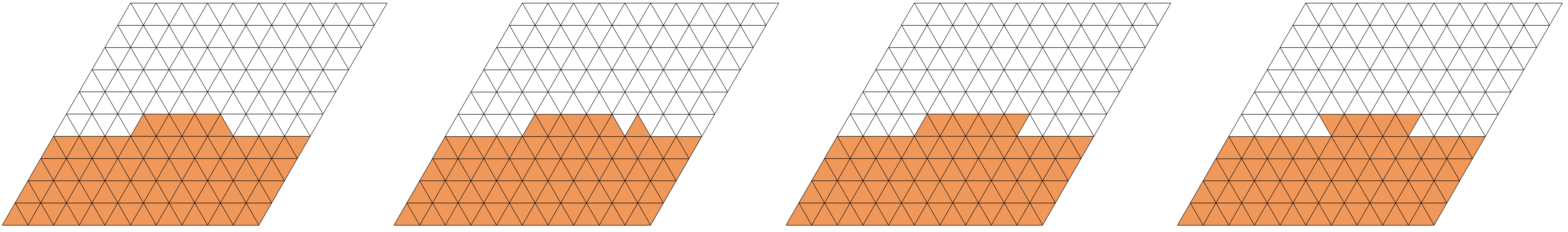}\caption{\label{fig62}Canonical sets and configurations. For the first three
figures, the sets of orange triangles are canonical sets between $\scal{h}(P)$
and $\scal{h}(P')$ where $P=\llbracket1,\,4\rrbracket$ and $P'=\llbracket1,\,5\rrbracket$.
If we assign spins $a$ and $b$ at white and orange triangles, respectively,
then the configurations corresponding to the first, second, and third
figures belong to $\mathcal{C}_{\scal{h}(P,\,P'),\,\scal{o}}^{a,\,b}$,
$\mathcal{C}_{\scal{h}(P,\,P'),\,\scal{e}}^{a,\,b}$, and $\mathcal{C}_{\scal{h}(P,\,P'),\,\scal{e}}^{a,\,b}$,
respectively. \emph{Note that the set of orange triangles at the rightmost
figure is not a canonical set since the condition \eqref{cond_prot}
is violated.}}
\end{figure}

\begin{defn}[Canonical sets]
\label{d_cs} Fix $a,\,b\in\Omega$, $\scal{s}\in\{\scal{h},\,\scal{v},\,\scal{d}\}$,
and $P,\,P'\in\mathfrak{S}_{L}$ such that $P\prec P'$. Let $P'\setminus P=\{\ell\}$.
We now define the canonical sets between $\scal{s}(P)$ and $\scal{s}(P')$.
We refer to Figure \ref{fig62}.
\begin{enumerate}
\item A set $\mathfrak{p}\subseteq\scal{s}_{\ell}$ is called a \textit{protuberance}
attached to $\scal{s}(P)$ if $\mathfrak{p}$ is an one-dimensional
canonical set and moreover, for $|P|\in\llbracket1,\,L-2\rrbracket$,
it holds that 
\begin{equation}
\big|\,\{x\in\mathfrak{p}:x\text{ shares a side with some }y\in\scal{s}(P)\}\,\big|\ge\frac{|\mathfrak{p}|}{2}\;.\label{cond_prot}
\end{equation}
\item The set $\scal{s}(P)\cup\mathfrak{p}$, where $\mathfrak{p}$ is a
protuberance attached to $\scal{s}(P)$, is called a \textit{canonical
set} between $\scal{s}(P)$ and $\scal{s}(P')$.
\end{enumerate}
\end{defn}

We are now finally able to define the canonical configurations. In
the following definition, the alphabets $\scal{o}$ and $\scal{e}$
in the subscripts stand for \textit{odd} and \textit{even}, respectively.
\begin{defn}[Canonical configurations]
\label{d_can} We define the canonical configurations (we refer to
Figure \ref{fig62} for an illustrations). 
\begin{enumerate}
\item Fix $a,\,b\in\Omega$, $\scal{s}\in\{\scal{h},\,\scal{v},\,\scal{d}\}$
and $P,\,P'\in\mathfrak{S}_{L}$ with $P\prec P'$. We say that a
configuration $\sigma\in\mathcal{X}$ is a canonical configuration
between two regular configurations $\xi_{\scal{s}(P)}^{a,\,b}$ and
$\xi_{\scal{s}(P')}^{a,\,b}$ if 
\[
\sigma=\xi_{A}^{a,\,b}\text{ for some canonical set }A\text{ between }\scal{s}(P)\text{ and }\scal{s}(P')\;.
\]
We denote by $\widetilde{\mathcal{C}}_{\scal{s}(P,\,P')}^{a,\,b}$
the collection of canonical configurations between $\xi_{\scal{s}(P)}^{a,\,b}$
and $\xi_{\scal{s}(P')}^{a,\,b}$.
\begin{enumerate}
\item For each $\sigma=\xi_{A}^{a,\,b}\in\widetilde{\mathcal{C}}_{\scal{s}(P,\,P')}^{a,\,b}$,
we can decompose $A$ into $\scal{s}(P)$ and the protuberance attached
to it (cf. Definition \ref{d_cs}). We denote this protuberance by
$\mathfrak{p}^{a,\,b}(\sigma)$\footnote{Note that if $\sigma\in\widetilde{\mathcal{C}}_{\scal{s}(P,\,P')}^{a,\,b}$,
then we also have $\sigma\in\widetilde{\mathcal{C}}_{\scal{s}(\mathbb{T}_{L}\setminus P',\,\mathbb{T}_{L}\setminus P)}^{b,\,a}$
and moreover $\mathfrak{p}^{b,\,a}(\sigma)=\scal{s}_{\ell}\setminus\mathfrak{p}^{a,\,b}(\sigma)$.}.
\item We write
\begin{align*}
\mathcal{C}_{\scal{s}(P,\,P')}^{a,\,b} & =\widetilde{\mathcal{C}}_{\scal{s}(P,\,P')}^{a,\,b}\cup\big\{\,\xi_{\scal{s}(P)}^{a,\,b},\,\xi_{\scal{s}(P')}^{a,\,b}\,\big\}\;,\\
\mathcal{C}_{\scal{s}(P,\,P'),\,\scal{o}}^{a,\,b} & =\big\{\,\sigma\in\widetilde{\mathcal{C}}_{\scal{s}(P,\,P')}^{a,\,b}:|\mathfrak{p}^{a,\,b}(\sigma)|\text{ is odd}\,\big\}\;,\\
\mathcal{C}_{\scal{s}(P,\,P'),\,\scal{e}}^{a,\,b} & =\big\{\,\sigma\in\widetilde{\mathcal{C}}_{\scal{s}(P,\,P')}^{a,\,b}:|\mathfrak{p}^{a,\,b}(\sigma)|\text{ is even}\,\big\}\;.
\end{align*}
\end{enumerate}
\item For $n\in\llbracket0,\,L-1\rrbracket$ and $a,\,b\in\Omega$, we define
\[
\mathcal{C}_{n}^{a,\,b}=\bigcup_{\scal{s}\in\{\scal{h},\,\scal{v},\,\scal{d}\}}\bigcup_{P\prec P':\,|P|=n}\mathcal{C}_{\scal{s}(P,\,P')}^{a,\,b}\;,
\]
and define $\mathcal{C}_{n,\,\scal{o}}^{a,\,b}$ and $\mathcal{C}_{n,\,\scal{e}}^{a,\,b}$
in the same manner. The configurations belonging to $\mathcal{C}_{n}^{a,\,b}$
for some $n\in\llbracket0,\,L-1\rrbracket$ are called \textit{canonical
configurations} between $\mathbf{a}$ and $\mathbf{b}$. 
\item For each proper partition $(A,\,B)$ (cf. Notation \ref{not61}),
we write
\[
\mathcal{C}_{n,\,\scal{o}}^{A,\,B}=\bigcup_{a\in A}\bigcup_{b\in B}\mathcal{C}_{n,\,\scal{o}}^{a,\,b}\;\;\;\;\text{and}\;\;\;\;\mathcal{C}_{n,\,\scal{e}}^{A,\,B}=\bigcup_{a\in A}\bigcup_{b\in B}\mathcal{C}_{n,\,\scal{e}}^{a,\,b}\;.
\]
\end{enumerate}
\end{defn}

\begin{rem}[Energy of canonical configurations]
\label{r_can} The following properties of regular and canonical
configurations are straightforward from the definitions. In particular,
the discussion on Section \ref{sec22} or \eqref{e_Edec} can be used,
and we omit the detail of the proof. Let $a,\,b\in\Omega$.
\begin{enumerate}
\item For $n\in\llbracket1,\,L-2\rrbracket$, we can decompose 
\[
\mathcal{C}_{n}^{a,\,b}=\mathcal{R}_{n}^{a,\,b}\cup\mathcal{R}_{n+1}^{a,\,b}\cup\mathcal{C}_{n,\,\scal{o}}^{a,\,b}\cup\mathcal{C}_{n,\,\scal{e}}^{a,\,b}
\]
and we have 
\[
H(\sigma)=\begin{cases}
2L & \text{if }\sigma\in\mathcal{R}_{n}^{a,\,b}\cup\mathcal{R}_{n+1}^{a,\,b}\;,\\
2L+1 & \text{if }\sigma\in\mathcal{C}_{n,\,\scal{o}}^{a,\,b}\;,\\
2L+2 & \text{if }\sigma\in\mathcal{C}_{n,\,\scal{e}}^{a,\,b}\;.
\end{cases}
\]
\item If $\sigma\in\mathcal{C}_{n}^{a,\,b}$ for $n=0$ or $L-1$, we have
$H(\sigma)\le2L+1$.
\end{enumerate}
In conclusion, we have $H(\sigma)\le2L+2$ for all canonical configurations
$\sigma$.
\end{rem}

\begin{rem}[Canonical paths]
\label{r_canpath} Fix $a,\,b\in\Omega$, $\scal{s}\in\{\scal{h},\,\scal{v},\,\scal{d}\}$,
and $P,\,P'\in\mathfrak{S}_{L}$ with $P\prec P'$. Then, it is clear
by definition that there are natural paths in $\mathcal{C}_{\scal{s}(P,\,P')}^{a,\,b}$
from $\xi_{\scal{s}(P)}^{a,\,b}$ to $\xi_{\scal{s}(P')}^{a,\,b}$
as in the following figure.

\begin{figure}[h]
\includegraphics[width=15cm]{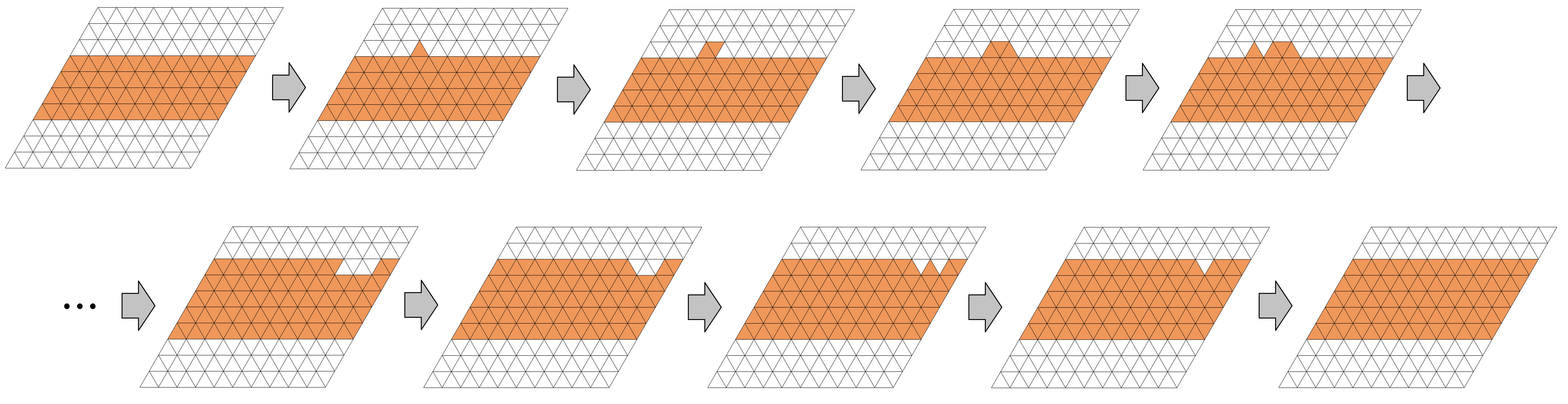}\caption{\label{fig63}Canonical path from $\xi_{\scal{h}(\llbracket4,\,7\rrbracket)}^{a,\,b}$
to $\xi_{\scal{h}(\llbracket4,\,8\rrbracket)}^{a,\,b}$.}
\end{figure}

\noindent These paths are called \textit{canonical paths} between
$\xi_{\scal{s}(P)}^{a,\,b}$ and $\xi_{\scal{s}(P')}^{a,\,b}$. By
attaching the canonical paths consecutively, one can obtain a path
between $\mathbf{a}$ and $\mathbf{b}$. This path is called a \textit{canonical
path} between $\mathbf{a}$ and $\mathbf{b}$. Note that there are
numerous possible canonical paths between $\mathbf{a}$ and $\mathbf{b}$,
and that each canonical path is a $(2L+2)$-path (cf. Section \ref{sec53})
by Remark \ref{r_can} above.
\end{rem}

\subsection{\label{sec62}Configurations with low energy}

Since the energy barrier between ground states is $2L+2$ (as will
be proved in this section), the saddle structure between ground states
is essentially the $\widehat{\mathcal{N}}$-neighborhood (cf. Definition
\ref{d_nbd}) of canonical configurations. Therefore, to understand
the saddle structure, it is crucial to characterize the configurations
with energy exactly $2L+2$. This characterization is relatively simple
for the square lattice (cf. \cite[Proposition 6.8 and Lemma 7.2]{KS2}),
as dead-ends are attached only at the very end of the canonical paths.
However, this characterization is highly non-trivial for the hexagonal
lattice, as we shall see that a complicated dead-end structure is
attached at each regular configuration. This and the next subsections
are devoted to study this structure.

A configuration $\sigma$ is called \textit{cross-free} if it does
not have a cross (cf. Definition \ref{d_strip}-(3)). The purpose
of the current subsection is to characterize all the cross-free configurations
$\sigma$ such that $H(\sigma)\le2L+2$. We first prove that a cross-free
configuration $\sigma$ has energy at least $2L$, and moreover the
energy is exactly $2L$ if and only if $\sigma$ is a regular configuration
(cf. Definition \ref{d_reg}).
\begin{prop}
\label{p_E<=00003D2L}Suppose that a cross-free configuration $\sigma\in\mathcal{X}$
satisfies $H(\sigma)\le2L$. Then, $\sigma$ is a regular configuration,
i.e., $\sigma\in\mathcal{R}_{n}^{a,\,b}$ for some $a,\,b\in\Omega$
and $n\in\llbracket1,\,L-1\rrbracket$. In particular, we have $H(\sigma)=2L$.
\end{prop}

\begin{proof}
We fix a cross-free configuration $\sigma\in\mathcal{X}$ with $H(\sigma)\le2L$.
By Lemma \ref{l_Elb}, $\sigma$ has at least $L$ bridges. Since
these bridges must be of the same direction, there are exactly $L$
bridges of the same direction (say, horizontal), and by the second
assertion of Lemma \ref{l_Elb}, all the vertical and diagonal strips
must be semibridges of the same form. We can conclude that $\sigma$
is a regular configuration by combining the observations above.
\end{proof}
It now remains to characterize cross-free configurations with energy
$2L+1$ or $2L+2$. The following lemma is useful for those characterizations.
\begin{lem}
\label{l_crosschar}Suppose that a cross-free configuration $\sigma\in\mathcal{X}$
satisfies $H(\sigma)\le2L+2$, has $k\in\{L-2,\,L-1\}$ horizontal
bridges, and has at least one vertical or diagonal semibridge. Then,
the following statements hold for the configuration $\sigma$.
\begin{enumerate}
\item There exist two spins $a,\,b\in\Omega$ such that all horizontal bridges
are either $a$- or $b$-bridges.
\item Following (1), define two sets $P_{a}$ and $P_{b}$ by 
\begin{equation}
P_{c}=\{\ell\in\mathbb{T}_{L}:\scal{h}_{\ell}\text{ is a }c\text{-bridge}\}\;\;\;\;;\;c\in\{a,\,b\}\;.\label{e_Pc}
\end{equation}
Suppose that $P_{a},\,P_{b}\ne\emptyset$. Then, we have $P_{a},\,P_{b}\in\mathfrak{S}_{L}$
and moreover
\begin{enumerate}
\item if $k=L-2$, then all non-bridge strips are $\{a,\,b\}$-semibridges
and $H(\sigma)=2L+2$,
\item if $k=L-1$ and $H(\sigma)\le2L+1$, then all non-bridge strips are
$\{a,\,b\}$-semibridges.
\end{enumerate}
\end{enumerate}
\end{lem}

\begin{rem}
The conclusion $P_{a},\,P_{b}\in\mathfrak{S}_{L}$ holds even when
either $P_{a}$ or $P_{b}$ is empty, but its proof will be given
later in Lemma \ref{l_MB}.
\end{rem}

\begin{proof}[Proof of Lemma \ref{l_crosschar}]
 (1) The conclusion is immediate since if the vertical or diagonal
semibridge of $\sigma$ (which exists because of the assumption of
the lemma) is an $\{a,\,b\}$-semibridge for some $a,\,b\in\Omega$,
then each horizontal bridge must be either $a$- or $b$-bridge.\medskip{}

\noindent (2) Suppose first that no $a$-bridge is adjacent to a $b$-bridge.
Then as $P_{a},\,P_{b}\ne\emptyset$ and $k\le L-1$, we may take
one connected subset $C_{a}$ of $P_{a}$ so that $|C_{a}|\le L-2$.
Then, the two strips adjacent to $C_{a}$ must not be $b$-bridges,
so that they are not bridges. Then since $k\ge L-2$, we conclude
that $C_{a}=P_{a}$ and all the strips which are not adjacent to $C_{a}$
are $b$-bridges. This implies that $P_{a},\,P_{b}\in\mathfrak{S}_{L}$.

Next, suppose that some $a$-bridge is adjacent to a $b$-bridge.
Without loss of generality, we assume that $1\in P_{a}$ and $L\in P_{b}$.
Let $m=\max\{i\in\llbracket1,\,L-1\rrbracket:i\in P_{a}\}$ and we
claim that $P_{a}=\llbracket1,\,m\rrbracket$. There is nothing to
prove if $m=1$ or $2$, since the claim holds immediately. Suppose
$m\ge3$ and there exists $i\in\llbracket2,\,m-1\rrbracket$ such
that $i\notin P_{a}$. Then, there exists a triangle in the $i$-th
horizontal strip at which the spin is not $a$. The vertical and diagonal
strips containing this triangle have energy at least $4$, because
$1\in P_{a}$, $m\in P_{a}$, and $L\in P_{b}$. All the vertical
and diagonal strips other than these two have energy at least $2$
(since the configuration $\sigma$ is cross-free). Since at least
one of the horizontal strip must be a non-bridge and has energy at
least $2$, we can conclude from \eqref{e_Edec} that 
\[
H(\sigma)\ge\frac{1}{2}\big[\,2+[4+2(L-1)]+[4+2(L-1)]\,\big]=2L+3\;.
\]
This yield a contradiction and thus we can conclude that $P_{a}=\llbracket1,\,m\rrbracket\in\mathfrak{S}_{L}$.
The proof of $P_{b}\in\mathfrak{S}_{L}$ is the same.\medskip{}

\noindent (2-a) For this case, we first note that there are $L-2$
bridges. If $H(\sigma)\le2L+1$, by Lemma \ref{l_Elb}, there are
at least $3L-H(\sigma)\ge L-1$ bridges and we get a contradiction.
Hence, we have $H(\sigma)=2L+2$ and there are $3L-H(\sigma)$ bridges;
hence, by the second assertion of Lemma \ref{l_Elb} all the non-bridge
strips are semibridges. It is clear that indeed, they must be $\{a,\,b\}$-semibridges.\medskip{}

\noindent (2-b) The proof for this part is almost identical to (2-a)
and we omit the detail.
\end{proof}
We next characterize all the cross-free configurations with energy
$2L+1$. Indeed, they must be canonical configurations. 
\begin{prop}
\label{p_E=00003D2L+1}Suppose that a cross-free configuration $\sigma\in\mathcal{X}$
satisfies $H(\sigma)=2L+1$. Then, $\sigma\in\mathcal{C}_{n,\,\scal{o}}^{a,\,b}$
for some $a,\,b\in\Omega$ and $n\in\llbracket0,\,L-1\rrbracket$.
Moreover, if $n=0$ (resp. $n=L-1$), then $|\mathfrak{p}^{a,\,b}(\sigma)|=2L-1$
(resp. $|\mathfrak{p}^{a,\,b}(\sigma)|=1$).
\end{prop}

\begin{proof}
By Lemma \ref{l_Elb}, the configuration $\sigma$ has at least $L-1$
bridges. Since $\sigma$ is cross-free, these bridges are of the same
direction, say horizontal. If there are $L$ horizontal bridges, then
all the vertical and diagonal strips are of the same form and thus
the energy of $\sigma$ should be a multiple of $L$ in view of \eqref{e_Edec}.
It contradicts $H(\sigma)=2L+1$, and hence there are exactly $L-1=3L-H(\sigma)$
bridges. By the second assertion of Lemma \ref{l_Elb}, \emph{all
the non-bridge strips are semibridges.} Now, by Lemma \ref{l_crosschar},
there exist $a,\,b\in\Omega$ such that all the horizontal bridges
are either $a$- or $b$-bridges. Define $P_{a}$ and $P_{b}$ as
in Lemma \ref{l_crosschar} and write $\{\ell\}=\mathbb{T}_{L}\setminus(P_{a}\cup P_{b})$.

Suppose first that either $P_{a}$ or $P_{b}$ is empty, say $P_{b}=\emptyset$
and $P_{a}=\mathbb{T}_{L}\setminus\{\ell\}$. Then as all strips are
either bridges or semibridges, we conclude that $\scal{h}_{\ell}$
is an $\{a,\,c\}$-semibridge for some $c\ne a$. As $\sigma$ is
cross-free, we must have $|\mathfrak{p}^{a,\,c}(\sigma)|=2L-1$. The
other case $P_{a}=\emptyset$ can be handled identically.

Next, suppose that $P_{a},\,P_{b}\ne\emptyset$ so that we can apply
case (2-b) of Lemma \ref{l_crosschar}, which implies that $\scal{h}_{\ell}$
is an $\{a,\,b\}$-semibridge. As illustrated in the figure below,
since all the vertical and diagonal strips are semibridge, we can
deduce that the set of triangles in $\scal{h}_{\ell}$ with spin $b$
should be an odd protuberance (cf. Definition \ref{d_can}) between
$\xi_{\scal{h}(P)}^{a,\,b}$ and $\xi_{\scal{h}(P')}^{a,\,b}$. Note
that for the other cases vertical strip with black bold boundary is
not a semibridge. 
\begin{center}
\includegraphics[width=13cm]{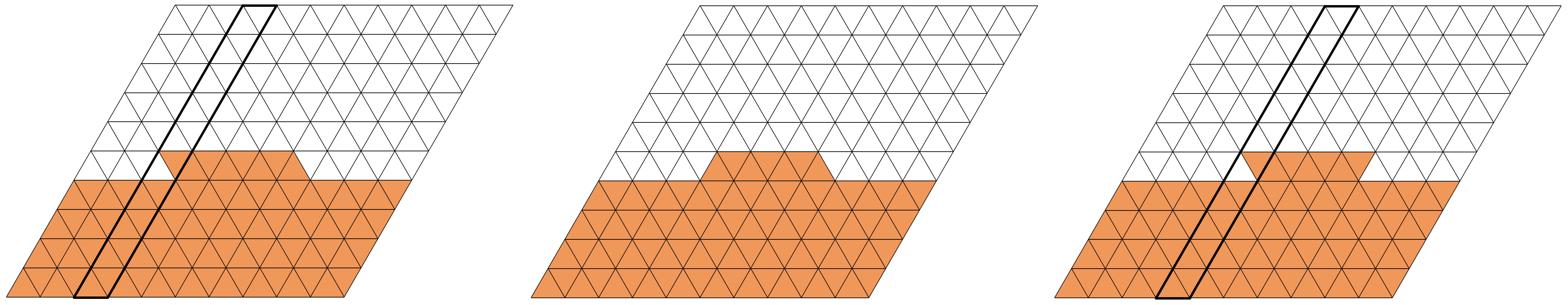}
\par\end{center}

\noindent Therefore, we can conclude that $\sigma\in\mathcal{C}_{n,\,\scal{o}}^{a,\,b}$
for some $n\in\llbracket1,\,L-2\rrbracket$. 
\end{proof}
Now, it remains to characterize the cross-free configurations with
energy $2L+2$. To this end, we introduce six different types of cross-free
configurations with energy $2L+2$ in the following definition. 
\begin{defn}[Cross-free configurations with energy $2L+2$]
\label{d_2L+2} The following types characterize the cross-free configurations
with energy $2L+2$. We refer to Figure \ref{fig64} below for illustrations
and to \eqref{e_Edec} for the verification of the fact that these
configurations (except\textbf{ (MB)}) have energy $2L+2$. 
\begin{lyxlist}{00.00.0000}
\item [{\textbf{(ODP)}}] One-sided Double Protuberances: two odd protuberances
are attached to one side of a regular configuration.
\item [{\textbf{(TDP)}}] Two-sided Double Protuberances: two odd protuberances
are attached to different sides of a regular configuration.
\item [{\textbf{(SP)}}] Superimposed Protuberances: an odd protuberance
is attached to a regular configuration, and another smaller odd protuberance
is attached to the first odd protuberance.
\item [{\textbf{(EP)}}] Even Protuberance: an even protuberance is attached
to a regular configuration.
\item [{\textbf{(PP)}}] Peculiar Protuberance: a protuberance of a third
spin and of size $1$ is attached to a regular configuration.
\item [{\textbf{(MB)}}] Monochromatic Bridges: all the bridges are parallel
and of the same spin, where more refined characterization of this
type will be given in Lemma \ref{l_MB}. 
\end{lyxlist}
\end{defn}

\begin{figure}[h]
\includegraphics[width=14cm]{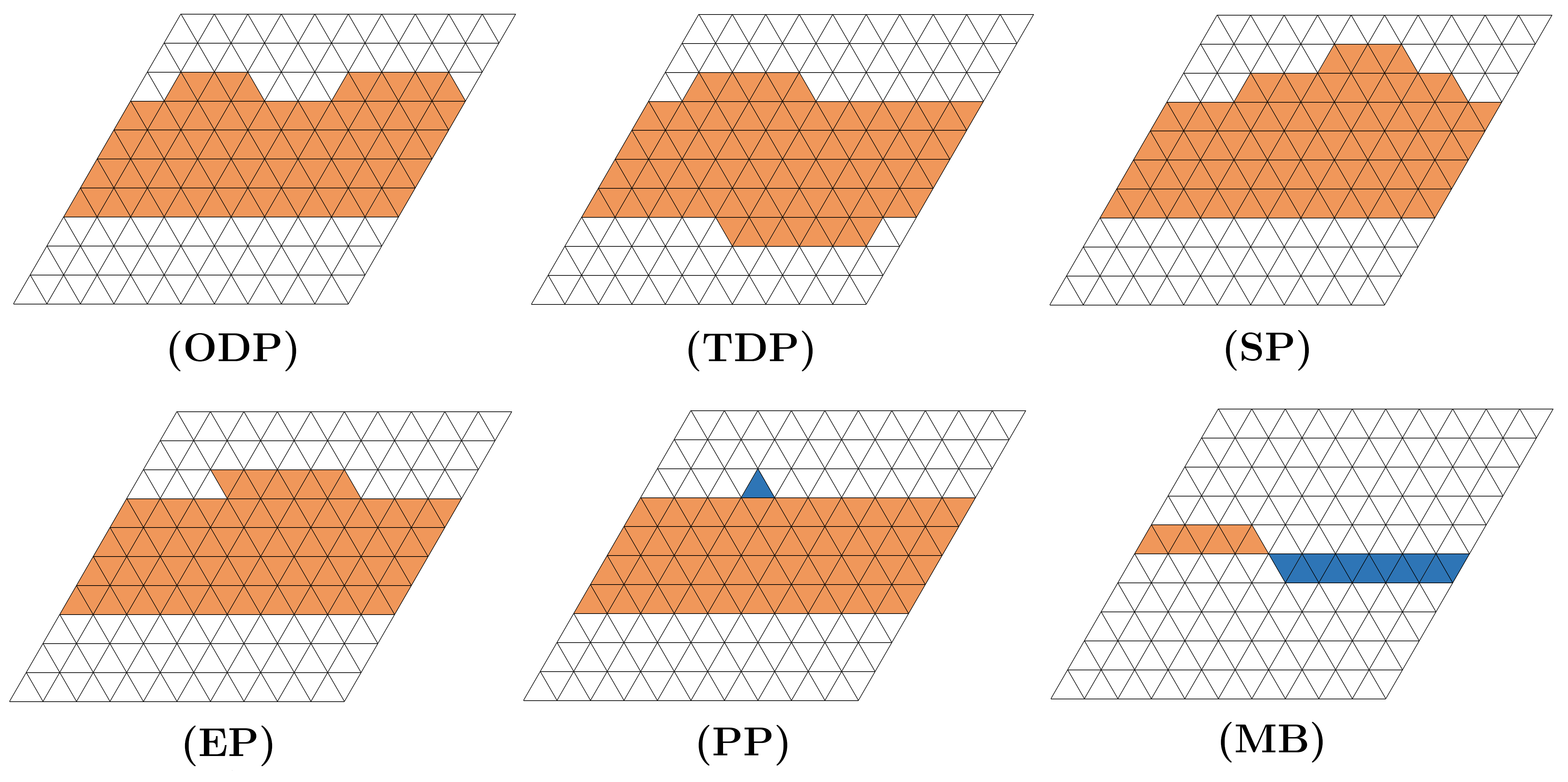}\caption{\label{fig64}Six types of cross-free configurations with energy $2L+2$
introduced in Definition \ref{d_2L+2}. We refer to Figures \ref{fig65}
and \ref{fig66} for more refined characterization of type \textbf{(MB)}. }
\end{figure}

Now, we are finally ready to characterize cross-free configurations
with energy $2L+2$. 
\begin{prop}
\label{p_E=00003D2L+2}Suppose that a cross-free configuration $\sigma\in\mathcal{X}$
satisfies $H(\sigma)=2L+2$. Then, $\sigma$ is of one of the six
types introduced in Definition \ref{d_2L+2}.
\end{prop}

\begin{proof}
By Lemma \ref{l_Elb}, $\sigma$ has at least $L-2$ bridges. Since
$\sigma$ is cross-free, these bridges are of the same direction,
say horizontal. Then, as in the proof of Proposition \ref{p_E=00003D2L+1},
we can observe that the number of horizontal bridges cannot be $L$
and thus the number of horizontal bridges should be either $L-1$
or $L-2$.\medskip{}

\noindent \textbf{(Case 1: $\sigma$ has $L-1$ horizontal bridges)}
If there is no vertical or diagonal semibridge, we must have $\Delta H_{\scal{v}_{\ell}}(\sigma),\,\Delta H_{\scal{d}_{\ell}}(\sigma)\ge3$
for all $\ell\in\mathbb{T}_{L}$ and therefore by \eqref{e_Edec},
we get $H(\sigma)\ge3L$ which yields a contradiction. Hence, there
exists at least one vertical or diagonal semibridge and thus by Lemma
\ref{l_crosschar}, there exist $a,\,b\in\Omega$ such that all the
horizontal bridges are $a$- or $b$-bridges. Let us define $P_{a}$
and $P_{b}$ as in Lemma \ref{l_crosschar} and let $\{\ell_{0}\}=\mathbb{T}_{L}\setminus(P_{a}\cup P_{b})$.
If $P_{a}=\emptyset$ or $P_{b}=\emptyset$, then $\sigma$ is of
type \textbf{(MB)} by definition. Now, we assume that $P_{a},\,P_{b}\ne\emptyset$.

\medskip{}

\noindent \textit{Case 1-1: The strip $\scal{h}_{\ell_{0}}$ contains
a triangle with spin which is not $a$ or $b$.} If there are two
or more such triangles, then there are at least three vertical or
diagonal strips containing these triangles with energy at least $3$.
Since all the other vertical and diagonal strips have energy at least
$2$, we can conclude from \eqref{e_Edec} that 
\[
H(\sigma)\ge\frac{1}{2}\big[\,2+3\times3+(2L-3)\times2\,\big]>2L+2
\]
which yields a contradiction. Therefore, the strip $\scal{h}_{\ell_{0}}$
contains \textit{exactly one triangle} with spin which is not $a$
or $b$. The vertical and diagonal strips containing this triangle
have energy at least $3$. Thus, if the strip $\scal{h}_{\ell_{0}}$
has energy at least $3$, we similarly get a contradiction since we
should have 
\[
H(\sigma)\ge\frac{1}{2}\big[\,3+2\times3+(2L-2)\times2\,\big]>2L+2\;.
\]
Therefore, the strip $\scal{h}_{\ell_{0}}$ has energy $2$. This
implies that all the triangles in this strip other than the one with
spin $c$ have the same spin, which is either $a$ or $b$. Hence,
$\sigma$ is of type \textbf{(PP)}.

\medskip{}

\noindent \textit{Case 1-2: The strip $\scal{h}_{\ell_{0}}$ consists
of spins $a$ and $b$ only.} By \eqref{e_Edec}, the energy of this
strip is at most $4$, and hence is either $2$ or $4$ (since it
cannot be an odd integer). If the energy of this strip is $2$, i.e.,
it is an $\{a,\,b\}$-semibridge by Lemma \ref{l_strip}, we can check
with the argument given in the proof of Proposition \ref{p_E=00003D2L+1}
based on Figure \ref{fig64} that the only possible form of configuration
$\sigma$ is of type \textbf{(EP)}. On the other hand, if the energy
of this strip is $4$, then in view of \eqref{e_Edec}, all the vertical
and diagonal bridges must have energy $2$ and thus must be semibridges.
Since this strip $\scal{h}_{\ell_{0}}$ of energy $4$ is divided
into four connected components where two of them are of spin $a$
and the remaining two are of spin $b$, by the same argument given
in Proposition \ref{p_E=00003D2L+1} based on Figure \ref{fig64},
we can readily check that $\sigma$ is of type \textbf{(ODP)}.

\medskip{}

\noindent \textbf{(Case 2: $\sigma$ has $L-2$ horizontal bridges)
}By the second statement of Lemma \ref{l_Elb}, we can still apply
Lemma \ref{l_crosschar}, and we can follow the same argument with
\textbf{(Case 1)} above to handle the case where $P_{a}$ or $P_{b}$
is empty. Hence, let us suppose that $P_{a},\,P_{b}\ne\emptyset$
and write $\mathbb{T}_{L}\setminus(P_{a}\cup P_{b})=\{\ell_{1},\,\ell_{2}\}$.
By (2) of Lemma \ref{l_crosschar}, we have $P_{a},\,P_{b}\in\mathfrak{S}_{L}$
and hence we can assume without loss of generality that $P_{a}=\llbracket1,\,m\rrbracket$
so that 
\[
\{\ell_{1},\,\ell_{2}\}\in\big\{\,\{L,\,m+1\},\,\{m+1,\,m+2\},\,\{L-1,\,L\}\,\big\}\;.
\]
Note from (2-a) of Lemma \ref{l_crosschar} that 
\begin{equation}
\text{all the non-bridge strips of }\sigma\text{ are }\{a,\,b\}\text{-semibridges}.\label{eq:CONS}
\end{equation}
If $\{\ell_{1},\,\ell_{2}\}=\{L,\,m+1\}$, by the same argument with
Proposition \ref{p_E=00003D2L+1}, strips $\scal{h}_{\ell_{1}}$ and
$\scal{h}_{\ell_{2}}$ should be aligned as in the middle one of Figure
\ref{fig64} in order to achieve \eqref{eq:CONS}, and we can conclude
that $\sigma$ is of type\textbf{ (TDP)}. A similar argument indicates
that if $\{\ell_{1},\,\ell_{2}\}=\{m+1,\,m+2\}$ or $\{L-1,\,L\}$,
the configuration $\sigma$ should be of type \textbf{(SP)} to fulfill
\eqref{eq:CONS}. 

Hence, we demonstrated that for any cases, $\sigma$ is one of the
six types given in Definition \ref{d_2L+2}.
\end{proof}
\begin{rem}
\label{r_E=00003D2L+2}A careful reading of the proof of the previous
proposition reveals that, if $\sigma$ is of type \textbf{(MB) }then
it has either $L-1$ or $L-2$ parallel bridges.
\end{rem}

In the next lemmas, we investigate more on the configurations of type
\textbf{(MB)}, since the definition of this type is vague and thus
more detailed understanding is crucially required to analyze the energy
landscape of $\mathcal{N}$-neighborhoods of the ground states. In
the analyses carried out below, we will omit elementary details in
the characterization of possible forms, since it always reduces to
a small number of subcases that should be tediously checked case by
case.
\begin{lem}
\label{l_MB}Suppose that $\sigma\in\mathcal{X}$ is of type \textbf{\textup{(MB)}}
with parallel bridges of spin $a\in\Omega$. Then, exactly one of
$(\star)$ and $(\star\star)$ given below holds.
\begin{enumerate}
\item[$(\star)$]  There exists a $(2L+2)$-path $(\omega_{n})_{n=0}^{N}$ from $\sigma$
to $\mathbf{a}$ so that $N\le4L$ and each configuration $\omega_{n}$
has at least $L-2$ $a$-bridges. 
\item[$(\star\star)$]  The configuration $\sigma$ is isolated in the sense that $\widehat{\mathcal{N}}(\sigma)=\{\sigma\}$.
\end{enumerate}
\end{lem}

\begin{proof}
It is immediate that $(\star)$ and $(\star\star)$ cannot hold simultaneously.
Hence, it suffices to prove that $\sigma$ satisfies $(\star)$ or
$(\star\star)$. Without loss of generality, we assume that the parallel
$a$-bridges are horizontal, and define $P_{a}$ as in \eqref{e_Pc}
so that we have $|P_{a}|=L-1$ or $L-2$ by Remark \ref{r_E=00003D2L+2}.
\textbf{\medskip{}
}

\noindent \textbf{(Case 1: $|P_{a}|=L-1$)} Without loss of generality,
write $\mathbb{T}_{L}\setminus P_{a}=\{1\}$.

\noindent If there are two adjacent triangles in $\scal{h}_{1}$ with
spin $a$, we can find an $a$-cross and therefore we get a contradiction
to the fact that $\sigma$ is cross-free. Hence, the strip $\scal{h}_{1}$
cannot have consecutive triangles with spin $a$. Moreover, since
all the vertical and diagonal strips have energy at least $2$, by
\eqref{e_Edec}, we have $\Delta H_{\scal{h}_{1}}(\sigma)\le4$. From
this, we can readily deduce that $\sigma$ should be of one of the
four types \textbf{(MB1)}-\textbf{(MB4)} as in Figure \ref{fig65}. 

\begin{figure}[h]
\includegraphics[width=15cm]{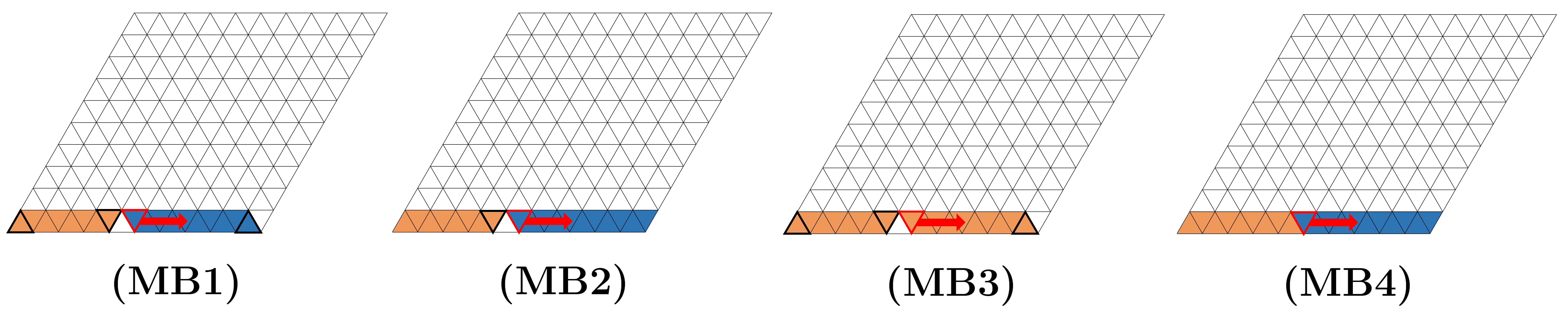}\caption{\label{fig65}Types \textbf{(MB1)}-\textbf{(MB4)}: We can update the
triangles according to the indicated arrow starting from the one with
red bold boundary to reach $\mathbf{a}$. One can change the starting
triangle to the ones with black bold boundary and then modify the
order of updates.}
\end{figure}

\noindent We now demonstrate that $(\star)$ holds for all these types.
For types \textbf{(MB1)}-\textbf{(MB3)}, we select any triangle adjacent
to a triangle with spin $a$, and for type \textbf{(MB4)}, we select
a triangle adjacent to a triangle with different spin. Then, we update
the spins in $\scal{h}_{1}$ to $a$ successively from the selected
triangle to obtain the configuration $\mathbf{a}$ (cf. Figure \ref{fig65}).
This procedure provides a $(2L+2)$-path connecting $\sigma$ and
$\mathbf{a}$ of length at most $2L$. It is immediate that all the
configurations visited by this path have at least $2L-1$ $a$-bridges
and hence we can verify the condition $(\star)$ for these types.

\medskip{}

\noindent \textbf{(Case 2: $|P_{a}|=L-2$)} Write $\mathbb{T}_{L}\setminus P_{a}=\{\ell_{1},\,\ell_{2}\}$.
By the second statement of Lemma \ref{l_Elb}, all the strips which
are not bridges must be semibridges. Moreover, if $\scal{h}_{\ell_{i}}$
for some $i\in\{1,\,2\}$ is a $\{b,\,c\}$-semibridge for some $b,\,c\in\Omega\setminus\{a\}$,
then we can find a vertical or diagonal strip which is not a semibridge
(the one which contains the adjacent triangles of spins $b$ and $c$
in $\scal{h}_{\ell_{i}}$) and thus we obtain a contradiction. Therefore,
there exist $b_{1}\ne a$ and $b_{2}\ne a$ so that $\scal{h}_{\ell_{i}}$
is an $\{a,\,b_{i}\}$-semibridge for each $i\in\{1,\,2\}$. We denote
by \textit{$b_{i}$-protuberance in $\scal{h}_{\ell_{i}}$} the set
of triangles in $\scal{h}_{\ell_{i}}$ which have spin $b_{i}$. \medskip{}

\noindent \textbf{(Claim)}\textit{ Two strips $\scal{h}_{\ell_{1}}$
and $\scal{h}_{\ell_{2}}$ are adjacent.}\medskip{}

\noindent To prove this claim, suppose the contrary that $\scal{h}_{\ell_{1}}$
and $\scal{h}_{\ell_{2}}$ are not adjacent. We denote by $m_{i}\in\llbracket1,\,2L-1\rrbracket$
the number of spins $b_{i}$ in $\scal{h}_{\ell_{i}}$ for $i=1,\,2$.
Then since each $b_{i}$-protuberance in $\scal{h}_{\ell_{i}}$ has
perimeter $m_{i}+2$, we can deduce from \eqref{e_dualHam} that $H(\sigma)=(m_{1}+2)+(m_{2}+2)$.
Since we assumed that $H(\sigma)=2L+2$, we get
\begin{equation}
m_{1}+m_{2}=2L-2\;.\label{e_m1m2}
\end{equation}
Let us first assume that $m_{2}$ is even as in the left figure below
(where $\ell_{2}$ is assumed to be $5$ and spin $b_{i}$ is denoted
by orange).
\begin{center}
\includegraphics[width=11cm]{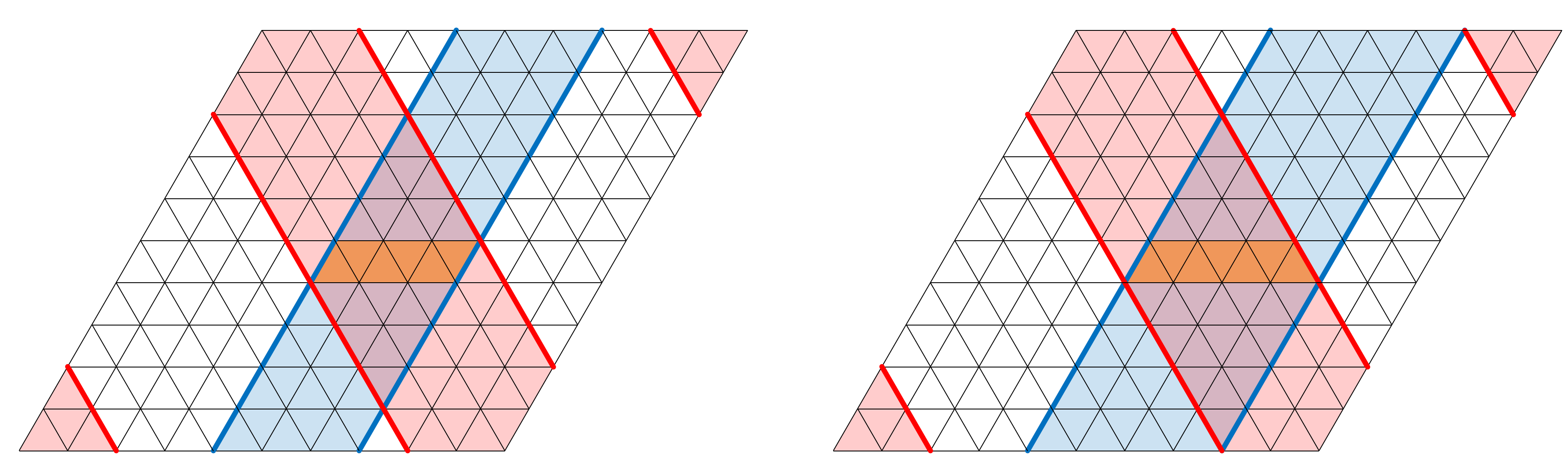}
\par\end{center}

\noindent Since the vertical strips contained in blue region must
be $\{a,\,b_{2}\}$-semibridges, set $A$ of triangles in strip $\scal{h}_{\ell_{1}}$
contained in these blue region should be of spin $a$. By the same
reasoning, the set $B$ of triangles in strip $\scal{h}_{\ell_{1}}$
contained in the red region should be of spin $a$. Since $|A|=m_{2}$,
$|B|=m_{2}+2$, and $|A\cap B|\le m_{2}-3$ provided that $\scal{h}_{\ell_{1}}$
and $\scal{h}_{\ell_{2}}$ are not adjacent, we get 
\[
m_{1}\le|\scal{h}_{\ell_{1}}\setminus(A\cup B)|=2L-|A|-|B|+|A\cap B|\le2L-m_{2}-(m_{2}+2)+(m_{2}-3)=2L-m_{2}-5\;.
\]
This contradicts \eqref{e_m1m2}. We can handle the case when $m_{2}$
is odd as in the right figure above in the same manner. For this case,
we have $|A|=|B|=m_{2}+1$ and $|A\cap B|\le m_{2}-4$, and we can
conclude $m_{1}\le2L-m_{2}-6$ to get a contradiction to \eqref{e_m1m2}.
Thus, the proof is completed.

Thanks to this claim, we can now assume without loss of generality
that $\ell_{1}=1$ and $\ell_{2}=2$. We then show that there are
nine possible types as in the following figure. 

\begin{figure}[h]
\includegraphics[width=12cm]{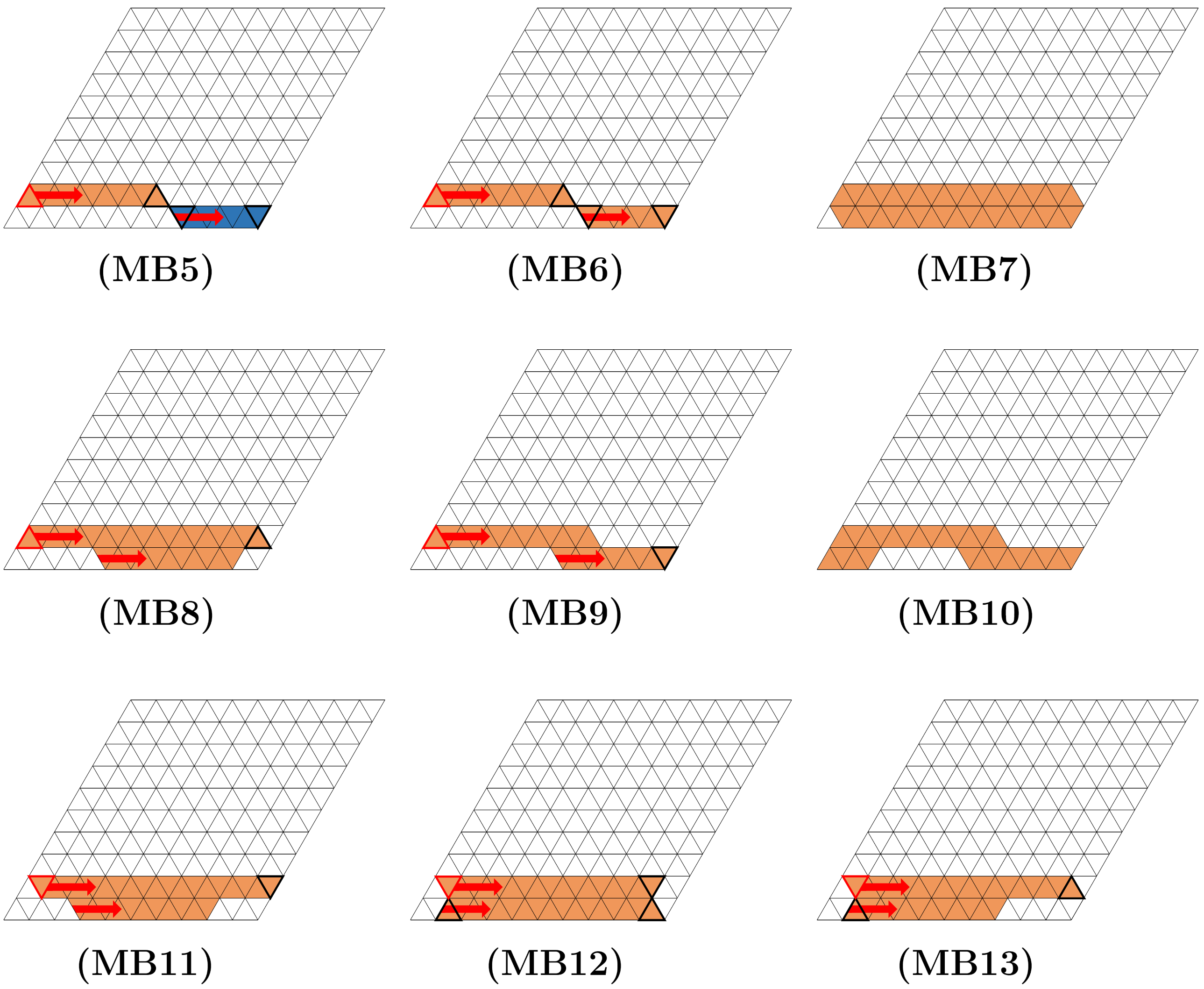}

\caption{\label{fig66}Types \textbf{(MB5)}-\textbf{(MB13)}: We refer to the
last part of the proof regarding the explanation of these figures.}
\end{figure}

\noindent To justify this classification, we first consider the case
when $b_{1}\neq b_{2}$. Then, the $b_{1}$-protuberance in $\scal{h}_{1}$
and the $b_{2}$-protuberance in $\scal{h}_{2}$ must not be adjacent
to each other, since otherwise there exists a vertical or diagonal
non-semibridge strip. Since $\sigma$ is a cross-free configuration,
we can readily conclude that $\sigma$ should be of type \textbf{(MB5)}.

Next, we consider the case $b_{1}=b_{2}=b$ and we assume without
loss of generality that the size of the $b$-protuberance in $\scal{h}_{2}$
is not smaller than that in $\scal{h}_{1}$. We can then divide the
analysis into three subcases according to the shape of the $b$-protuberance
in $\scal{h}_{2}$:
\begin{enumerate}
\item it has odd number of triangles and its lower side is longer than its
upper side,
\item it has odd number of triangles and its upper side is longer than its
lower side, or
\item it has even number of triangles.
\end{enumerate}
Without loss of generality we assume that the $b$-protuberance in
$\scal{h}_{2}$ is located at the leftmost part of the lattice as
in Figure \ref{fig66}. For case (1), we can observe that the protuberance
of $b$ in $\scal{h}_{1}$ also has odd number of triangles and its
upper side should be longer than its lower side, since otherwise there
will be a non-semibridge strip. According to five different types
of locations of this protuberance in the strip $\scal{h}_{1}$, we
get the types \textbf{(MB6)}-\textbf{(MB10)} illustrated in Figure
\ref{fig66}. For case (2), we can similarly observe that the protuberance
of $b$ in $\scal{h}_{1}$ also have odd number of triangles and it
should be aligned as in\textbf{ (MB11)} or \textbf{(MB12)}. (In \textbf{(MB12)},
$\scal{h}_{1}$ and $\scal{h}_{2}$ have the same size of $b$-protuberances.)
Finally, for case (3), the $b$-protuberance in $\scal{h}_{1}$ should
consist of even number of triangles and be aligned exactly as in \textbf{(MB13)}.
(In particular, it must be right-aligned.)

We have now fully characterized the configurations of type \textbf{(MB)},
and it only remains to investigate the path-connectivity of types
\textbf{(MB6)-(MB13)} to the configuration $\mathbf{a}$. We consider
three cases separately.
\begin{itemize}
\item \textbf{(MB10)}: Any update on this type of configuration increases
the energy. Thus, those of this type satisfy $(\star\star)$.
\item \textbf{(MB7)}: We first flip a spin $a$ in $\scal{h}_{2}$ to spin
$b$, so that we obtain a canonical configuration in $\mathcal{C}_{1}^{a,\,b}$
with protuberance size $2L-1$ (cf. Definition \ref{d_can}). Then,
we can follow a canonical path (cf. Remark \ref{r_canpath}) from
there to reach the configuration $\mathbf{a}$. The path associated
with these updates is a $(2L+2)$-path of length $4L$. Moreover,
$a$-bridges in $\scal{h}(\llbracket3,\,L\rrbracket)$ are conserved
along the path, and we can conclude that the configurations of this
type satisfy $(\star)$.
\item \textbf{(MB5)},\textbf{ (MB6)},\textbf{ (MB8)},\textbf{ (MB9)}, \textbf{(MB11)-(MB13)}:
We update spins $b$ to $a$ in the order indicated in Figure \ref{fig66}
to obtain the configuration $\mathbf{a}$. More precisely, for types
\textbf{(MB5)}, \textbf{(MB6)}, \textbf{(MB8)}, and \textbf{(MB9)},
we update the triangles according to the indicated arrow starting
from the one with red bold boundary to reach $\mathbf{a}$. For types
\textbf{(MB11)-(MB13)}, we first update the triangle with red bold
boundary, then update one of the triangles with black bold boundary,
and then update the remaining spins $b$ to $a$ according to the
arrow to reach $\mathbf{a}$. In all the aforementioned types, one
can select the starting triangle as the ones with black bold boundary.
We remark that for \textbf{(MB13)}, if there are same number of orange
triangles in $\scal{h}_{1}$ and $\scal{h}_{2}$, then the black triangle
at $\scal{h}_{2}$ is no longer available as a starting triangle.
For \textbf{(MB8)}, the red or black triangle might not be available
as a starting triangle if the $b$-protuberance in $\scal{h}_{1}$
is aligned to the right or the left. Then, as in the previous case,
we can readily observe that the path associated with these updates
satisfies all the requirements in $(\star)$, and thus the configurations
of these types satisfy $(\star)$. 
\end{itemize}
This completes the proof.
\end{proof}
We can deduce the following lemma from a careful inspection of the
proof of the previous lemma.
\begin{lem}
\label{l_MB2}Let $\sigma\in\mathcal{X}$ be a configuration of type\textbf{\textup{
(MB)}} except \textbf{\textup{(MB7)}} with parallel bridges of spin
$a\in\Omega$, and let $\zeta\in\mathcal{X}$ be a configuration satisfying
$\sigma\sim\zeta$ such that either $H(\zeta)\le2L+1$ or $\zeta$
has a cross\footnote{In fact, if $\zeta$ has a cross, then we can prove that $H(\zeta)\le2L+1$.}.
Then, there exists a $(2L+1)$-path of length less than $4L$ connecting
$\zeta$ and $\mathbf{a}$. In particular, $\zeta\in\mathcal{N}(\mathbf{a})$.
\end{lem}

\begin{proof}
We can notice from Figures \ref{fig65} and \ref{fig66} that such
a $\zeta$ exists only when $\sigma$ is of type \textbf{(MB1)-(MB3)},\textbf{
(MB5)},\textbf{ (MB6)},\textbf{ (MB8)},\textbf{ (MB9)},\textbf{ }or
\textbf{(MB11)-(MB13)}, and moreover $\zeta$ is obtained from $\sigma$
by one of the following ways: 
\begin{enumerate}
\item Updating the spin at a triangle highlighted by (either black or red)
bold boundary in Figures \ref{fig65} and \ref{fig66} into $a$. 
\item For type \textbf{(MB3)}, $\zeta$ can be obtained by flipping spin
$a$ at the strip $\scal{h}_{1}$ to $b$. For this case, $\zeta$
is a canonical configuration with $2L-1$ triangles of $b$ at a strip.
\item For type\textbf{ (MB8)} such that the strip $\scal{h}_{1}$ contains
only one triangle with spin $b$, the configuration $\zeta$ can additionally
obtained by flipping that spin $b$ to spin $a$. We note that $\zeta$
is of the same type as in case (2) above. 
\end{enumerate}
For case (1), if $\zeta$ is obtained from $\sigma$ by flipping the
spin at a triangle with red boundary, then we can continue to update
according to the order indicated in the figure to reach $\mathbf{a}$.
Then, the path corresponding to the sequence of updates provides a
$(2L+1)$-path of length less than $4L$ connecting $\zeta$ and $\mathbf{a}$.
The case when $\zeta$ is obtained from $\sigma$ by flipping a spin
at a triangle with black boundary can be handled in a similar way.
For cases (2) and (3), since $\zeta$ is a canonical configuration,
it is connected to $\mathbf{a}$ via a canonical path (cf. Remark
\ref{r_canpath}) which is a $(2L+1)$-path of length $2L-1$. 
\end{proof}
\begin{rem}
If we consider the Ising case, then the type \textbf{(PP) }is unavailable
and also the analysis of type \textbf{(MB) }becomes much simpler.
\end{rem}

As a byproduct of the characterization carried out in the current
section, we derive a rough bound on the number of cross-free configurations
which will be required in later computations. For a sequence $(a_{L})_{L=1}^{\infty}$,
we write $a_{L}=O(f(L))$ if there exists a constant $C>0$ such that
$|a_{L}|\le Cf(L)$ for all $L$.
\begin{lem}
\label{l_cfcount}The number of cross-free configuration with energy
less than or equal to $2L+2$ is $O(L^{6})$. 
\end{lem}

\begin{proof}
Since we get a full characterization of cross-free configurations
in Propositions \ref{p_E<=00003D2L}, \ref{p_E=00003D2L+1}, and \ref{p_E=00003D2L+2},
the conclusion of the lemma follows directly from elementary counting.
\end{proof}

\subsection{\label{sec63}Dead-ends}

In this subsection, we summarize the geometry of the energy landscape
near the canonical configurations. As a consequence, we are able to
get the full characterization of dead-ends (cf. Definition \ref{d_dead})
the process encounters in the course of transitions between ground
states.

We first introduce some notation.
\begin{itemize}
\item For a configuration $\sigma\in\mathcal{X}$ and $c\in\Omega$, we
say that a subset $C$ of $\Lambda^{*}$ is a $c$-cluster if it is
a monochromatic cluster consisting of spin $c$.
\item The\textit{ boundary} of a set $A\subseteq\Lambda^{*}$ refers to
the collection of triangles in $\Lambda^{*}\setminus A$ adjacent
to triangles in $A$. An example is given by the following figure;
if $A$ is the collection of orange triangles, the blue triangles
are the boundary of $A$.
\begin{center}
\includegraphics[width=6cm]{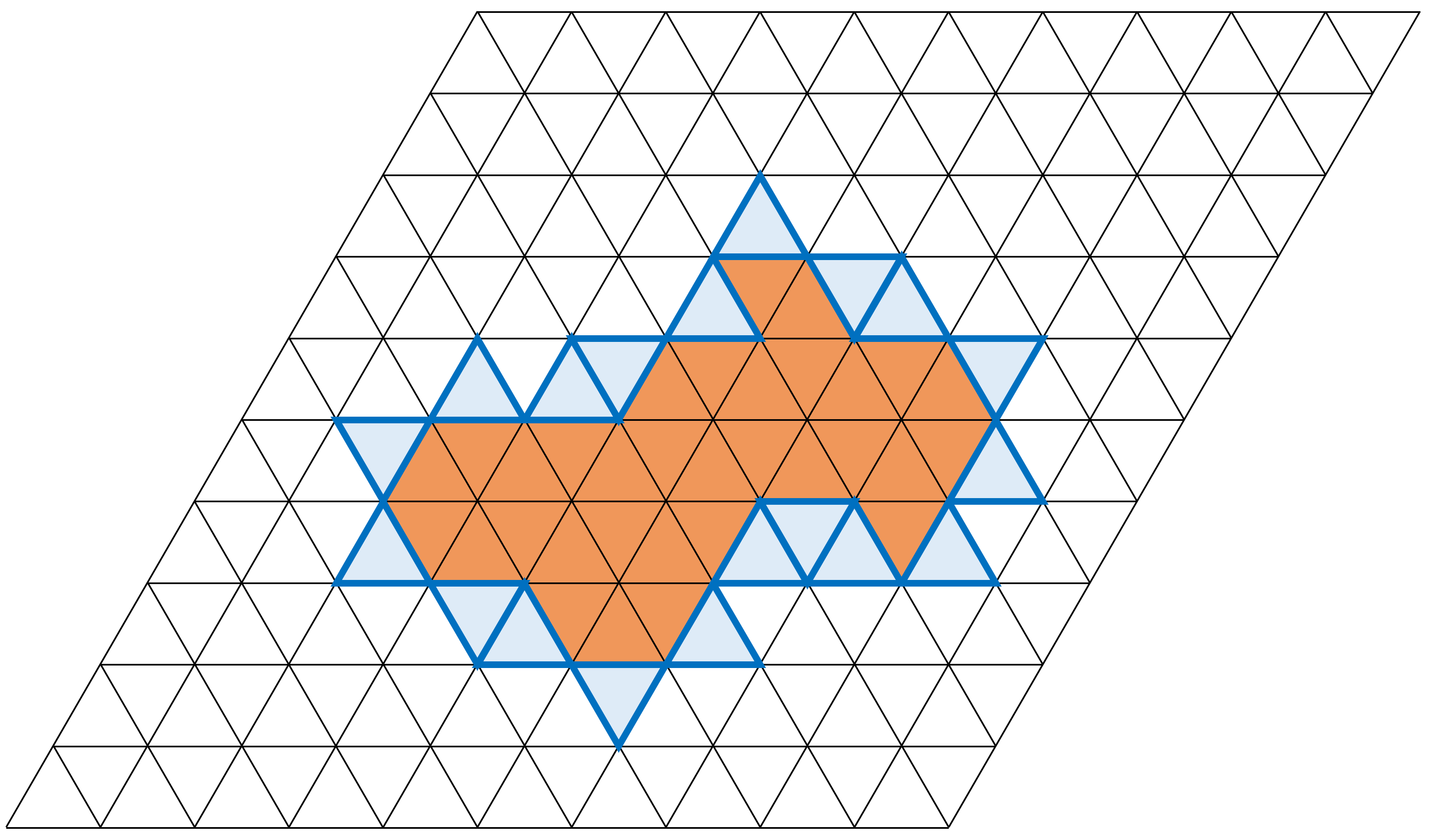}
\par\end{center}
\item For a configuration $\sigma\in\mathcal{X}$, we say that a triangle
$x\in\Lambda^{*}$ is a boundary triangle of $\sigma$ if $x$ belongs
to a boundary of a certain cluster of $\sigma$. Since a non-boundary
triangle $x$ of $\sigma$ has the same spin with its three adjacent
triangles, we can observe that 
\begin{equation}
\text{flipping the spin at a non-boundary triangle }x\text{ of }\sigma\text{ increases the energy by }3\;,\label{e_nonbdry}
\end{equation}
while flipping the spin at a boundary triangle increases the energy
by at most $2$ (or decreases the energy up to $3$).
\item Let $\sigma\in\mathcal{X}$ be a configuration satisfying $H(\sigma)\le2L+2$.
If $\zeta\in\mathcal{X}$ is obtained by a flip of spin of $\sigma$
(i.e., $\sigma\sim\zeta$) and $H(\zeta)\le2L+2$, we write $\sigma\approx\zeta$
and the corresponding flip is called a \textit{good flip}.
\end{itemize}
We now characterize all the configurations connected to a canonical
configuration $\sigma$ and having energy at most $2L+2$. We decompose
our investigation into three cases: $\sigma\in\mathcal{R}_{n}^{a,\,b}$
(Lemma \ref{l_reg}), $\sigma\in\mathcal{C}_{n,\,\scal{o}}^{a,\,b}$
(Lemma \ref{l_odd}), and $\sigma\in\mathcal{C}_{n,\,\scal{e}}^{a,\,b}$
(Lemma \ref{l_even}). To that end, we define the following collections
for $a,\,b\in\Omega$.
\begin{itemize}
\item $\mathcal{P}_{n}^{a,\,b}$, $n\in\llbracket2,\,L-2\rrbracket$: the
collection of configurations of type \textbf{(PP)} which can be obtained
by a good flip of a configuration in $\mathcal{R}_{n}^{a,\,b}$.
\item $\mathcal{Q}_{n}^{a,\,b}$, $n\in\llbracket1,\,L-2\rrbracket$: the
collection of configurations of type \textbf{(ODP)},\textbf{ (TDP)},\textbf{
}or\textbf{ (SP) }which can be obtained by a good flip of a configuration
in $\mathcal{C}_{n,\,\scal{o}}^{a,\,b}$.
\item $\widehat{\mathcal{R}}_{n}^{a,\,b}$, $n\in\llbracket2,\,L-2\rrbracket$:
the collection of configurations $\zeta$ such that
\[
\text{either }\zeta\in\mathcal{C}_{n,\,\scal{o}}^{a,\,b}\text{ with }|\mathfrak{p}^{a,\,b}(\zeta)|=1\;,\text{ or }\zeta\in\mathcal{C}_{n-1,\,\scal{o}}^{a,\,b}\text{ with }|\mathfrak{p}^{a,\,b}(\zeta)|=2L-1\;.
\]
Namely, $\widehat{\mathcal{R}}_{n}^{a,\,b}$ is the collection of
canonical configurations obtained by a good flip of a regular configuration
in $\mathcal{R}_{n}^{a,\,b}$.
\end{itemize}
We now start the characterization. We fix $a,\,b\in\Omega$ in the
remainder of the current section.
\begin{lem}
\label{l_reg}Suppose that $\sigma\in\mathcal{R}_{n}^{a,\,b}$ with
$n\in\llbracket2,\,L-2\rrbracket$ and $\zeta\in\mathcal{X}$ satisfies
$\sigma\approx\zeta$. Then, we have either $\zeta\in\widehat{\mathcal{R}}_{n}^{a,\,b}$
or $\zeta\in\mathcal{P}_{n}^{a,\,b}$. In particular, we have $\widehat{\mathcal{R}}_{n}^{a,\,b}=\mathcal{N}(\mathcal{R}_{n}^{a,\,b})\setminus\mathcal{R}_{n}^{a,\,b}$.
\end{lem}

\begin{proof}
Let us fix $\sigma\in\mathcal{R}_{n}^{a,\,b}$. Since $H(\sigma)=2L$
and $H(\zeta)\le2L+2$, by \eqref{e_nonbdry}, the configuration $\zeta$
is obtained from $\sigma$ by flipping a boundary triangle. First,
we assume that we flip a spin at a boundary triangle of the $b$-cluster
of $\sigma$ (which has spin $a$) to $c$ to get $\zeta$. As one
can check from the figure below, we get $\zeta\in\widehat{\mathcal{R}}_{n}^{a,\,b}$
(in particular, $\zeta\in\mathcal{C}_{n,\,\scal{o}}^{a,\,b}$ with
$|\mathfrak{p}^{a,\,b}(\zeta)|=1$) or \textbf{$\zeta\in\mathcal{P}_{n}^{a,\,b}$
}if $c=a$ or $c\notin\{a,\,b\}$, respectively. 
\begin{center}
\includegraphics[width=14cm]{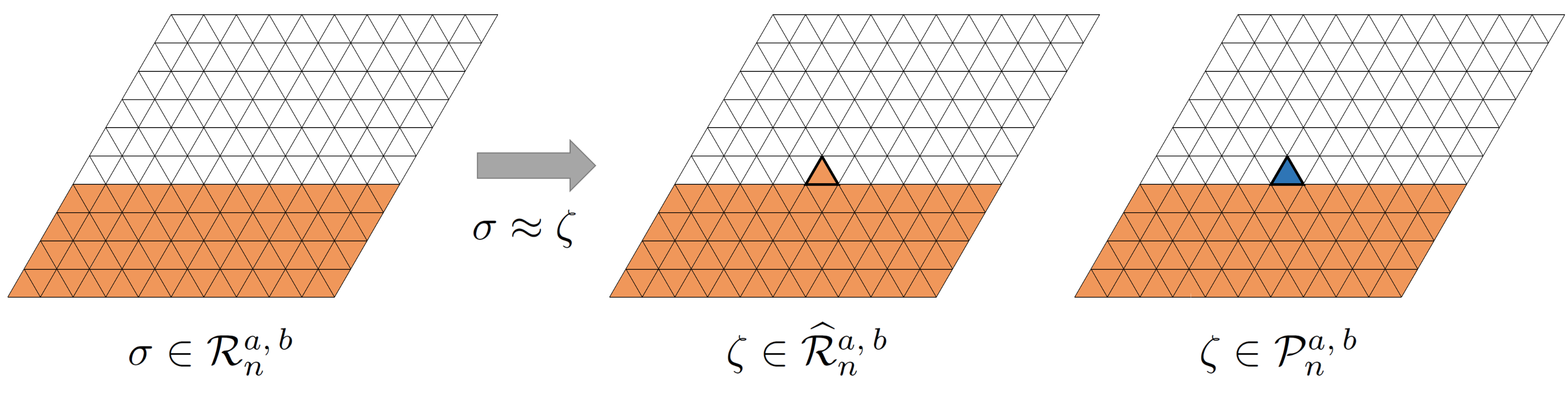}
\par\end{center}

\noindent The case when we flip a boundary triangle of the $a$-cluster
is identical to the previous case and we can conclude the proof of
the first statement. For the second statement, we first observe that
if $\xi\sim\zeta$ for some $\zeta\in\widehat{\mathcal{R}}_{n}^{a,\,b}$
and $H(\xi)<2L+2$, then we must have $\xi\in\mathcal{R}_{n}^{a,\,b}$.
Since the configuration of type \textbf{(PP)} has energy $2L+2$,
the second assertion of the lemma is direct from the first one. 
\end{proof}
Thanks to Lemma \ref{l_reg}, we will hereafter discard the notation
$\widehat{\mathcal{R}}_{n}^{a,\,b}$ and use $\mathcal{N}(\mathcal{R}_{n}^{a,\,b})\setminus\mathcal{R}_{n}^{a,\,b}$
instead.
\begin{lem}
\label{l_odd}Suppose that $\sigma\in\mathcal{C}_{n,\,\scal{o}}^{a,\,b}$
with $n\in\llbracket2,\,L-2\rrbracket$ and $\zeta\in\mathcal{X}$
satisfies $\sigma\approx\zeta$. Then, we have either $\zeta\in\mathcal{R}_{n}^{a,\,b}\cup\mathcal{R}_{n+1}^{a,\,b}\cup\mathcal{C}_{n,\,\scal{e}}^{a,\,b}$,
$\zeta\in\mathcal{P}_{n}^{a,\,b}\cup\mathcal{P}_{n+1}^{a,\,b}$, or
$\zeta\in\mathcal{Q}_{n}^{a,\,b}$. In particular, if $3\le|\mathfrak{p}^{a,\,b}(\sigma)|\le2L-3$,
we have either $\zeta\in\mathcal{C}_{n,\,\scal{e}}^{a,\,b}$ or $\zeta\in\mathcal{Q}_{n}^{a,\,b}$.
\end{lem}

\begin{proof}
We fix $\sigma\in\mathcal{C}_{n,\,\scal{o}}^{a,\,b}$ and first consider
the case $|\mathfrak{p}^{a,\,b}(\sigma)|=1$. By \eqref{e_nonbdry},
we can notice that we have to flip a boundary triangle of $\sigma$
to get $\zeta$. We can group the boundary triangles of $\sigma$
into seven types as in Figure \ref{fig67}-(left).

\begin{figure}[h]
\includegraphics[width=15cm]{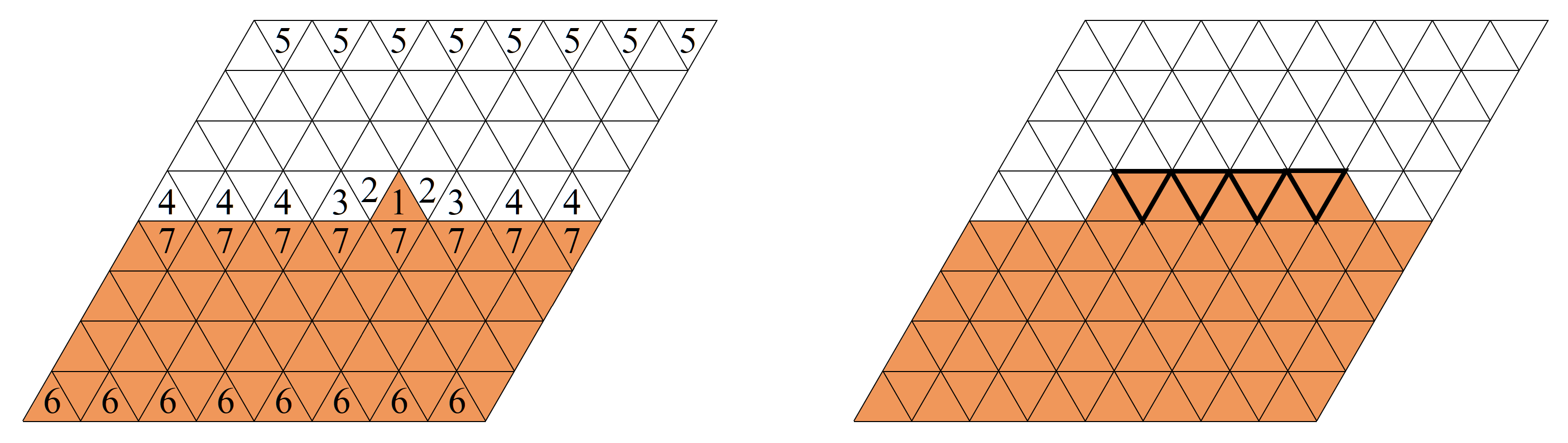}

\caption{\label{fig67}Good flip of a configuration in $\mathcal{C}_{n,\,\scal{o}}^{a,\,b}$.
\textbf{(Left)} $|\mathfrak{p}^{a,\,b}(\sigma)|=1$ or $2L-1$.\textbf{
(Right)} $3\le|\mathfrak{p}^{a,\,b}(\sigma)|\le2L-3$.}
\end{figure}

\noindent If we flip the triangle of type $1$, we get $\zeta\in\mathcal{R}_{n}^{a,\,b}$
or $\zeta\in\mathcal{P}_{n}^{a,\,b}$. If a flip of the spin of a
triangle in types $2$-$7$ is a good flip, the spin must be flipped
to either $a$ or $b$. Hence, we get a configuration in $\mathcal{C}_{n,\,\scal{e}}^{a,\,b}$
(resp. in $\mathcal{Q}_{n}^{a,\,b}$) if we flip the spin at a triangle
of types $2$ or $3$ (resp. types $4$-$7$). The case $|\mathfrak{p}^{a,\,b}(\sigma)|=2L-1$
can be handled in the exact same way with this case and we get either
$\zeta\in\mathcal{R}_{n+1}^{a,\,b}$, $\zeta\in\mathcal{P}_{n+1}^{a,\,b}$,
$\zeta\in\mathcal{C}_{n,\,\scal{e}}^{a,\,b}$, or $\zeta\in\mathcal{Q}_{n}^{a,\,b}$.

Next, we consider the case $3\le|\mathfrak{p}^{a,\,b}(\sigma)|\le2L-3$.
The proof is similar to the previous case. In particular, the flip
of triangles of types $2$-$7$ are of the identical nature. The only
difference appears in the flip of a triangle of type $1$, i.e., a
triangle in the protuberance of spin $b$. For this case, we have
to flip triangle denoted by bold black boundary in Figure \ref{fig67}-(right)
to get a configuration belonging to $\mathcal{C}_{n,\,\scal{e}}^{a,\,b}$
or \textbf{$\mathcal{Q}_{n}^{a,\,b}$}.
\end{proof}
\begin{lem}
\label{l_even}Suppose that $\sigma\in\mathcal{C}_{n,\,\scal{e}}^{a,\,b}$
with $n\in\llbracket2,\,L-2\rrbracket$ and $\zeta\in\mathcal{X}$
satisfies $\sigma\approx\zeta$. Then, $\zeta\in\mathcal{C}_{n,\,\scal{o}}^{a,\,b}$.
\end{lem}

\begin{proof}
There are essentially two cases (depending on whether the protuberance
is connected or not) to be considered as in the figure below.
\begin{center}
\includegraphics[width=12cm]{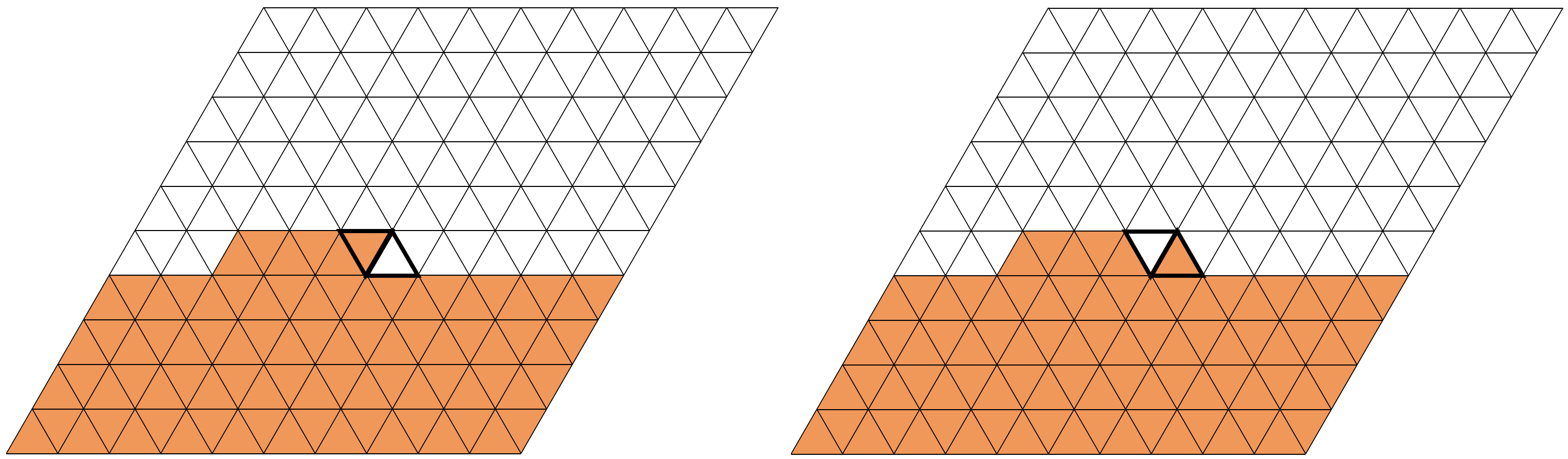}
\par\end{center}

\noindent Since the configuration $\sigma$ already has energy $2L+2$,
the good flip must not increase the energy, and therefore should flip
the spin at one of the triangles with bold black boundary in the figure
above either from $a$ to $b$ or from $b$ to $a$. Since the configuration
obtained from this any of such flips belongs to $\mathcal{C}_{n,\,\scal{o}}^{a,\,b}$,
the proof is completed. 
\end{proof}
The non-canonical configurations appeared in the preceding three lemmas
are defined now as the dead-ends.
\begin{defn}[Dead-ends]
\label{d_dead} For $a,\,b\in\Omega$, define 
\[
\mathcal{D}^{a,\,b}=\Big[\,\bigcup_{n=2}^{L-2}\mathcal{P}_{n}^{a,\,b}\,\Big]\cup\Big[\,\bigcup_{n=2}^{L-3}\mathcal{Q}_{n}^{a,\,b}\,\Big]\;.
\]
It is clear that $\sigma\in\mathcal{D}^{a,\,b}$ implies $H(\sigma)=2L+2$.
We say that a configuration $\sigma$ belonging to $\mathcal{D}^{a,\,b}$
is a \textit{dead-end} between $\mathbf{a}$ and $\mathbf{b}$. For
each proper partition $(A,\,B)$, we write
\[
\mathcal{D}^{A,\,B}=\bigcup_{a'\in A}\bigcup_{b'\in B}\mathcal{D}^{a',\,b'}\;.
\]
\end{defn}

Next, we perform further investigations on the dead-end configurations,
after which we can explain why these configurations are called dead-ends
(cf. Remark \ref{r_dead}).
\begin{lem}
\label{l_PP}Suppose that $\sigma\in\mathcal{P}_{n}^{a,\,b}$ with
$n\in\llbracket2,\,L-2\rrbracket$ and $\zeta\in\mathcal{X}$ satisfies
$\sigma\approx\zeta$. Then, we have either $\zeta\in\mathcal{N}(\mathcal{R}_{n}^{a,\,b})$
(two choices) or $\zeta\in\mathcal{P}_{n}^{a,\,b}$ ($q-3$ choices).
\end{lem}

\begin{proof}
A good flip of a configuration $\sigma\in\mathcal{P}_{n}^{a,\,b}$
must flip the spin at the peculiar protuberance. By flipping this
spin to $a$ or $b$, we get a configuration in $\mathcal{N}(\mathcal{R}_{n}^{a,\,b})$.
Otherwise, we get a configuration in $\mathcal{P}_{n}^{a,\,b}$ and
we are done.
\end{proof}
\begin{lem}
\label{l_QQ}Suppose that $\sigma\in\mathcal{Q}_{n}^{a,\,b}$ with
$n\in\llbracket2,\,L-3\rrbracket$ is obtained from $\xi\in\mathcal{C}_{n,\,\scal{o}}^{a,\,b}$
by flipping a spin. Suppose that $\zeta\in\mathcal{X}$ satisfies
$\sigma\approx\zeta$.
\begin{enumerate}
\item If $|\mathfrak{p}^{a,\,b}(\xi)|\neq1,\,2L-1$, we have $\zeta=\xi$.
\item If $|\mathfrak{p}^{a,\,b}(\xi)|=1$ (so that $\xi\in\mathcal{N}(\mathcal{R}_{n}^{a,\,b})$),
there are exactly two possible configurations for $\zeta$, which
are both in $\mathcal{N}(\mathcal{R}_{n}^{a,\,b})\setminus\mathcal{R}_{n}^{a,\,b}$.
\item If $|\mathfrak{p}^{a,\,b}(\xi)|=2L-1$ (so that $\xi\in\mathcal{N}(\mathcal{R}_{n+1}^{a,\,b})$),
there are exactly two possible configurations for $\zeta$, which
are both in $\mathcal{N}(\mathcal{R}_{n+1}^{a,\,b})\setminus\mathcal{R}_{n+1}^{a,\,b}$.
\end{enumerate}
\end{lem}

\begin{proof}
If $|\mathfrak{p}^{a,\,b}(\xi)|\neq1,\,2L-1$. we can notice from
the figure below that $\sigma$ is obtained from $\xi$ by flip the
spin at one of the triangles with bold boundary either from $a$ to
$b$ or $b$ to $a$. 
\noindent \begin{center}
\includegraphics[width=4.7cm]{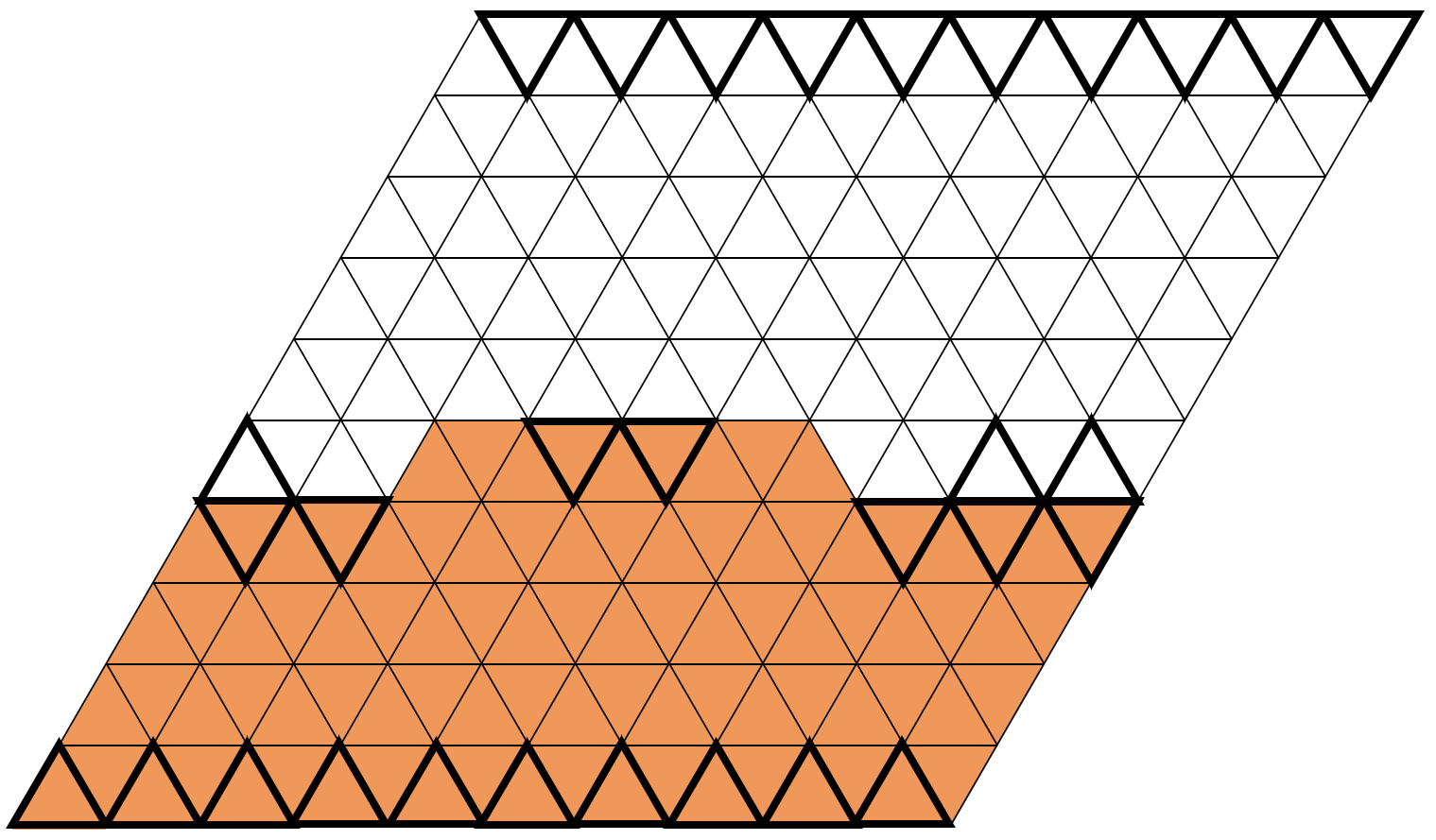}
\par\end{center}

\noindent Then, it is direct from that a good flip of $\sigma$ must
flip back this updated spin, since otherwise the energy will be further
increased to at least $2L+3$. Hence, we get $\zeta=\xi$.

Next, we consider the case $|\mathfrak{p}^{a,\,b}(\xi)|=1$. Then,
as in the figure below, $\sigma$ should be obtained by adding a protuberance
of spin $a$ or $b$ of size one to $\xi$, and there are four different
types. 
\noindent \begin{center}
\includegraphics[width=16cm]{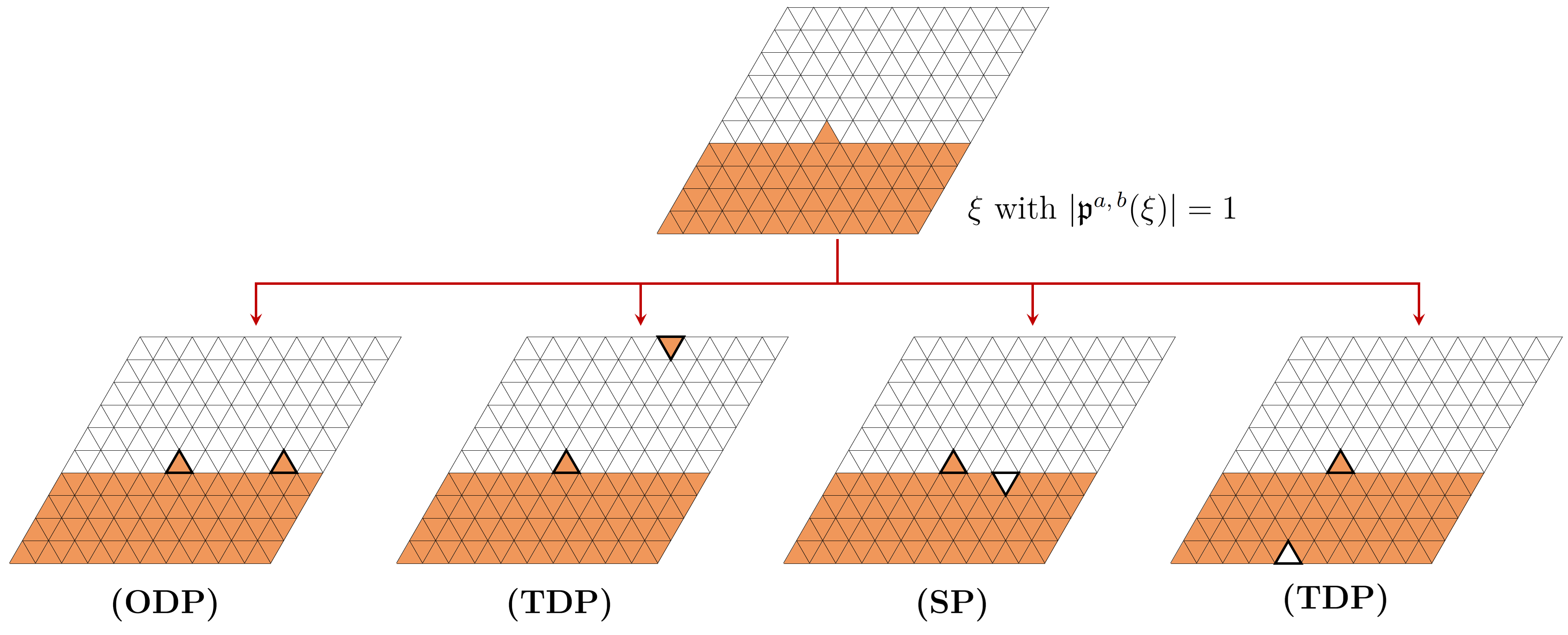}
\par\end{center}

\noindent Therefore, $\sigma$ has two protuberances of size one denoted
by bold boundary and a good flip must remove one of them. Thus, there
are exactly two possible configurations $\zeta_{1},\,\zeta_{2}$ for
$\zeta$ and it is immediate that $\zeta_{1},\,\zeta_{2}\in\mathcal{N}(\mathcal{R}_{n}^{a,\,b})\setminus\mathcal{R}_{n}^{a,\,b}$.
The proof for the case $|\mathfrak{p}^{a,\,b}(\xi)|=2L-1$ is almost
identical to the case $|\mathfrak{p}^{a,\,b}(\xi)|=1$ and we omit
the detail.
\end{proof}
Finally, we provide a summary of the preceding results.
\begin{prop}
\label{p_dead}Let $\sigma\in\bigcup_{n=2}^{L-3}\mathcal{C}_{n}^{a,\,b}$
or $\sigma\in\mathcal{D}^{a,\,b}$ and suppose that $\zeta\in\mathcal{X}$
satisfies $\zeta\approx\sigma$. Then, $\zeta$ is either a canonical
configuration\footnote{Indeed, we have $\zeta\in[\bigcup_{n=2}^{L-3}\mathcal{C}_{n}^{a,\,b}]\cup\mathcal{N}(\mathcal{R}_{2}^{a,\,b})\cup\mathcal{N}(\mathcal{R}_{L-2}^{a,\,b})$.}
or a dead-end in $\mathcal{D}^{a,\,b}$.
\end{prop}

\begin{proof}
This proposition is a direct consequence of Lemmas \ref{l_reg}, \ref{l_odd},
\ref{l_even}, \ref{l_PP}, and \ref{l_QQ}.
\end{proof}
\begin{rem}
\label{r_dead}Now, we are able to explain why the configurations
in $\mathcal{D}^{a,\,b}$ is called dead-end configurations. According
to the definition of $\mathcal{D}^{a,\,b}$, a dead-end $\sigma$
is adjacent to either $\mathcal{N}(\mathcal{R}_{n}^{a,\,b})$ for
some $n\in\llbracket2,\,L-2\rrbracket$ or $\xi\in\mathcal{C}_{n,\,\scal{o}}^{a,\,b}$
such that $|\mathfrak{p}^{a,\,b}(\xi)|\in\llbracket3,\,2L-3\rrbracket$
for some $n\in\llbracket2,\,L-3\rrbracket$. Let $\sigma\approx\zeta$.
Then, for the former case, by Lemmas \ref{l_PP} and \ref{l_QQ}-(2)(3),
$\zeta$ is either another dead-end configuration adjacent to $\mathcal{N}(\mathcal{R}_{n}^{a,\,b})$
or a configuration in $\mathcal{N}(\mathcal{R}_{n}^{a,\,b})$. Hence,
these ones indeed serve as \textit{dead-ends} attached to $\mathcal{N}(\mathcal{R}_{n}^{a,\,b})$
(which consists of canonical configurations only by Lemma \ref{l_reg}).
For the latter case, $\zeta=\xi$ by Lemma \ref{l_QQ}-(1) and therefore
$\{\sigma\}$ is a single \textit{dead-end} attached to the canonical
configuration $\xi$.
\end{rem}

\subsection{\label{sec64}Energy barrier}

Now, we are ready to prove Theorem \ref{t_Ebarrier}. We first establish
the upper bound.
\begin{prop}
\label{p_EBub}For any $a,\,b\in\Omega$, we have $\Phi(\mathbf{a},\,\mathbf{b})\le2L+2$.
\end{prop}

\begin{proof}
Let $P_{0}=\emptyset$ and let $P_{n}=\{1,\,\dots,\,n\}\subseteq\mathbb{T}_{L}$
for $n\in\llbracket1,\,L\rrbracket$ so that $P_{0}\prec P_{1}\prec\cdots\prec P_{L}$.
Since $\xi_{\scal{h}(P_{0})}^{a,\,b}=\mathbf{a}$ and $\xi_{\scal{h}(P_{L})}^{a,\,b}=\mathbf{b}$,
it suffices to show that $\Phi(\xi_{\scal{h}(P_{n})}^{a,\,b},\,\xi_{\scal{h}(P_{n+1})}^{a,\,b})\le2L+2$
for all $n\in\llbracket0,\,L-1\rrbracket$. This follows from Remark
\ref{r_canpath}. 
\end{proof}
Next, we turn to the matching lower bound which is the crucial part
in the proof. 
\begin{prop}
\label{p_EBlb}For any $a,\,b\in\Omega$, we have $\Phi(\mathbf{a},\,\mathbf{b})\ge2L+2$.
\end{prop}

\begin{proof}
Suppose the contrary so that there exists a $(2L+1)$-path $(\omega_{n})_{n=0}^{N}$
in $\mathcal{X}$ with $\omega_{0}=\mathbf{a}$, $\omega_{N}=\mathbf{b}$.
For each $n\in\llbracket0,\,N\rrbracket$, define $u(n)$ as the number
of $b$-bridges in $\omega_{n}$ so that we have $u(0)=0$ and $u(N)=3L$.
Now, we define 
\begin{equation}
n^{*}=\min\{n\in\llbracket0,\,N\rrbracket:u(n)\ge2\}\;,\label{e_EBlb2}
\end{equation}
so that we have a trivial bound $n^{*}\ge3$. Notice that a spin flip
at a certain triangle can only affect the three strips containing
that triangle and hence 
\begin{equation}
|u(n+1)-u(n)|\le3\;\;\;\;\text{for all}\;n\in\llbracket0,\,L-1\rrbracket\;.\label{e_EBlb}
\end{equation}
From this observation, we know that $u(n^{*})\in\llbracket2,\,4\rrbracket$.
On the other hand, by Lemma \ref{l_Elb}, we have at least $3L-(2L+1)=L-1$
bridges, and hence there exists a bridge with spin not $b$. This
implies that $\omega_{n^{*}}$ does not have a cross. Then, we must
have $u(n^{*})-u(n^{*}-1)=1$ and therefore we have $u(n^{*})=2$.

By Propositions \ref{p_E<=00003D2L} and \ref{p_E=00003D2L+1}, we
have either $\omega_{n^{*}}\in\mathcal{R}_{2}^{a',\,b}$ or $\omega_{n^{*}}\in\mathcal{C}_{2,\,\scal{o}}^{a',\,b}$
for some $a'\in\Omega\setminus\{b\}$. If $\omega_{n^{*}}\in\mathcal{R}_{2}^{a',\,b}$,
then by Lemma \ref{l_reg} and the minimality assumption of $n^{*}$,
we must have $\omega_{n^{*}-1}\in\mathcal{C}_{1,\,\scal{o}}^{a',\,b}$.
Then, since $H(\omega_{n^{*}-2})\le2L+1$, we can deduce from Lemma
\ref{l_odd} that $\omega_{n^{*}-2}\in\mathcal{R}_{2}^{a',\,b}$ which
contradicts the minimality of $n^{*}$ in \eqref{e_EBlb2}. On the
other hand, if $\omega_{n^{*}}\in\mathcal{C}_{2,\,\scal{o}}^{a',\,b}$,
then since $H(\omega_{n^{*}-1})\le2L+1$, we can infer from Lemma
\ref{l_odd} that $\omega_{n^{*}-1}\in\mathcal{R}_{2}^{a',\,b}\cup\mathcal{R}_{3}^{a',\,b}$,
and therefore we again get a contradiction to the minimality of $n^{*}$.
Since we got a contradiction for both cases, the proof is completed.
\end{proof}
Now, we can conclude the proof of Theorem \ref{t_Ebarrier}.
\begin{proof}[Proof of Theorem \ref{t_Ebarrier}]
By Propositions \ref{p_EBub} and \ref{p_EBlb}, it suffices to prove
that $\Phi(\sigma,\,\mathcal{S})-H(\sigma)<2L+2$ for all $\sigma\notin\mathcal{S}$.
The proof of this bound is identical to \cite[Lemma 6.11]{KS2} and
we refer to the detailed proof therein.
\end{proof}

\section{\label{sec7_Saddle}Saddle Structure}

In order to get the Eyring--Kramers-type quantitative analysis of
metastability, we need more detailed understanding of the energy landscape.
We perform this in the current section by completely analyzing the
saddle structure between ground states. We remark that the discussion
given in this section has similar flavor with that of \cite[Section 7]{KS2},
but the detail is quite different because we are considering hexagonal
lattice with a complicated dead-end structure, and also we are working
on the large volume regime.

\subsection{\label{sec71}Typical configurations}
\begin{defn}[Typical configurations]
\label{d_typ} Let $(A,\,B)$ be a proper partition of $\Omega$. 
\begin{enumerate}
\item For $a,\,b\in\Omega$, we define the collection of \textit{bulk typical
configurations} between $\mathbf{a}$ and $\mathbf{b}$ as 
\begin{equation}
\mathcal{B}^{a,\,b}=\bigcup_{n=2}^{L-3}\mathcal{C}_{n}^{a,\,b}\cup\mathcal{D}^{a,\,b}\;.\label{e_Bulk}
\end{equation}
Then, we define the collection of bulk configurations between $\mathcal{S}(A)$
and $\mathcal{S}(B)$ as 
\[
\mathcal{B}^{A,\,B}=\bigcup_{a'\in A}\bigcup_{b'\in B}\mathcal{B}^{a',\,b'}\;.
\]
\item For $a,\,b\in\Omega$, we write 
\begin{align*}
\mathcal{B}_{\Gamma}^{a,\,b} & =\{\sigma\in\mathcal{B}^{a,\,b}:H(\sigma)=\Gamma\}=\bigcup_{n=2}^{L-3}\mathcal{C}_{n,\,\scal{e}}^{a,\,b}\cup\mathcal{D}^{a,\,b}\;,\\
\mathcal{B}_{\Gamma}^{A,\,B} & =\{\sigma\in\mathcal{B}^{A,\,B}:H(\sigma)=\Gamma\}=\bigcup_{a'\in A}\bigcup_{b'\in B}\mathcal{B}_{\Gamma}^{a',\,b'}\;.
\end{align*}
Then, we define (cf. Definition \ref{d_nbd})
\begin{equation}
\mathcal{E}^{A}=\widehat{\mathcal{N}}\big(\,\mathcal{S}(A);\,\mathcal{B}_{\Gamma}^{A,\,B}\,\big)\;\;\;\;\text{and}\;\;\;\;\mathcal{E}^{B}=\widehat{\mathcal{N}}\big(\,\mathcal{S}(B);\,\mathcal{B}_{\Gamma}^{A,\,B}\,\big)\;.\label{e_Edge}
\end{equation}
The collection of \textit{edge typical configurations} between $\mathcal{S}(A)$
and $\mathcal{S}(B)$ is defined as 
\[
\mathcal{E}^{A,\,B}=\mathcal{E}^{A}\cup\mathcal{E}^{B}\;.
\]
\end{enumerate}
\end{defn}

In the remainder of the current section, we fix a proper partition
$(A,\,B)$ of $\Omega$.
\begin{rem}
\label{r_typ}In fact, all canonical configurations are indeed typical
configurations. To see this, we let $c_{1},\,c_{2}\in\Omega$ and
demonstrate that $\mathcal{C}_{n}^{c_{1},\,c_{2}}\subseteq\mathcal{B}^{A,\,B}\cup\mathcal{E}^{A,\,B}$
for all $n\in\llbracket0,\,L-1\rrbracket$. First, if $c_{1},\,c_{2}\in A$
or $c_{1},\,c_{2}\in B$, then by Remark \ref{r_canpath}, it is straightforward
that $\mathcal{C}_{n}^{c_{1},\,c_{2}}\subseteq\mathcal{E}^{A}$ or
$\mathcal{C}_{n}^{c_{1},\,c_{2}}\subseteq\mathcal{E}^{B}$, respectively.
Next, we assume that $c_{1}$ and $c_{2}$ belong to different sets,
say $c_{1}\in A$ and $c_{2}\in B$. We divide into two cases.
\begin{enumerate}
\item In the bulk part, for $n\in\llbracket2,\,L-3\rrbracket$ we have $\mathcal{C}_{n}^{c_{1},\,c_{2}}\subseteq\mathcal{B}^{A,\,B}$
which is immediate from the definition \eqref{e_Bulk}.
\item Remarks \ref{r_can} and \ref{r_canpath} imply that all the canonical
configurations in $\mathcal{C}_{0}^{c_{1},\,c_{2}}\cup\mathcal{C}_{1}^{c_{1},\,c_{2}}$
are connected by a $\Gamma$-path in $\mathcal{C}_{0}^{c_{1},\,c_{2}}\cup\mathcal{C}_{1}^{c_{1},\,c_{2}}$
to $\mathbf{c_{1}}\in\mathcal{S}(A)$. As we clearly have $(\mathcal{C}_{0}^{c_{1},\,c_{2}}\cup\mathcal{C}_{1}^{c_{1},\,c_{2}})\cap\mathcal{B}_{\Gamma}^{A,\,B}=\emptyset$,
the definition of $\mathcal{E}^{A}$ implies that $\mathcal{C}_{0}^{c_{1},\,c_{2}}\cup\mathcal{C}_{1}^{c_{1},\,c_{2}}\subseteq\mathcal{E}^{A}$.
Similarly, we have $\mathcal{C}_{L-2}^{c_{1},\,c_{2}}\cup\mathcal{C}_{L-1}^{c_{1},\,c_{2}}\subseteq\mathcal{E}^{B}$.
\end{enumerate}
\end{rem}

\begin{prop}
\label{p_typ}The following properties hold.
\begin{enumerate}
\item We have $\mathcal{E}^{A}\cap\mathcal{B}^{A,\,B}=\mathcal{N}(\mathcal{R}_{2}^{A,\,B})$
and $\mathcal{E}^{B}\cap\mathcal{B}^{A,\,B}=\mathcal{N}(\mathcal{R}_{L-2}^{A,\,B})$. 
\item It holds that $\mathcal{E}^{A,\,B}\cup\mathcal{B}^{A,\,B}=\widehat{\mathcal{N}}(\mathcal{S})$.
\end{enumerate}
\end{prop}

This proposition explains why we defined the typical configurations
as in Definition \ref{d_typ}. In particular, since $\widehat{\mathcal{N}}(\mathcal{S})$
is the collection of all configurations connected to the ground states
by a $\Gamma$-path, we can observe from part (2) of the previous
proposition that  the sets $\mathcal{E}^{A,\,B}$ and $\mathcal{B}^{A,\,B}$
are properly defined to explain the saddle structure between $\mathcal{S}(A)$
and $\text{\ensuremath{\mathcal{S}(B)}. }$ 

The proof of Proposition \ref{p_typ} is identical to that of \cite[Proposition 7.5]{KS2},
since the proof therein is robust against the microscopic feature
of the model. It suffices to replace \cite[Lemma 7.2]{KS2} for the
square lattice with Lemmas \ref{l_reg}, \ref{l_odd}, and \ref{l_even}
for the hexagonal lattice. It should be mentioned that in \cite[Proposition 7.5]{KS2},
it was asserted that $\mathcal{E}^{A}\cap\mathcal{B}^{A,\,B}=\mathcal{R}_{2}^{A,\,B}$
and $\mathcal{E}^{B}\cap\mathcal{B}^{A,\,B}=\mathcal{R}_{L-2}^{A,\,B}$
(instead of $\mathcal{N}(\mathcal{R}_{2}^{A,\,B})$ and $\mathcal{N}(\mathcal{R}_{L-2}^{A,\,B})$)
since in that case it holds that $\mathcal{N}(\mathcal{R}_{i}^{A,\,B})=\mathcal{R}_{i}^{A,\,B}$
for all $i\in\llbracket2,\,L-2\rrbracket$.

The following proposition is the hexagonal version of \cite[Proposition 7.4]{KS2}
and asserts that $\mathcal{E}^{A}$ and $\mathcal{E}^{B}$ are disjoint.
We provide a proof since it is technically more difficult than that
of \cite[Proposition 7.4]{KS2}. A union of two strips of different
directions is called a \textit{$b$-semicross} of $\sigma\in\mathcal{X}$
if all the spins at these two strips are $b$, except for the one
in the intersection which is not $b$, as in the following figure. 
\noindent \begin{center}
\includegraphics[width=4.8cm]{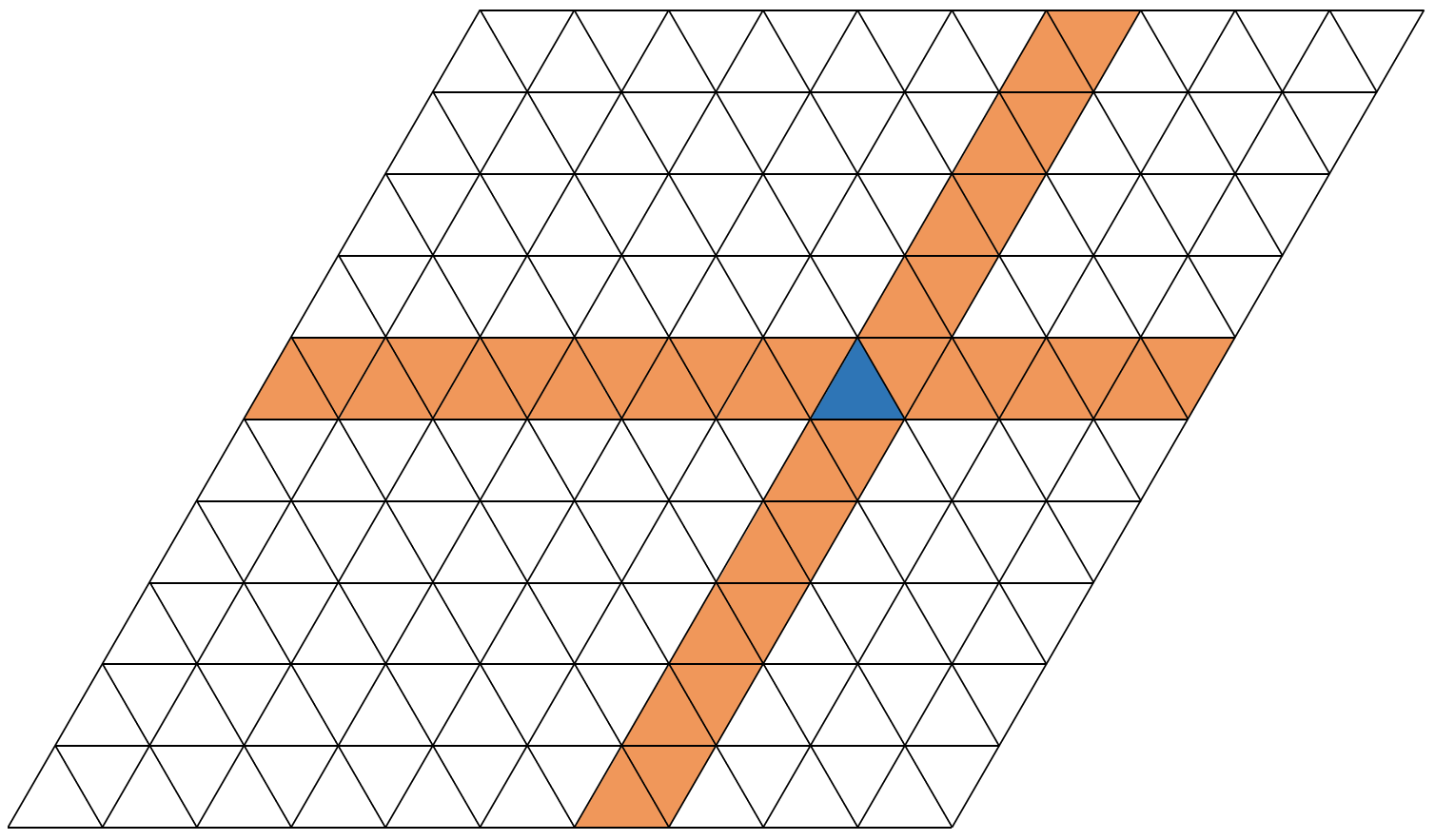}
\par\end{center}
\begin{prop}
\label{p_EAEB}Let $(A,\,B)$ be a proper partition. Each configuration
in $\mathcal{E}^{A}$ does not have a $b$-cross for all $b\in B$.
In particular, it holds that $\mathcal{E}^{A}\cap\mathcal{E}^{B}=\emptyset$.
\end{prop}

\begin{proof}
Suppose on the contrary that $\sigma\in\mathcal{E}^{A}$ has a $b$-cross
for some $b\in B$. Then since $\sigma\in\mathcal{E}^{A}$, we can
find a $\Gamma$-path $(\omega_{n})_{n=0}^{N}$ in $\mathcal{X}\setminus\mathcal{B}_{\Gamma}^{A,\,B}$
with $\omega_{0}\in\mathcal{S}(A)$ and $\omega_{N}=\sigma$. For
$n\in\llbracket0,\,N\rrbracket$, define $u(n)$ as the number of
$b$-bridges in $\omega_{n}$ so that 
\begin{equation}
u(0)=0\;,\;u(N)\ge2\;,\;\text{and }|u(n+1)-u(n)|\le3\text{ for all }n\in\llbracket0,\,N-1\rrbracket\;,\label{e_edge1}
\end{equation}
as in the proof of Proposition \ref{p_EBlb}. Define
\[
n_{0}=\max\{n\ge1:u(n-1)\le1\text{ and }u(n)\ge2\}
\]
so that, summing up, 
\begin{equation}
(\omega_{n})_{n=0}^{N}\text{ is a }\Gamma\text{-path in }\mathcal{X}\setminus\mathcal{B}_{\Gamma}^{A,\,B}\;\;\;\;\text{and}\;\;\;\;u(n)\ge2\text{ for all }n\ge n_{0}\;.\label{e_bft}
\end{equation}
We divide the proof into two cases.\medskip{}

\noindent \textbf{(Case 1: $u(n_{0})-u(n_{0}-1)\ge2$)} In this case,
a single spin update from $\omega_{n_{0}-1}$ to $\omega_{n_{0}}$
creates at least two $b$-bridges. This is possible only when we update
the triangle at the intersection of a $b$-semicross to obtain a $b$-cross.
Since $u(n_{0}-1)\in\{0,\,1\}$, the configuration $\omega_{n_{0}-1}$
cannot have a $b$-cross. Moreover, the existence of a $b$-semicross
implies that there is no $c$-bridge for all $c\in\Omega\setminus\{b\}$.
Hence, $\omega_{n_{0}-1}$ has at most one bridge and its energy is
at least $3L-1$ by Lemma \ref{l_Elb}. This contradicts the fact
that $(\omega_{n})$ is a $\Gamma$-path.\medskip{}

\noindent \textbf{(Case 2: $u(n_{0})-u(n_{0}-1)=1$)} Here, we must
have $u(n_{0}-1)=1$ and $u(n_{0})=2$. Since $\omega_{n_{0}-1}$
has exactly one $b$-bridge, it is cross-free. Moreover, since a single
spin update from $\omega_{n_{0}-1}$ to $\omega_{n_{0}}$ should create
the second $b$-bridge, we can apply Propositions \ref{p_E<=00003D2L},
\ref{p_E=00003D2L+1}, and \ref{p_E=00003D2L+2} to assert that there
are only four possible forms of $\omega_{n_{0}-1}$ as in the figure
below. 
\noindent \begin{center}
\includegraphics[width=16cm]{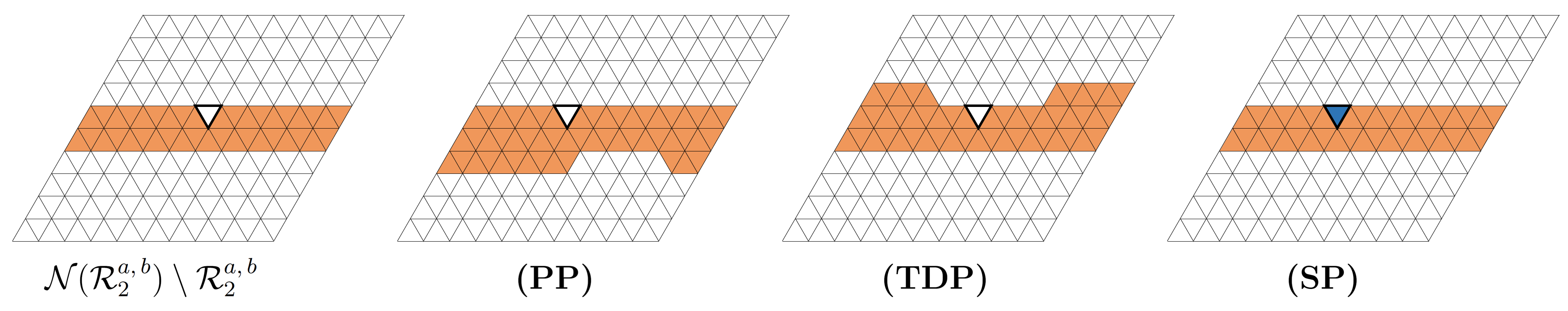}
\par\end{center}

\noindent Note that we have to update the spin at a triangle with
bold boundary to $b$ to get $\omega_{n_{0}},$and hence we have $\omega_{n_{0}}\in\mathcal{R}_{2}^{a,\,b}\cup\mathcal{C}_{2,\,\scal{o}}^{a,\,b}$.
Note that $\omega_{n_{0}}$ does not have a $b$-cross so that $n_{0}<N$.
Now, we consider four subcases.
\begin{itemize}
\item $\omega_{n_{0}}\in\mathcal{R}_{2}^{a,\,b}$: We can conclude from
Lemmas \ref{l_reg}, \ref{l_odd}, \ref{l_PP}, and \ref{l_QQ}-(2)
along with \eqref{e_bft} that $\omega_{n}\in\mathcal{N}(\mathcal{R}_{2}^{a,\,b})$
for all $n\ge n_{0}$. This contradicts the fact that $\omega_{N}$
has a $b$-cross.
\item $\omega_{n_{0}}\in\mathcal{C}_{2,\,\scal{o}}^{a,\,b}$ with $|\mathfrak{p}^{a,\,b}(\omega_{n_{0}})|\in\llbracket3,\,2L-3\rrbracket$:
Since $\omega_{n_{0}}\approx\omega_{n_{0}+1}$, by Lemma \ref{l_odd},
we get $\omega_{n_{0}+1}\in\mathcal{C}_{2,\,\scal{e}}^{a,\,b}\cup\mathcal{Q}_{2}^{a,\,b}\subseteq\mathcal{B}_{\Gamma}^{A,\,B}$.
This contradicts \eqref{e_bft}.
\item $\omega_{n_{0}}\in\mathcal{C}_{2,\,\scal{o}}^{a,\,b}$ with $|\mathfrak{p}^{a,\,b}(\omega_{n_{0}})|=1$:
The same logic with the case $\omega_{n_{0}}\in\mathcal{R}_{2}^{a,\,b}$
leads to the same conclusion.
\item $\omega_{n_{0}}\in\mathcal{C}_{2,\,\scal{o}}^{a,\,b}$ with $|\mathfrak{p}^{a,\,b}(\omega_{n_{0}})|=2L-1$:
By the same logic with the case $\omega_{n_{0}}\in\mathcal{R}_{2}^{a,\,b}$,
we get $\omega_{n}\in\mathcal{N}(\mathcal{R}_{3}^{a,\,b})$ for all
$n\ge n_{0}$. This again contradicts the fact that $\omega_{N}$
has a $b$-cross.
\end{itemize}
Since we get a contradiction for all cases, the first assertion of
the proposition is proved. For the second assertion, we first observe
from the first part of the proposition that $\mathcal{E}^{A}\cap\mathcal{S}(B)=\emptyset$.
Then, by the definitions of $\mathcal{E}^{A}$ and $\mathcal{E}^{B}$,
it also holds that $\mathcal{E}^{A}\cap\mathcal{E}^{B}=\emptyset$.
\end{proof}

\subsection{\label{sec72}Structure of edge configurations}

Fix a proper partition $(A,\,B)$ throughout this subsection. We now
investigate the structure of the sets $\mathcal{E}^{A}$ and $\mathcal{E}^{B}$
more deeply as in \cite[Section 7.3]{KS2}.

We start by decomposing $\mathcal{E}^{A}=\mathcal{I}^{A}\cup\mathcal{O}^{A}$
where
\[
\mathcal{O}^{A}=\{\sigma\in\mathcal{E}^{A}:H(\sigma)=\Gamma\}\;\;\;\;\text{and}\;\;\;\;\mathcal{I}^{A}=\{\sigma\in\mathcal{E}^{A}:H(\sigma)<\Gamma\}\;.
\]
Further, we take a representative set $\mathcal{I}_{\textup{rep}}^{A}\subseteq\mathcal{I}^{A}$
in such a way that each $\sigma\in\mathcal{I}^{A}$ satisfies $\sigma\in\mathcal{N}(\zeta)$
for exactly one $\zeta\in\mathcal{I}_{\textup{rep}}^{A}$. With this
notation, we can further decompose the set $\mathcal{E}^{A}$ into
\[
\mathcal{E}^{A}=\mathcal{O}^{A}\cup\Big(\,\bigcup_{\zeta\in\mathcal{I}_{\textup{rep}}^{A}}\mathcal{N}(\zeta)\,\Big)\;.
\]
For the convenience of notation, we can assume that $\mathcal{S}(A),\,\mathcal{R}_{2}^{A,\,B}\subseteq\mathcal{I}_{\textup{rep}}^{A}$
so that configurations in $\mathcal{N}(\mathbf{a})$ with $\mathbf{a}\in\mathcal{S}(A)$
and in $\mathcal{N}(\sigma)$ with $\sigma\in\mathcal{R}_{2}^{A,\,B}$
are represented by $\mathbf{a}$ and $\sigma$, respectively\footnote{We can notice from Lemma \ref{l_reg} that $\mathcal{N}$-neighborhoods
of two different $\sigma,\,\sigma\in\mathcal{R}_{2}^{A,\,B}$ are
indeed disjoint.}.

We now assign a graph structure on $\mathcal{E}^{A}$ based on this
decomposition. More precisely, we introduce a graph $\mathscr{G}^{A}=(\mathscr{V}^{A},\,E(\mathscr{V}^{A}))$
where the vertex set is defined by $\mathscr{V}^{A}=\mathcal{O}^{A}\cup\mathcal{I}_{\textup{rep}}^{A}$
and the edge set is defined by $\{\sigma,\,\sigma'\}\in E(\mathscr{V}^{A})$
for $\sigma,\,\sigma'\in\mathscr{V}^{A}$ if and only if
\[
\begin{cases}
\sigma,\,\sigma'\in\mathcal{O}^{A}\text{ and }\sigma\sim\sigma'\;\;\text{or}\\
\sigma\in\mathcal{O}^{A}\;,\;\sigma'\in\mathcal{I}_{\textup{rep}}^{A}\text{ and }\sigma\sim\zeta\text{ for some }\zeta\in\mathcal{N}(\sigma')\;.
\end{cases}
\]

Next, we construct a continuous-time Markov chain $\{Z^{A}(t)\}_{t\ge0}$
on $\mathscr{V}^{A}$ with rate $r^{A}:\mathscr{V}^{A}\times\mathscr{V}^{A}\rightarrow\mathbb{R}$
defined by 
\begin{equation}
r^{A}(\sigma,\,\sigma')=\begin{cases}
1 & \text{if }\sigma,\,\sigma'\in\mathcal{O}^{A}\;,\\
|\{\zeta\in\mathcal{N}(\sigma):\zeta\sim\sigma'\}| & \text{if }\sigma\in\mathcal{I}_{\textup{rep}}^{A}\;,\;\sigma'\in\mathcal{O}^{A}\;,\\
|\{\zeta\in\mathcal{N}(\sigma'):\zeta\sim\sigma\}| & \text{if }\sigma\in\mathcal{O}^{A}\;,\;\sigma'\in\mathcal{I}_{\textup{rep}}^{A}\;,
\end{cases}\label{e_rA}
\end{equation}
and $r^{A}(\sigma,\,\sigma')=0$ if $\{\sigma,\,\sigma'\}\notin E(\mathscr{V}^{A})$.
Since the rate $r^{A}(\cdot,\,\cdot)$ is symmetric, the Markov chain
$Z^{A}(\cdot)$ is reversible with respect to the uniform distribution
on $\mathscr{V}^{A}$.
\begin{notation}
\label{n_edge}We denote by $L^{A}(\cdot)$, $h_{\cdot,\,\cdot}^{A}(\cdot)$,
$\mathrm{cap}^{A}(\cdot,\,\cdot)$, and $D^{A}(\cdot)$ the generator,
equilibrium potential, capacity, and Dirichlet form, respectively,
of the Markov chain $Z^{A}(\cdot)$. For those who are not familiar
with these notions, we refer to Section \ref{sec81} for the definitions. 
\end{notation}

\subsubsection*{Configurations in $\mathcal{E}^{A}$}

In the following series of lemmas, we study several essential features
of the configurations in $\mathcal{E}^{A}$. 
\begin{lem}
\label{l_edge1}Suppose that $\sigma\in\mathcal{E}^{A}$ has an $a$-cross
for some $a\in A$. Then, we have that $h_{\mathcal{S}(A),\,\mathcal{R}_{2}^{A,B}}^{A}(\sigma)=1$
(cf. Notation \ref{n_edge}).
\end{lem}

\begin{proof}
We fix $\sigma\in\mathcal{E}^{A}$ which has an $a$-cross for some
$a\in A$. It suffices to prove that any $\Gamma$-path from $\sigma$
to $\mathcal{R}_{2}^{A,\,B}$ in $\mathcal{X}\setminus\mathcal{B}_{\Gamma}^{A,\,B}$
must visit $\mathcal{N}(\mathcal{S}(A))$. Suppose the contrary that
there exists a $\Gamma$-path $(\omega_{n})_{n=0}^{N}$ in $\mathcal{X}\setminus[\mathcal{N}(\mathcal{S}(A))\cup\mathcal{B}_{\Gamma}^{A,\,B}]$
connecting $\omega_{0}\in\mathcal{R}_{2}^{A,\,B}$ and $\omega_{N}=\sigma$.
Let
\[
n_{1}=\min\{n\ge1:\omega_{n}\text{ has an }a\text{-cross}\}\in\llbracket1,\,N\rrbracket\;,
\]
so that $\omega_{n_{1}-1}$ is clearly cross-free. Hence, we are able
to apply Propositions \ref{p_E<=00003D2L}, \ref{p_E=00003D2L+1},
and \ref{p_E=00003D2L+2} to conclude that $\omega_{n_{1}-1}$ is
either of type \textbf{(MB)} or satisfies $\omega_{n_{1}-1}\in\mathcal{C}_{0,\,\scal{o}}^{a,\,b}$
with $|\mathfrak{p}^{a,\,b}(\omega_{n_{1}-1})|=2L-1$. For the former
case, $\omega_{n_{1}-1}$ is clearly not of type \textbf{(MB7)} since
$\omega_{n_{1}}$ must have a cross, and thus by Lemma \ref{l_MB2},
we have $\omega_{n_{1}}\in\mathcal{N}(\mathbf{a})$ and we get a contradiction.
For the latter case, since $\omega_{n_{1}}$ has an $a$-cross, the
configurations $\omega_{n_{1}-1}$ and $\omega_{n_{1}}$ must be of
the following form.
\begin{center}
\includegraphics[width=10.3cm]{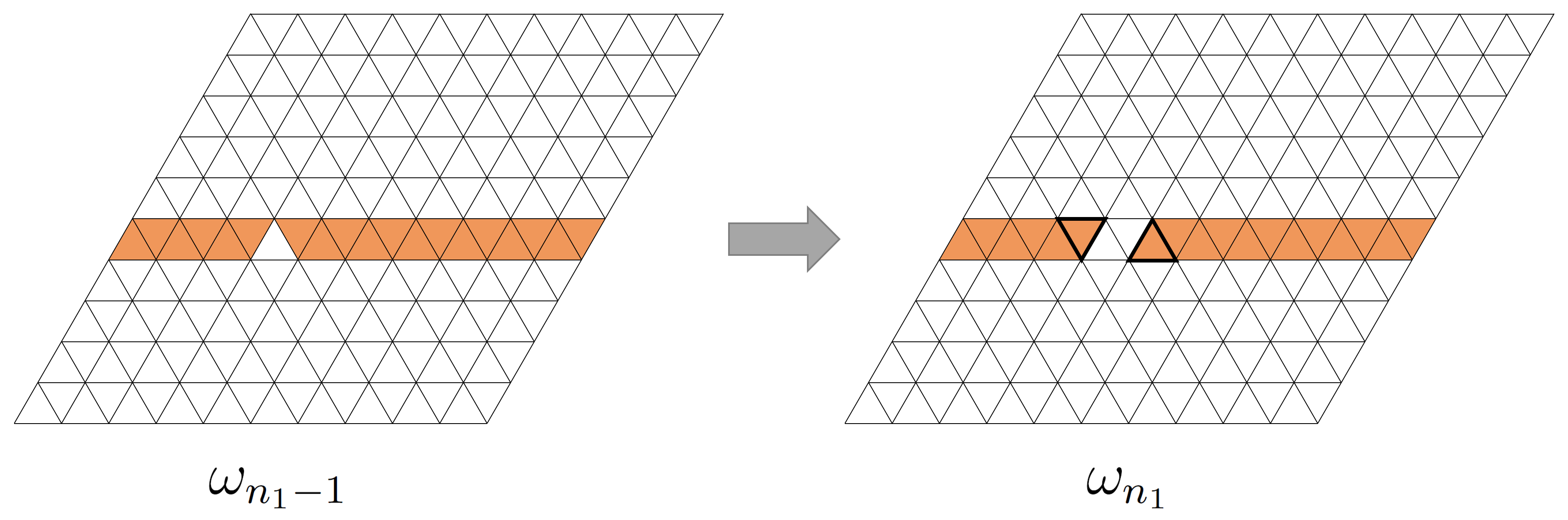}
\par\end{center}

\noindent Thus, we update each spin $b$ in $\omega_{n_{1}}$ to spin
$a$ in a consecutive manner as in the proof of Lemma \ref{l_MB}
(where we start the update from a triangle highlighted by bold boundary)
to get a $(\Gamma-1)$-path from $\omega_{n_{1}}$ to $\mathbf{a}$.
Hence, we have $\omega_{n_{1}}\in\mathcal{N}(\mathbf{a})$ and we
get a contradiction for this case as well.
\end{proof}
\begin{lem}
\label{l_edge2}Fix $a\in A$ and suppose that $\sigma\in\mathcal{N}(\mathbf{a})$
and that there exists $\zeta\in\mathcal{O}^{A}$ such that $\sigma\sim\zeta$
and $h_{\mathcal{S}(A),\,\mathcal{R}_{2}^{A,B}}^{A}(\zeta)\ne1$.
Then, the followings hold. 
\begin{enumerate}
\item There exists a $(\Gamma-1)$-path from $\sigma$ to $\mathbf{a}$
of length less than $4L$. 
\item We have
\[
\big|\,\{\sigma\in\mathcal{N}(\mathbf{a}):\exists\zeta\in\mathcal{O}^{A}\text{ with }\sigma\sim\zeta\text{ and }h_{\mathcal{S}(A),\,\mathcal{R}_{2}^{A,B}}^{A}(\zeta)\ne1\}\,\big|=O(L^{8})\;.
\]
\end{enumerate}
\end{lem}

\begin{proof}
(1) By Proposition \ref{p_EAEB} and Lemma \ref{l_edge1}, $\zeta$
is a cross-free configuration of energy $\Gamma$. Thus, by Proposition
\ref{p_E=00003D2L+2}, $\zeta$ is of type\textbf{ (ODP)},\textbf{
(TDP)},\textbf{ (SP)},\textbf{ (EP)},\textbf{ (PP) }or\textbf{ (MB)}.\textbf{
}For the first five types, since $\sigma\in\mathcal{N}(\mathbf{a})$,
we can readily infer that the only possible cases are $\sigma\in\mathcal{R}_{1}^{a,\,c}$
or $\sigma\in\mathcal{C}_{1,\,\scal{o}}^{a,\,c}$ with $|\mathfrak{p}^{a,\,c}(\sigma)|=1$
for some $c\in\Omega\setminus\{a\}$. Then, we clearly have a $(\Gamma-1)$-path
from $\sigma$ to $\mathbf{a}$ of length $2L$ or $2L+1$ which is
indeed a canonical path (cf. Remarks \ref{r_can} and \ref{r_canpath}).\textbf{
}Next we assume that $\zeta$ is of type\textbf{ (MB)}. If $\zeta$
is of type \textbf{(MB7)}, then clearly we have $\sigma\in\mathcal{C}_{1,\,\scal{o}}^{c,\,c'}$
for some $c,\,c'\in\Omega$, which contradicts $\sigma\in\mathcal{N}(\mathbf{a})$
by Lemmas \ref{l_reg} and \ref{l_odd}. Otherwise, the statement
of part (1) is direct from Lemma \ref{l_MB2}. \medskip{}

\noindent (2) Since $\zeta$ is a cross-free configuration of energy
$\Gamma$, Lemma \ref{l_cfcount} implies that there are $O(L^{6})$
possibilities for $\zeta$. As there are $O(L^{2})$ ways of flipping
a spin, we get a (loose) bound $O(L^{8})$ for the number of possible
configurations for $\sigma$.
\end{proof}
\begin{lem}
\label{l_edge3}Let $\zeta\in\mathcal{I}_{\textup{rep}}^{A}\setminus(\mathcal{S}(A)\cup\mathcal{R}_{2}^{A,\,B}\cup\mathcal{C}_{1,\,\scal{o}}^{A,\,B})$
and let $\sigma\in\mathcal{N}(\zeta)$. Then, $\sigma$ has an $a$-cross
for some $a\in A$. If $\xi\in\mathcal{O}^{A}$ satisfies $\xi\sim\sigma$,
then $\xi$ also has an $a$-cross.
\end{lem}

\begin{proof}
By Propositions \ref{p_E<=00003D2L} and \ref{p_E=00003D2L+1}, $\sigma$
cannot be a cross-free configuration. Since $\sigma$ cannot have
a $b$-cross for $b\in B$ by Proposition \ref{p_EAEB}, it must have
an $a$-cross for some $a\in A$. For the second part of the lemma,
since $\xi\sim\sigma$, the configuration $\xi$ should be cross-free
if it does not have an $a$-cross. If $\xi$ is cross-free, then by
Proposition \ref{p_E=00003D2L+2}, $\xi$ is of type \textbf{(ODP)},\textbf{
(TDP)},\textbf{ (SP)},\textbf{ (EP)},\textbf{ (PP) }or\textbf{ (MB)}.
The first five types are impossible since $\sigma$ has a cross. If
$\xi$ is of type \textbf{(MB)} other than \textbf{(MB7)}, then Lemma
\ref{l_MB2} implies that $\sigma\in\mathcal{N}(\mathbf{a})$ which
contradicts $\zeta\notin\mathcal{S}(A)$. Finally, if $\xi$ is of
type \textbf{(MB7)}, then $\sigma$ cannot have a cross, and thus
we have a contradiction. This concludes the proof.
\end{proof}

\subsubsection*{Estimate of jump rate}

The following proposition explains the reason why we introduced the
Markov chain $Z^{A}(\cdot)$.
\begin{prop}
\label{p_ZA}Define a projection map $\Pi^{A}:\mathcal{E}^{A}\rightarrow\mathscr{V}^{A}$
by
\[
\Pi^{A}(\sigma)=\begin{cases}
\zeta & \text{if }\sigma\in\mathcal{N}(\zeta)\text{ for some }\zeta\in\mathcal{I}_{\textup{rep}}^{A}\;,\\
\sigma & \text{if }\sigma\in\mathcal{O}^{A}\;.
\end{cases}
\]
Suppose that $L^{2/3}\ll e^{\beta}$. Then, the followings hold.
\begin{enumerate}
\item For $\sigma_{1},\,\sigma_{2}\in\mathcal{O}^{A}$ with $\sigma_{1}\sim\sigma_{2}$,
we have
\[
\frac{1}{q}e^{-\Gamma\beta}r^{A}\big(\,\Pi^{A}(\sigma_{1}),\,\Pi^{A}(\sigma_{2})\,\big)=(1+o_{L}(1))\times\mu_{\beta}(\sigma_{1})c_{\beta}(\sigma_{1},\,\sigma_{2})\;.
\]
\item For $\sigma_{1}\in\mathcal{O}^{A}$ and $\sigma_{2}\in\mathcal{I}^{A}$
with $\sigma_{1}\sim\sigma_{2}$, we have
\[
\frac{1}{q}e^{-\Gamma\beta}r^{A}\big(\,\Pi^{A}(\sigma_{1}),\,\Pi^{A}(\sigma_{2})\,\big)=(1+o_{L}(1))\times\sum_{\zeta\in\mathcal{N}(\sigma_{2})}\mu_{\beta}(\sigma_{1})c_{\beta}(\sigma_{1},\,\zeta)\;.
\]
\end{enumerate}
\end{prop}

\begin{proof}
(1) By definition, we have $r^{A}\big(\,\Pi^{A}(\sigma_{1}),\,\Pi^{A}(\sigma_{2})\,\big)=1$
and thus the conclusion follows immediately from \eqref{e_detbal}
and part (1) of Theorem \ref{t_Gibbs1}.\medskip{}

\noindent (2) For this case, by definition we can write 
\[
\frac{1}{q}e^{-\Gamma\beta}r^{A}\big(\,\Pi^{A}(\sigma_{1}),\,\Pi^{A}(\sigma_{2})\,\big)=\frac{1}{q}e^{-\Gamma\beta}\times\big|\,\{\zeta\in\mathcal{N}(\sigma_{2}):\zeta\sim\sigma_{1}\}\,\big|\;.
\]
By part (1) of Theorem \ref{t_Gibbs1} and \eqref{e_detbal}, the
right-hand side equals
\begin{align*}
 & (1+o_{L}(1))\times\mu_{\beta}(\sigma_{1})\times\big|\,\{\zeta\in\mathcal{N}(\sigma_{2}):\zeta\sim\sigma_{1}\}\,\big|\\
 & =(1+o_{L}(1))\times\sum_{\zeta\in\mathcal{N}(\sigma_{2}):\,\zeta\sim\sigma_{1}}\mu_{\beta}(\sigma_{1})c_{\beta}(\sigma_{1},\,\zeta)\;,
\end{align*}
where we implicitly used the fact that $\min\{\mu_{\beta}(\sigma_{1}),\,\mu_{\beta}(\sigma_{2})\}=\mu_{\beta}(\sigma_{1})$
at the identity. This proves part (2).
\end{proof}

\subsubsection*{An auxiliary constant}

Finally, we define a constant 
\begin{equation}
\mathfrak{e}_{A}=\frac{1}{|\mathscr{V}^{A}|\mathrm{cap}^{A}\big(\,\mathcal{S}(A),\,\mathcal{R}_{2}^{A,\,B}\,\big)}\;.\label{e_eA}
\end{equation}

\begin{prop}
\label{p_eAest}We have $\mathfrak{e}_{A}\le L^{-1}$.
\end{prop}

\begin{proof}
As in the proof of \cite[Proposition 7.9]{KS2}, the proof is completed
by applying Thomson principle along with a test flow defined along
a canonical path from $\mathcal{R}_{2}^{A,\,B}$ to $\mathcal{S}(A)$.
We do not tediously repeat the proof and refer the readers to \cite[Proposition 7.9]{KS2}
for more detailed explanation of this method.
\end{proof}
To conclude this section, it should be remarked that we can repeat
the same constructions on the other set $\mathcal{E}^{B}$ and obviously
the same conclusions also hold for this set as well.

\section{\label{sec8_PT}General Strategy for Eyring--Kramers Formula}

In the remainder of the article, we focus on the proof of Eyring--Kramers
formula (Theorem \ref{t_EK}) based on our careful investigation of
the energy landscape carried out in the previous section. To that
end, we review the robust strategy for the potential-theoretic proof
of the Eyring--Kramers formula in this section. Although our contents
are self-contained, we refer to \cite[Sections 3 and 4]{KS2} for
more comprehensive discussions on the strategy given in this section.

\subsection{\label{sec81}Proof of Theorem \ref{t_EK} via capacity estimates}

The \textit{Dirichlet form} $D_{\beta}(\cdot)$ associated with the
Metropolis dynamics $\sigma_{\beta}(\cdot)$ (cf. Section \ref{sec23})
is given by, for each $f:\mathcal{X}\rightarrow\mathbb{R}$,
\[
D_{\beta}(f)=\frac{1}{2}\sum_{\sigma,\,\zeta\in\mathcal{X}}\mu_{\beta}(\sigma)c_{\beta}(\sigma,\,\zeta)[f(\zeta)-f(\sigma)]^{2}\;.
\]
For disjoint and non-empty subsets $\mathcal{A}$ and $\mathcal{B}$
of $\mathcal{X}$, the \textit{equilibrium potential} and \textit{capacity}
between $\mathcal{A}$ and $\mathcal{B}$ are defined as
\[
h_{\mathcal{A},\,\mathcal{B}}^{\beta}(\sigma)=\mathbb{P}_{\sigma}^{\beta}\big[\,\tau_{\mathcal{A}}<\tau_{\mathcal{B}}\,\big]\;\;\;\;\text{and}\;\;\;\;\mathrm{Cap}_{\beta}(\mathcal{A},\,\mathcal{B})=D_{\beta}\big(\,h_{\mathcal{A},\,\mathcal{B}}^{\beta}\,\big)\;,
\]
respectively. The following is the sharp capacity estimate between
ground states.
\begin{thm}[Capacity estimate for hexagonal lattice]
\label{t_cap} Suppose that $\beta=\beta_{L}$ satisfies $L^{10}\ll e^{\beta}$
and let $(A,\,B)$ be a proper partition. Then, it holds that 
\begin{equation}
\mathrm{Cap}_{\beta}\big(\,\mathcal{S}(A),\,\mathcal{S}(B)\,\big)=\Big[\,\frac{12|A|(q-|A|)}{q}+o_{L}(1)\,\Big]\,e^{-\Gamma\beta}\;.\label{e_cap}
\end{equation}
\end{thm}

\begin{rem}[Capacity estimate for square lattice]
 For the square lattice, we have the same form of capacity estimate
under the condition $L^{3}\ll e^{\beta}$, where the only difference
is that the constant $12$ in the right-hand side of \eqref{e_cap}
should be replaced by $8$.
\end{rem}

The proof of this theorem will be given in Sections \ref{sec9_UB}
and \ref{sec10_LB}. At this moment, let us conclude the proof of
Theorem \ref{t_EK} by assuming Theorem \ref{t_cap}.
\begin{proof}[Proof of Theorem \ref{t_EK}]
 We first consider the formula in \eqref{e_EK}. By the well-known
formula established in \cite[display (3.18)]{BEGK2} (for more detailed
discussion, we also refer to \cite[Proposition 6.10]{BL1}.), we can
write
\begin{equation}
\mathbb{E}_{\mathbf{a}}^{\beta}\big[\,\tau_{\mathcal{S}\setminus\{\mathbf{a}\}}\,\big]=\frac{\sum_{\sigma\in\mathcal{X}}\mu_{\beta}(\sigma)\,h_{\mathbf{a},\,\mathcal{S}\setminus\{\mathbf{a}\}}^{\beta}(\sigma)}{\mathrm{Cap}_{\beta}(\mathbf{a},\,\mathcal{S}\setminus\{\mathbf{a}\})}\;\;\;\;\text{for }\mathbf{a}\in\mathcal{S}\;.\label{e_EK1}
\end{equation}
Since 
\[
\Big|\,\sum_{\sigma\in\mathcal{S}}\mu_{\beta}(\sigma)h_{\mathbf{a},\,\mathcal{S}\setminus\{\mathbf{a}\}}^{\beta}(\sigma)-\mu_{\beta}(\mathbf{a})\,\Big|\le\mu_{\beta}(\mathcal{X}\setminus\mathcal{S})=o_{L}(1)
\]
by part (1) of Theorem \ref{t_Gibbs1}, we can conclude that the numerator
in the right-hand side of \eqref{e_EK1} is $\frac{1}{q}+o_{L}(1)$.
Since the denominator is $[\frac{12(q-1)}{q}+o_{L}(1)]e^{-\Gamma\beta}$
by Theorem \ref{t_cap} with $A=\{a\}$ and $B=\Omega\setminus\{a\}$,
we can prove the first formula in \eqref{e_EK}.

Now, let us turn to the second formula of \eqref{e_EK}. By the symmetry
of the model, we have $\mathbb{P}_{\mathbf{a}}^{\beta}[\sigma_{\beta}(\tau_{\mathcal{S}\setminus\{\mathbf{a}\}})=\mathbf{b}]=\frac{1}{q-1}$.
If $\sigma_{\beta}(\tau_{\mathcal{S}\setminus\{\mathbf{a}\}})\neq\mathbf{b}$,
we can refresh the dynamics from $t=\tau_{\mathcal{S}\setminus\{\mathbf{a}\}}$.
Then, by the strong Markov property, we get a geometric random variable
(with success probability $1/(q-1)$) structure and thus deduce that
\[
\mathbb{E}_{\mathbf{a}}^{\beta}\big[\,\tau_{\mathbf{b}}\,\big]=(q-1)\mathbb{E}_{\mathbf{a}}^{\beta}\big[\,\tau_{\mathcal{S}\setminus\{\mathbf{a}\}}\,\big]\;.
\]
For more rigorous and formal proof of this argument, we refer to \cite[Section 3.2]{KS2}.
Hence, the first and second formulas in \eqref{e_EK} are equivalent
to each other and we conclude the proof.
\end{proof}

\subsection{\label{sec82}Strategy to estimate capacity}

We now explain two variational principles to estimate the capacity,
which will be crucially used in the proof of Theorem \ref{t_cap}.
Although our discussion is self-contained, we refer to \cite[Section 4]{KS2}
for more detailed explanation.

\subsubsection*{Dirichlet principle}

We fix two disjoint and non-empty subsets $\mathcal{A}$ and $\mathcal{B}$
of $\mathcal{X}$ in this subsection. We first establish the minimization
principle for $\mathrm{Cap}_{\beta}(\mathcal{A},\,\mathcal{B})$.
Denote by $\mathfrak{C}(\mathcal{A},\,\mathcal{B})$ the collection
of functions $f:\mathcal{X}\rightarrow\mathbb{R}$ such that $f=1$
on $\mathcal{A}$ and $f=0$ on $\mathcal{B}$. 
\begin{thm}[Dirichlet principle]
\label{t_DP} We have
\[
\mathrm{Cap}_{\beta}(\mathcal{A},\,\mathcal{B})=\min_{f\in\mathfrak{C}(\mathcal{A},\,\mathcal{B})}D_{\beta}(f)\;.
\]
Moreover, the equilibrium potential $h_{\mathcal{A},\,\mathcal{B}}^{\beta}$
is the unique optimizer of the minimization problem.
\end{thm}

For the proof of this well-known principle, we refer to \cite[Chapter 7]{BdH}.
We remark that this principle holds for the process $\sigma_{\beta}(\cdot)$
since it is reversible. 

\subsubsection*{Generalized Thomson principle}

In order to explain the maximization problem for capacity, we recall
the flow structure. A function $\phi:\mathcal{X}\times\mathcal{X}\rightarrow\mathbb{R}$
is called a\textit{ flow} on $\mathcal{X}$ associated with the Markov
process $\sigma_{\beta}(\cdot)$, if it is anti-symmetric in the sense
that $\phi(\sigma,\,\zeta)=-\phi(\zeta,\,\sigma)$ for all $\sigma,\,\zeta\in\mathcal{X}$\footnote{We set $\phi(\sigma,\,\sigma)=0$ for all $\sigma\in\mathcal{X}.$},
and satisfies $\phi(\sigma,\,\zeta)\ne0$ only if $\sigma\sim\zeta$.

For each flow $\phi$, we define the \textit{norm} and \textit{divergence}
of the flow $\phi$ by 
\begin{align*}
\Vert\phi\Vert_{\beta}^{2} & =\frac{1}{2}\sum_{\sigma,\,\zeta\in\mathcal{X}:\,\sigma\sim\zeta}\frac{\phi(\sigma,\,\zeta)^{2}}{\mu_{\beta}(\sigma)c_{\beta}(\sigma,\,\zeta)}\;,\\
(\mathrm{div}\,\phi)(\sigma) & =\sum_{\zeta\in\mathcal{X}}\phi(\sigma,\,\zeta)=\sum_{\zeta\in\mathcal{X}:\,\sigma\sim\zeta}\phi(\sigma,\,\zeta)\;\;\;\;\text{for all }\sigma\in\mathcal{X}\;.
\end{align*}
The\textit{ harmonic flow} \textit{$\varphi_{\mathcal{A},\,\mathcal{B}}^{\beta}$}
between $\mathcal{A}$ and $\mathcal{B}$ is defined by 
\[
\varphi_{\mathcal{A},\,\mathcal{B}}^{\beta}(\sigma,\,\zeta)=\mu_{\beta}(\sigma)c_{\beta}(\sigma,\,\zeta)\big[\,h_{\mathcal{A},\,\mathcal{B}}^{\beta}(\sigma)-h_{\mathcal{A},\,\mathcal{B}}^{\beta}(\zeta)\,\big]\;\;\;\;;\;\sigma\in\mathcal{X}\;,\;\zeta\in\mathcal{X}\;.
\]
The detailed balance condition \eqref{e_detbal} ensures that $\varphi_{\mathcal{A},\,\mathcal{B}}^{\beta}$
is indeed a flow, and we can readily verify from definitions that
$\Vert\varphi_{\mathcal{A},\,\mathcal{B}}^{\beta}\Vert_{\beta}^{2}=\mathrm{Cap}_{\beta}(\mathcal{A},\,\mathcal{B})$.
\begin{thm}[Generalized Thomson principle \cite{Seo ZRP}]
\label{t_gTP} We have 
\[
\mathrm{Cap}_{\beta}(\mathcal{A},\,\mathcal{B})=\max_{\phi\ne0}\frac{1}{\|\phi\|_{\beta}^{2}}\Big[\,\sum_{\sigma\in\mathcal{X}}h_{\mathcal{A},\,\mathcal{B}}^{\beta}(\sigma)(\mathrm{div}\,\phi)(\sigma)\,\Big]^{2}\;.
\]
Moreover, the flow $c\varphi_{\mathcal{A},\,\mathcal{B}}^{\beta}$
for $c\ne0$ is an optimizer of the maximization problem.
\end{thm}

We refer to \cite[Theorem 4.7]{KS2} for the proof. As in Theorem
\ref{t_DP}, the reversibility of the process $\sigma_{\beta}(\cdot)$
is essentially used in the formulation of this maximization problem. 

We use Theorems \ref{t_DP} and \ref{t_gTP} to establish the sharp
upper and lower bounds of the capacity $\text{Cap}_{\beta}(\mathcal{S}(A),\,\mathcal{S}(B))$
for each proper partition $(A,\,B)$ in Sections \ref{sec9_UB} and
\ref{sec10_LB}, respectively. 

\section{\label{sec9_UB}Upper Bound for Capacities}
\begin{notation}
\label{n_sec910}In this and the next sections, we fix a proper partition
$(A,\,B)$. Then, we define the constants $\mathfrak{b}=\mathfrak{b}(L,\,A,\,B)$
and $\mathfrak{c}=\mathfrak{c}(L,\,A,\,B)$ as
\begin{equation}
\mathfrak{b}=\frac{(5L-3)(L-4)}{60L^{2}|A|(q-|A|)}\;\;\;\;\text{and}\;\;\;\;\mathfrak{c}=\mathfrak{b}+\mathfrak{e}_{A}+\mathfrak{e}_{B}\label{e_b}
\end{equation}
where the constants $\mathfrak{e}_{A}$ and $\mathfrak{e}_{B}$ are
defined in \eqref{e_eA}. Then, by \eqref{e_b} and Proposition \ref{p_eAest},
\begin{equation}
\mathfrak{c}=\mathfrak{b}+\mathfrak{e}_{A}+\mathfrak{e}_{B}=\frac{1}{12|A|(q-|A|)}(1+o_{L}(1))\;.\label{e_c}
\end{equation}

The purpose of the current section is to establish a suitable test
function in order to use the Dirichlet principle to get a sharp upper
bound of $\text{Cap}_{\beta}(\mathcal{S}(A),\,\mathcal{S}(B))$. The
corresponding computation for the square lattice in the $\beta\rightarrow\infty$
regime was carried out in \cite[Section 9]{KS2}, where the discontinuity
of the test function along the boundary of $\widehat{\mathcal{N}}(\mathcal{S})$
and the set $\widehat{\mathcal{N}}(\mathcal{S})^{c}$ was easily handled
since energy is the only dominating factor of the system (as $L$
is fixed). For the current model, we need to be very careful when
controlling the discontinuity of the test function along this boundary,
because the number of configurations with higher energy also increases
as $L$ tends to $\infty$. This difficulty imposes the sub-optimal
condition on $\beta$ (i.e., $L^{10}\ll e^{\beta}$).
\end{notation}

\subsection{\label{sec91}Construction of test function}

The following definition (Definition \ref{d_tf}) constructs our test
function which approximates the equilibrium potential $h_{\mathcal{S}(A),\,\mathcal{S}(B)}^{\beta}$
between $\mathcal{S}(A)$ and $\mathcal{S}(B)$ in view of Theorem
\ref{t_DP}. 
\begin{notation}
\label{n_92}The following notation will be used in the remainder
of the article.
\begin{enumerate}
\item We simply write $\mathfrak{h}^{A}=h_{\mathcal{S}(A),\,\mathcal{R}_{2}^{A,B}}^{A}:\mathscr{V}^{A}\rightarrow[0,\,1]$
(cf. Notation \ref{n_edge}) and naturally extend $\mathfrak{h}^{A}$
to a function on $\mathcal{E}^{A}$ by letting $\mathfrak{h}^{A}(\sigma)=\mathfrak{h}^{A}(\zeta)$
if $\sigma\in\mathcal{N}(\zeta)$ for some $\zeta\in\mathcal{I}_{\textup{rep}}^{A}$.
\item For each $a\in\Omega$ and $\mathcal{\sigma\in X}$, we write $\Vert\sigma\Vert_{a}$
the number of sites in $\Lambda$ with spin $a$ (in $\sigma$), i.e.,
\begin{equation}
\Vert\sigma\Vert_{a}=\sum_{x\in\Lambda}\mathbf{1}\{\sigma(x)=a\}\;.\label{e_counta}
\end{equation}
\end{enumerate}
\end{notation}

\begin{defn}
\label{d_tf}We now define a function $f=f^{A,\,B}:\mathcal{X}\rightarrow\mathbb{R}$.
Recall Notation \ref{n_sec910} and \ref{n_92}.\medskip{}

\noindent (1) \textbf{Construction on $\mathcal{E}^{A,\,B}=\mathcal{E}^{A}\cup\mathcal{E}^{B}$:
}We define (cf. \eqref{e_eA})
\[
f(\sigma)=\begin{cases}
1-\frac{\mathfrak{e}_{A}}{\mathfrak{c}}[1-\mathfrak{h}^{A}(\sigma)] & \text{if }\sigma\in\mathcal{E}^{A}\;,\\
\frac{\mathfrak{e}_{B}}{\mathfrak{c}}[1-\mathfrak{h}^{B}(\sigma)] & \text{if }\sigma\in\mathcal{E}^{B}\;.
\end{cases}
\]

\noindent (2) \textbf{Construction on $\mathcal{B}^{A,\,B}$:} Let
$a\in A$ and $b\in B$. By \eqref{e_Bulk}, it suffices to consider
the following cases.
\begin{itemize}
\item $\sigma\in\mathcal{N}(\mathcal{R}_{n}^{a,\,b})$ with $n\in\llbracket2,\,L-2\rrbracket$:
\[
f(\sigma)=\frac{1}{\mathfrak{c}}\Big[\,\frac{L-2-n}{L-4}\mathfrak{b}+\mathfrak{e}_{B}\,\Big]\;.
\]
\item $\sigma\in\mathcal{C}_{n}^{a,\,b}$ with $n\in\llbracket2,\,L-3\rrbracket$
and $|\mathfrak{p}^{a,\,b}(\sigma)|\in\llbracket2,\,2L-2\rrbracket$
(the case $|\mathfrak{p}^{a,\,b}(\sigma)|\in\{0,\,1,\,2L-1,\,2L\}$
is considered above): 
\[
f(\sigma)=\begin{cases}
\frac{1}{\mathfrak{c}}\big[\,\frac{(5L-3)(L-2-n)-(3-\mathfrak{d}(\sigma))}{(5L-3)(L-4)}\mathfrak{b}+\mathfrak{e}_{B}\,\big] & \text{if }|\mathfrak{p}^{a,\,b}(\sigma)|=2\;,\\
\frac{1}{\mathfrak{c}}\big[\,\frac{(5L-3)(L-3-n)+(3-\mathfrak{d}(\sigma))}{(5L-3)(L-4)}\mathfrak{b}+\mathfrak{e}_{B}\,\big] & \text{if }|\mathfrak{p}^{a,\,b}(\sigma)|=2L-2\;,\\
\frac{1}{\mathfrak{c}}\big[\,\frac{(5L-3)(L-2-n)-\frac{5m-3}{2}}{(5L-3)(L-4)}\mathfrak{b}+\mathfrak{e}_{B}\,\big] & \text{if }|\mathfrak{p}^{a,\,b}(\sigma)|=m\in\llbracket3,\,2L-3\rrbracket\;,
\end{cases}
\]
where $\mathfrak{d}(\sigma)=\mathbf{1}\{\sigma:\mathfrak{p}^{a,\,b}(\sigma)\text{ is disconnected}\}$.
\item $\sigma\in\mathcal{D}^{a,\,b}$: By the definition of $\mathcal{D}^{a,\,b}$,
we can find a canonical configuration $\zeta$ in $\mathcal{B}^{a,\,b}$
such that $\zeta\sim\sigma$. If such $\zeta$ is unique, we set $f(\sigma)=f(\zeta)$.
In view of Lemmas \ref{l_PP} and \ref{l_QQ}, it is also possible
that there are two such canonical configurations $\zeta_{1}$ and
$\zeta_{2}$, but in that case we have $\zeta_{1},\,\zeta_{2}\in\mathcal{N}(\mathcal{R}_{n}^{A,\,B})$
for some $n\in\llbracket2,\,L-2\rrbracket$ and therefore $f(\zeta_{1})=f(\zeta_{2})$
by the definition above. We set $f(\sigma)=f(\zeta_{1})=f(\zeta_{2})$
in this case.
\end{itemize}
We note in this moment that parts (1) and (2) do not collide on the
set $\mathcal{E}^{A,\,B}\cap\mathcal{B}^{A,\,B}=\mathcal{N}(\mathcal{R}_{2}^{A,\,B})\cup\mathcal{N}(\mathcal{R}_{L-2}^{A,\,B})$
(cf. Proposition \ref{p_typ}-(1)), since both definitions assign
the same value $\frac{\mathfrak{b}+\mathfrak{e}_{B}}{\mathfrak{c}}$
(resp. $\frac{\mathfrak{e}_{B}}{\mathfrak{c}}$) on $\mathcal{N}(\mathcal{R}_{2}^{A,\,B})$
(resp. $\mathcal{N}(\mathcal{R}_{L-2}^{A,\,B})$).\medskip{}

\noindent (3) \textbf{Construction on $\widehat{\mathcal{N}}(\mathcal{S})^{c}$:}
For $\sigma\in\widehat{\mathcal{N}}(\mathcal{S})^{c}$, we define
(cf. \eqref{e_counta})
\[
f(\sigma)=\begin{cases}
1 & \text{if }\sum_{a\in A}\|\sigma\|_{a}\ge L^{2}\;,\\
0 & \text{if }\sum_{a\in A}\|\sigma\|_{a}<L^{2}\;.
\end{cases}
\]
By Proposition \ref{p_typ}-(2), the constructions above define $f$
on the set $\mathcal{X}$.
\end{defn}

In the remainder of the current section, we shall prove the following
proposition.
\begin{prop}
\label{p_tf}The test function $f=f^{A,\,B}$ constructed in the previous
definition belongs to $\mathfrak{C}(\mathcal{S}(A),\,\mathcal{S}(B))$
and moreover satisfies
\[
D_{\beta}(f)=\frac{1+o_{L}(1)}{q\mathfrak{c}}e^{-\Gamma\beta}\;.
\]
\end{prop}

\subsection{Configurations with intermediate energy}

The purpose of the current section is to provide some estimates that
control the discontinuity of the test function $f$ along the boundary
of $\widehat{\mathcal{N}}(\mathcal{S})$, which will be the most difficult
part in the proof of Proposition \ref{p_tf} and was not encountered
in the small volume regime considered in \cite{KS2}.

A pair of configurations $(\sigma,\,\zeta)$ in $\mathcal{X}$ is
called a \textit{nice pair} if they satisfy (cf. \eqref{e_counta})
\begin{equation}
\sigma,\,\zeta\notin\widehat{\mathcal{N}}(\mathcal{S})\;,\;\;\;\sigma\sim\zeta\;,\;\;\;\sum_{a\in A}\|\sigma\|_{a}=L^{2}\;\;\;\text{and}\;\;\;\sum_{a\in A}\|\zeta\|_{a}=L^{2}-1\;.\label{e_itmE}
\end{equation}
The following counting of nice pairs is the main result of the current
section. 
\begin{prop}
\label{p_count}For $i\ge0$, denote by $U_{i}=U_{i}^{A,\,B,\,L}$
the number of nice pairs $(\sigma,\,\zeta)$ satisfying $\max\{H(\sigma),\,H(\zeta)\}=2L+i$.
\begin{enumerate}
\item We have $U_{i}=0$ for all $i\le2$.
\item There exists a constant $C=C(q)>0$ such that $U_{i}\le(CL){}^{3i+1}$
for all $i<(\sqrt{6}-2)L-1$.
\end{enumerate}
\end{prop}

To prove this proposition, we first establish an isoperimetric inequality.
\begin{lem}
\label{l_cross}Suppose that $\sigma\in\mathcal{X}$ has an $a$-cross
for some $a\in\Omega$. Then, we have $\sum_{b\in\Omega\setminus\{a\}}\|\sigma\|_{b}\le\frac{H(\sigma)^{2}}{6}$.
\end{lem}

\begin{proof}
We fix $b_{0}\in\Omega\setminus\{a\}$ and define $\widetilde{\sigma}\in\mathcal{X}$
by
\[
\widetilde{\sigma}(x)=\begin{cases}
a & \text{if }\sigma(x)=a\;,\\
b_{0} & \text{if }\sigma(x)\ne a\;.
\end{cases}
\]
Then, it is immediate that $H(\widetilde{\sigma})\le H(\sigma)$ and
$\sum_{b\in\Omega\setminus\{a\}}\|\sigma\|_{b}=\|\widetilde{\sigma}\|_{b_{0}}$.
Therefore, it suffices to prove that $\|\widetilde{\sigma}\|_{b_{0}}\le\frac{H(\widetilde{\sigma})^{2}}{6}$.
As $\widetilde{\sigma}$ also has an $a$-cross, this is a direct
consequence of the isoperimetric inequality \cite[Theorem 1.2]{Gru}. 
\end{proof}
\begin{proof}[Proof of Proposition \ref{p_count}]
 Let $(\sigma,\,\zeta)$ be a nice pair satisfying $\max\{H(\sigma),\,H(\zeta)\}<\sqrt{6}L-1$.
Suppose now that $\eta\in\{\sigma,\,\zeta\}$ has a $c$-cross for
some $c\in\Omega$. Then, by Lemma \ref{l_cross}, we have
\[
\sum_{c':\,c'\ne c}\|\eta\|_{c'}\le\frac{H(\eta)^{2}}{6}<\frac{(\sqrt{6}L-1)^{2}}{6}<L^{2}-1\;.
\]
If $c\in B$, we get a contradiction to $\sum_{a\in A}\|\eta\|_{a}\ge L^{2}-1$,
and we get a similar contradiction when $c\in A$. Thus, both $\sigma$
and $\zeta$ are cross-free.\medskip{}

\noindent (1) Suppose that there exists a nice pair $(\sigma,\,\zeta)$
such that $H(\sigma),\,H(\zeta)\le\Gamma$. If the cross-free configuration
$\eta\in\{\sigma,\,\zeta\}$ satisfies $H(\eta)<\Gamma$, we can apply
Propositions \ref{p_E<=00003D2L} and \ref{p_E=00003D2L+1} to conclude
that $\eta\in\mathcal{R}_{n}^{a_{1},\,a_{2}}\cup\mathcal{C}_{n,\,\scal{o}}^{a_{1},\,a_{2}}$
for some $n$ and $a_{1},\,a_{2}\in\Omega$. Then by Remark \ref{r_typ},
we obtain $\eta\in\widehat{\mathcal{N}}(\mathcal{S})$ which yields
a contradiction. Therefore, we must have that $H(\sigma)=H(\zeta)=\Gamma$.
Since $\sigma\sim\zeta$, by Proposition \ref{p_E=00003D2L+2}, we
can notice that $\sigma$ and $\zeta$ must be both of type \textbf{(PP)}
or both of type\textbf{ (MB)}. If they are both of type \textbf{(PP)},
then Lemma \ref{l_PP} implies that $\sigma,\,\zeta\in\widehat{\mathcal{N}}(\mathcal{S})$.
If they are both of type \textbf{(MB)}, then Lemma \ref{l_MB} implies
that both $\sigma$ and $\zeta$ satisfy $(\star)$ and thus $\sigma,\,\zeta\in\widehat{\mathcal{N}}(\mathcal{S})$.
Hence, we get contradiction in both cases and the proof of part (1)
is completed.\medskip{}

\noindent (2) Fix $2<i<(\sqrt{6}-2)L-1$ and let $\eta\in\{\sigma,\,\zeta\}$
be the configuration with energy $2L+i$. Since $\eta$ is cross-free,
all the bridges of $\eta$ (whose existence is guaranteed by Lemma
\ref{l_Elb}) must be of the same direction. Without loss of generality,
we suppose that all bridges of $\eta$ are horizontal. Denote these
horizontal bridges by
\[
\scal{h}_{k_{1}},\,\dots,\,\scal{h}_{k_{L-\alpha}}\;\;\;\;\text{where }1\le k_{1}<\cdots<k_{L-\alpha}\le L\;.
\]
Since $2L+i=H(\eta)$ is not a multiple of $L$ by the condition $i<(\sqrt{6}-2)L-1<L$,
at least one horizontal strip is not a bridge and hence $\alpha\ge1$.
Write
\[
\mathbb{T}_{L}\setminus\{k_{1},\,\dots,\,k_{L-\alpha}\}=\{k_{1}',\,\dots,\,k_{\alpha}'\}\;\;\;\;\text{where }1\le k_{1}'<\cdots<k_{\alpha}'\le L\;.
\]
By Lemma \ref{l_Elb}, we have that 
\begin{equation}
2L+\alpha=3L-(L-\alpha)\le H(\sigma)=2L+i\;\;\;\;\text{and hence }\alpha\le i\;.\label{e:alphabd}
\end{equation}
Define $\delta\in\mathbb{N}$ as
\begin{equation}
\delta=\sum_{\ell=1}^{\alpha}\Delta H_{\scal{h}_{k_{\ell}'}}(\eta)\ge2\alpha\;,\label{e:lbdelta}
\end{equation}
where the inequality follows since $\Delta H_{\scal{h}_{k_{\ell}'}}(\eta)\ge2$
for all $\ell\in\llbracket1,\,\alpha\rrbracket$ (cf. Lemma \ref{l_strip}).
Now, we count possible number of nice pairs for fixed $\alpha$ and
$\delta$.\medskip{}

\noindent \textbf{(Step 1)} There are ${L \choose \alpha}$ ways to
choose the positions of strips $\scal{h}_{k_{1}},\,\dots,\,\scal{h}_{k_{L-\alpha}}$.\medskip{}

\noindent \textbf{(Step 2) Number of possible spin configurations
on $\scal{h}_{k_{1}}\cup\cdots\cup\scal{h}_{k_{L-\alpha}}$:} If these
horizontal bridges have three different spins, then all the vertical
and diagonal strips have energy at least $3$, and hence by \eqref{e_Edec}
we get 
\begin{equation}
H(\eta)\ge\frac{1}{2}(0+3L+3L)=3L\;.\label{eH3L}
\end{equation}
This contradicts $H(\eta)<\sqrt{6}L$. If all these bridges are of
the same spin, there are $q$ possible choices. If all these bridges
consist of two spins, there exist $1\le u<v\le L-\alpha$ and $a_{1},\,a_{2}\in\Omega$
such that 
\begin{equation}
\scal{h}_{k_{\ell}}\text{ is an }\begin{cases}
a_{1}\text{-bridge} & \text{if }u\le\ell<v\;,\\
a_{2}\text{-bridge} & \text{otherwise},
\end{cases}\label{e:claim}
\end{equation}
since otherwise all the vertical and diagonal strips have energy at
least $4$ and we get a contradiction as in \eqref{eH3L}. Now, we
will see which values of $(u,\,v)$ are available. By counting the
number of spins in $\scal{h}_{k_{1}}\cup\cdots\cup\scal{h}_{k_{L-\alpha}}$
we should have
\[
\|\eta\|_{a_{1}}\ge2L(v-u)\;\;\;\;\text{and}\;\;\;\;\|\eta\|_{a_{2}}\ge2L(L-\alpha-v+u)\;.
\]
On the other hand, by \eqref{e_itmE}, we have $\|\eta\|_{a_{1}},\,\|\eta\|_{a_{2}}\le L^{2}+1$.
Summing these up, we get $\frac{L}{2}-\alpha\le v-u\le\frac{L}{2}$.
Therefore, there are at most 
\[
q\times(q-1)\times L\times(\alpha+1)\le2\alpha Lq^{2}
\]
ways of assigning spins on $\scal{h}_{k_{1}}\cup\cdots\cup\scal{h}_{k_{L-\alpha}}$
satisfying \eqref{e:claim}. Summing up, there are at most $q+2\alpha Lq^{2}\le3\alpha Lq^{2}$
possible choices on these strips.\medskip{}

\noindent \textbf{(Step 3) Number of possible spin configurations
on $\scal{h}_{k_{1}'}\cup\cdots\cup\scal{h}_{k_{\alpha}'}$:} Write
$\delta_{\ell}=\Delta H_{\scal{h}_{k_{\ell}'}}(\eta)$ for $\ell\in\llbracket1,\,\alpha\rrbracket$
so that $\delta_{1}+\cdots+\delta_{\alpha}=\delta$. Since each strip
$\scal{h}_{k_{\ell}'}$ has energy $\delta_{\ell}$, it should be
divided into $\delta_{\ell}$ monochromatic clusters. There are ${2L \choose \delta_{\ell}}$
ways of dividing $\scal{h}_{k_{\ell}'}\simeq\mathbb{T}_{2L}$ into
$\delta_{\ell}$ connected clusters, and there are at most $q^{\delta_{\ell}}$
ways to assign spins to these clusters. Hence, given $\alpha$ and
$\delta$, the number of possible spin choices on $\scal{h}_{k_{1}'}\cup\cdots\cup\scal{h}_{k_{\alpha}'}$
is at most 
\begin{equation}
\sum_{\delta_{1},\,\dots,\,\delta_{\alpha}\ge0:\,\delta_{1}+\cdots+\delta_{\alpha}=\delta}{2L \choose \delta_{1}}\cdots{2L \choose \delta_{\alpha}}q^{\delta_{1}+\cdots+\delta_{\alpha}}={2\alpha L \choose \delta}q^{\delta}\;.\label{e_i>1.4}
\end{equation}

\noindent \textbf{(Step 4)} Since $\eta$ is one of $\{\sigma,\,\zeta\}$
with bigger energy, we next count the number of possible other configurations.
This configuration is obtained from $\eta$ by an update which does
not increase the energy. Since updating a spin in a bridge always
increases the energy, we have to update a spin in strips $\scal{h}_{k_{\ell}'}$,
$\ell\in\llbracket1,\,\alpha\rrbracket$. For each strip, $\scal{h}_{k_{\ell}'}$
has $\delta_{\ell}$ monochromatic clusters as observed in the previous
step, and thus there are at most $2\delta_{\ell}$ updatable triangles
in this strip (which are located at the edge of each monochromatic
cluster). Hence, we have in total $2\delta$ updatable triangles.
Since each spin in the triangle can be updated to at most three spins
in order not to increase the energy, there are at most $6\delta$
possible ways of updates.

Summing \textbf{(Step 1)-(Step 4)} up, the number of possible nice
pairs for given $\alpha$ and $\delta$ is bounded from above by
\[
3\times{L \choose \alpha}\times3\alpha Lq^{2}\times{2\alpha L \choose \delta}q^{\delta}\times6\delta\;,
\]
where the first factor $3$ reflects three possible directions for
parallel bridges. Since $\Delta H_{\scal{v}_{\ell}}(\eta)$, $\Delta H_{\scal{d}_{\ell}}(\eta)\ge2$
for all $\ell\in\mathbb{T}_{L}$ by Lemma \ref{l_strip}, we can deduce
from \eqref{e_Edec} that 
\[
\delta=\sum_{\ell=1}^{\alpha}\Delta H_{\scal{h}_{k_{\ell}'}}(\eta)\le2H(\eta)-4L=2i\;.
\]
Combining with \eqref{e:lbdelta}, we get $\delta\in\llbracket2\alpha,\,2i\rrbracket$.
Thus, we can finally bound the number of nice pairs by 
\begin{equation}
\sum_{\alpha=1}^{i}\sum_{\delta=2\alpha}^{2i}3\times{L \choose \alpha}\times3\alpha Lq^{2}\times{2\alpha L \choose \delta}q^{\delta}\times6\delta\;.\label{e:sumcnt}
\end{equation}
Since $i<(\sqrt{6}-2)L-1$, the following bounds hold for $\alpha\le i$:
\[
{L \choose \alpha}\le{L \choose i}\;,\;\;\;3\alpha Lq^{2}\le3iLq^{2}\;,\;\;\;\text{and}\;\;\;{2\alpha L \choose \delta}q^{\delta}\times6\delta\le12i{2iL \choose 2i}q^{2i}\;.
\]
Therefore, we can bound the summation \eqref{e:sumcnt} from above
by 
\[
i\times2i\times3\times{L \choose i}\times3iLq^{2}\times12i{2iL \choose 2i}q^{2i}\le216Li^{4}q^{2i+2}{L \choose i}{2iL \choose 2i}\;.
\]
By Stirling's formula and the bound ${L \choose i}\le\frac{L^{i}}{i!}$,
the right-hand side of the last formula can be bounded from above
by
\[
CLi^{4}q^{2i}\times\frac{L^{i}}{i!}\times(eL)^{2i}\le C(eqL)^{3i+1}
\]
for some constant $C>0$. This concludes the proof.
\end{proof}

\subsection{\label{sec93}Computation of Dirichlet form}

In turn, we calculate the Dirichlet form $D_{\beta}(f)$ of the test
function $f=f^{A,\,B}$. To this end, we decompose $D_{\beta}(f)$
as
\begin{equation}
\Big[\,\sum_{\{\sigma,\,\zeta\}\subseteq\widehat{\mathcal{N}}(\mathcal{S})}+\sum_{\sigma\in\widehat{\mathcal{N}}(\mathcal{S})}\sum_{\zeta\in\widehat{\mathcal{N}}(\mathcal{S})^{c}}+\sum_{\{\sigma,\,\zeta\}\subseteq\widehat{\mathcal{N}}(\mathcal{S})^{c}}\,\Big]\mu_{\beta}(\sigma)c_{\beta}(\sigma,\,\zeta)[f(\zeta)-f(\sigma)]^{2}\label{e_Dbetaf}
\end{equation}
and we shall estimate three summations separately. We recall that
we are imposing the condition $L^{10}\ll e^{\beta}$ on $\beta$.
We write for each $\mathcal{A}\subseteq\mathcal{X}$,
\begin{equation}
E(\mathcal{A})=\big\{\,\{\sigma,\,\zeta\}\subseteq\mathcal{A}:\sigma\sim\zeta\,\big\}\;.\label{e_EA1}
\end{equation}

\begin{lem}
\label{l_tf1}We have
\begin{equation}
\sum_{\{\sigma,\,\zeta\}\subseteq\widehat{\mathcal{N}}(\mathcal{S})}\mu_{\beta}(\sigma)c_{\beta}(\sigma,\,\zeta)[f(\zeta)-f(\sigma)]^{2}=\frac{1+o_{L}(1)}{q\mathfrak{c}}e^{-\Gamma\beta}\;.\label{e_ltf1}
\end{equation}
\end{lem}

\begin{proof}
By Propositions \ref{p_typ} and \ref{p_EAEB}, we can decompose the
left-hand side of \eqref{e_ltf1} as
\begin{equation}
\Big[\,\sum_{\{\sigma,\,\zeta\}\in E(\mathcal{B}^{A,B})}+\sum_{\{\sigma,\,\zeta\}\in E(\mathcal{E}^{A})}+\sum_{\{\sigma,\,\zeta\}\in E(\mathcal{E}^{B})}\,\Big]\mu_{\beta}(\sigma)c_{\beta}(\sigma,\,\zeta)[f(\zeta)-f(\sigma)]^{2}\;,\label{e_tf1.1}
\end{equation}
since the test function $f$ is constant on $\mathcal{E}^{A}\cap\mathcal{B}^{A,\,B}=\mathcal{N}(\mathcal{R}_{2}^{A,\,B})$
and $\mathcal{E}^{B}\cap\mathcal{B}^{A,\,B}=\mathcal{N}(\mathcal{R}_{L-2}^{A,\,B})$
as remarked in Definition \ref{d_tf}. 

Let us consider the first summation of \eqref{e_tf1.1}. Suppose that
$\sigma\in\mathcal{D}^{A,\,B}$. If $\sigma\in\mathcal{P}_{n}^{A,\,B}$
for some $n\in\llbracket2,\,L-2\rrbracket$, $a\in A$, and $b\in B$,
then Lemma \ref{l_PP} and Definition \ref{d_tf}-(2) assert that
$f(\sigma)=f(\zeta)$. If $\sigma\in\mathcal{Q}_{n}^{a,\,b}$ for
some $n\in\llbracket2,\,L-3\rrbracket$, $a\in A$, and $b\in B$,
then Lemma \ref{l_QQ} and Definition \ref{d_tf}-(2) imply that $f(\sigma)=f(\zeta)$.
The similar conclusion holds for the case $\zeta\in\mathcal{D}^{A,\,B}$
by the same logic and hence the summand vanishes if either $\sigma\in\mathcal{D}^{A,\,B}$
or $\zeta\in\mathcal{D}^{A,\,B}$. Thus, we can write the first summation
of \eqref{e_tf1.1} as
\[
\sum_{a\in A}\sum_{b\in B}\sum_{n=2}^{L-3}\sum_{\scal{s}\in\{\scal{h},\,\scal{v},\,\scal{d}\}}\sum_{P\prec P':\,|P|=n}\sum_{\{\sigma,\,\zeta\}\in E(\mathcal{C}_{\scal{s}(P,P')}^{a,b})}\mu_{\beta}(\sigma)c_{\beta}(\sigma,\,\zeta)[f(\zeta)-f(\sigma)]^{2}\;.
\]
By Lemmas \ref{l_reg}, \ref{l_odd}, \ref{l_even}, and the definition
of $f$, the last summation on $\{\sigma,\,\zeta\}$ can be rearranged
as
\[
\sum_{\sigma\in\mathcal{C}_{\scal{s}(P,P'),\scal{o}}^{a,b},\,\zeta\in\mathcal{C}_{\scal{s}(P,P'),\scal{e}}^{a,b}:\,\sigma\sim\zeta}\mu_{\beta}(\sigma)c_{\beta}(\sigma,\,\zeta)[f(\zeta)-f(\sigma)]^{2}\;.
\]
We now decompose this summation into three parts according to the
value of $|\mathfrak{p}^{a,\,b}(\zeta)|$. First suppose that $|\mathfrak{p}^{a,\,b}(\zeta)|\neq2,\,2L-2$.
Then, by the definition of $f$, \eqref{e_detbal}, and Theorem \ref{t_Gibbs1}-(1),
the summation under this restriction equals
\[
4L\times\sum_{m=3}^{2L-4}\frac{1}{Z_{\beta}}e^{-\Gamma\beta}\times\frac{\mathfrak{b}^{2}}{\mathfrak{c}^{2}}\frac{25/4}{(5L-3)^{2}(L-4)^{2}}=\frac{50\mathfrak{b}^{2}L(L-3)}{q\mathfrak{c}^{2}(5L-3)^{2}(L-4)^{2}}\times(1+o_{L}(1))e^{-\Gamma\beta}\;.
\]
Next we suppose that $|\mathfrak{p}^{a,\,b}(\zeta)|=2$. Then the
summation under this restriction can be decomposed into 
\[
\Bigg[\,\sum_{\substack{\zeta:\,|\mathfrak{p}^{a,b}(\zeta)|=2\text{ and}\\
\mathfrak{p}^{a,b}(\zeta)\text{ is connected}
}
}+\sum_{\substack{\zeta:\,|\mathfrak{p}^{a,b}(\zeta)|=2\text{ and}\\
\mathfrak{p}^{a,b}(\zeta)\text{ is disconnected}
}
}\,\Bigg]\sum_{\sigma:\,|\mathfrak{p}^{a,b}(\sigma)|\in\{1,\,3\}}\mu_{\beta}(\sigma)c_{\beta}(\sigma,\,\zeta)[f(\zeta)-f(\sigma)]^{2}\;.
\]
By the definition of $f$, \eqref{e_detbal}, and part (1) of Theorem
\ref{t_Gibbs1}, the last display equals $(1+o_{L}(1))$ times 
\begin{align*}
 & 2L\times\frac{1}{q}e^{-\Gamma\beta}\times\frac{\mathfrak{b}^{2}}{\mathfrak{c}^{2}}\frac{9+9}{(5L-3)^{2}(L-4)^{2}}+L\times\frac{1}{q}e^{-\Gamma\beta}\times\frac{\mathfrak{b}^{2}}{\mathfrak{c}^{2}}\frac{4+4+16}{(5L-3)^{2}(L-4)^{2}}\\
 & =\frac{60\mathfrak{b}^{2}L(1+o_{L}(1))}{q\mathfrak{c}^{2}(5L-3)^{2}(L-4)^{2}}e^{-\Gamma\beta}\;.
\end{align*}
For the case $|\mathfrak{p}^{a,\,b}(\zeta)|=2L-2$, we get the same
result with the case $|\mathfrak{p}^{a,\,b}(\zeta)|=2$ by an identical
argument. Gathering the computations above and applying the definition
\eqref{e_b} of $\mathfrak{b}$, we can conclude that the first summation
of \eqref{e_tf1.1} is $(1+o_{L}(1))$ times 
\begin{align}
 & \sum_{a\in A}\sum_{b\in B}\sum_{n=2}^{L-3}\sum_{\scal{s}\in\{\scal{h},\,\scal{v},\,\scal{d}\}}\sum_{P\prec P':\,|P|=n}\frac{50\mathfrak{b}^{2}L(L-3)+60\mathfrak{b}^{2}L+60\mathfrak{b}^{2}L}{q\mathfrak{c}^{2}(5L-3)^{2}(L-4)^{2}}e^{-\Gamma\beta}\nonumber \\
 & =|A|(q-|A|)\times\frac{60\mathfrak{b}^{2}L^{2}}{q\mathfrak{c}^{2}(5L-3)(L-4)}e^{-\Gamma\beta}=\frac{\mathfrak{b}}{q\mathfrak{c}^{2}}e^{-\Gamma\beta}\;.\label{e_tf1.2}
\end{align}

Next, we turn to the second summation of \eqref{e_tf1.1}. We decompose
this summation as
\begin{align*}
 & \sum_{\{\sigma_{1},\,\sigma_{2}\}\subseteq\mathcal{O}^{A}}\mu_{\beta}(\sigma_{1})c_{\beta}(\sigma_{1},\,\sigma_{2})[f(\sigma_{2})-f(\sigma_{1})]^{2}\\
 & +\sum_{\sigma_{1}\in\mathcal{O}^{A}}\sum_{\sigma_{2}\in\mathcal{I}_{\textup{rep}}^{A}}\sum_{\zeta\in\mathcal{N}(\sigma_{2})}\mu_{\beta}(\sigma_{1})c_{\beta}(\sigma_{1},\,\zeta)[f(\zeta)-f(\sigma_{1})]^{2}\;.
\end{align*}
By Proposition \ref{p_ZA}, this equals $(1+o_{L}(1))$ times 
\[
\Big[\,\sum_{\{\sigma_{1},\,\sigma_{2}\}\subseteq\mathcal{O}^{A}}+\sum_{\sigma_{1}\in\mathcal{O}^{A}}\sum_{\sigma_{2}\in\mathcal{I}_{\textup{rep}}^{A}}\,\Big]\frac{e^{-\Gamma\beta}}{q}r^{A}(\sigma_{1},\,\sigma_{2})[f(\sigma_{2})-f(\sigma_{1})]^{2}\;.
\]
By the definition of $f$, this can be written as 
\begin{align}
 & (1+o_{L}(1))\frac{\mathfrak{e}_{A}^{2}}{\mathfrak{c}^{2}}\sum_{\{\sigma_{1},\,\sigma_{2}\}\subseteq\mathscr{V}^{A}}\frac{e^{-\Gamma\beta}}{q}r^{A}(\sigma_{1},\,\sigma_{2})[\mathfrak{h}^{A}(\sigma_{2})-\mathfrak{h}^{A}(\sigma_{1})]^{2}\nonumber \\
 & =(1+o_{L}(1))\frac{e^{-\Gamma\beta}\mathfrak{e}_{A}^{2}}{q\mathfrak{c}^{2}}\times|\mathscr{V}^{A}|\mathrm{cap}^{A}\big(\,\mathcal{S}(A),\,\mathcal{R}_{2}^{A,\,B}\,\big)=(1+o_{L}(1))\frac{\mathfrak{e}_{A}}{q\mathfrak{c}^{2}}e^{-\Gamma\beta}\;.\label{e_tf1.3}
\end{align}
In conclusion, we get 
\[
\sum_{\{\sigma,\,\zeta\}\in E(\mathcal{E}^{A})}\mu_{\beta}(\sigma)c_{\beta}(\sigma,\,\zeta)[f(\zeta)-f(\sigma)]^{2}=(1+o_{L}(1))\frac{\mathfrak{e}_{A}}{q\mathfrak{c}^{2}}e^{-\Gamma\beta}\;.
\]
By an entirely same computation, the summation $\sum_{\{\sigma,\,\zeta\}\in E(\mathcal{E}^{B})}$
yields $(1+o_{L}(1))\frac{\mathfrak{e}_{B}}{q\mathfrak{c}^{2}}e^{-\Gamma\beta}$.
Gathering these results with \eqref{e_tf1.2}, we can finally conclude
that \eqref{e_tf1.1} equals
\[
(1+o_{L}(1))\times\frac{\mathfrak{b}+\mathfrak{e}_{A}+\mathfrak{e}_{B}}{q\mathfrak{c}^{2}}e^{-\Gamma\beta}=\frac{1+o_{L}(1)}{q\mathfrak{c}}e^{-\Gamma\beta}\;,
\]
as desired. This completes the proof.
\end{proof}
\begin{lem}
\label{l_tf2}We have
\begin{equation}
\sum_{\sigma\in\widehat{\mathcal{N}}(\mathcal{S})}\sum_{\zeta\in\widehat{\mathcal{N}}(\mathcal{S})^{c}}\mu_{\beta}(\sigma)c_{\beta}(\sigma,\,\zeta)[f(\zeta)-f(\sigma)]^{2}=o_{L}(e^{-\Gamma\beta})\;.\label{e_tf21}
\end{equation}
\end{lem}

\begin{proof}
If $\sigma\in\widehat{\mathcal{N}}(\mathcal{S})$ and $\zeta\in\widehat{\mathcal{N}}(\mathcal{S})^{c}$,
we have that $H(\sigma)\le2L+2<H(\zeta)$ and therefore by \eqref{e_detbal},
we can rewrite the left-hand side of \eqref{e_tf21} as
\begin{equation}
\Big[\,\sum_{\sigma\in\mathcal{E}^{A}}+\sum_{\sigma\in\mathcal{E}^{B}}+\sum_{\sigma\in\mathcal{B}^{A,B}\setminus\mathcal{E}^{A,B}}\,\Big]\sum_{\zeta\in\widehat{\mathcal{N}}(\mathcal{S})^{c}}\mu_{\beta}(\zeta)[f(\zeta)-f(\sigma)]^{2}\;.\label{e_tf2}
\end{equation}
Let us consider the first summation.
\begin{itemize}
\item $\sigma\in\mathcal{E}^{A}$ has a cross and $\zeta\in\widehat{\mathcal{N}}(\mathcal{S})^{c}$
is adjacent to $\sigma$: By Proposition \ref{p_EAEB} and Lemma \ref{l_edge1},
$\sigma$ has an $a$-cross for some $a\in A$ and $h_{\mathcal{S}(A),\,\mathcal{R}_{2}^{A,B}}^{A}(\sigma)=1$
so that $f(\sigma)=1$ by the definition of $f$ on $\mathcal{E}^{A}$.
Moreover, by Lemma \ref{l_cross}, we have
\[
\sum_{b\in B}\|\sigma\|_{b}\le\sum_{b\in\Omega\setminus\{a\}}\|\sigma\|_{b}\le\frac{H(\sigma)^{2}}{6}\le\frac{2(L+1)^{2}}{3}\;.
\]
Since $\zeta\sim\sigma$, we have $\sum_{b\in B}\|\zeta\|_{b}\le L^{2}$
and thus $f(\zeta)=1$ by the definition of $f$. Hence, we have $f(\sigma)=f(\zeta)$
and we can neglect this case.
\item $\sigma\in\mathcal{E}^{A}$ is cross-free and $\zeta\in\widehat{\mathcal{N}}(\mathcal{S})^{c}$
is adjacent to $\sigma$: By Lemma \ref{l_cfcount}, the number of
such $\sigma$ is $O(L^{6})$. Since there are at most $2qL^{2}$
possible $\zeta\in\widehat{\mathcal{N}}(\mathcal{S})^{c}$ with $\sigma\sim\zeta$,
we obtain 
\[
\sum_{\sigma\in\mathcal{E}^{A}}\sum_{\zeta\in\widehat{\mathcal{N}}(\mathcal{S})^{c}}\mu_{\beta}(\zeta)[f(\zeta)-f(\sigma)]^{2}\le O(L^{6})\times qL^{2}\times Ce^{-(\Gamma+1)\beta}=O(L^{8}e^{-(\Gamma+1)\beta})\;.
\]
\end{itemize}
By the same logic, the second summation of \eqref{e_tf2} is $O(L^{8}e^{-(\Gamma+1)\beta})$
as well.

For the third summation of \eqref{e_tf2}, we note that\footnote{This is not an equality; consider e.g., $\xi\in\mathcal{C}_{2,\,\scal{o}}^{a,\,b}$
with $|\mathfrak{p}^{a,\,b}(\xi)|=1$ for some $a\in A$ and $b\in B$.}
\begin{align}
\mathcal{B}^{A,\,B}\setminus\mathcal{E}^{A,\,B} & \subseteq\bigcup_{n=3}^{L-3}\mathcal{R}_{n}^{A,\,B}\cup\bigcup_{n=2}^{L-3}\mathcal{C}_{n,\,\scal{o}}^{A,\,B}\cup\bigcup_{n=2}^{L-3}\mathcal{C}_{n,\,\scal{e}}^{A,\,B}\cup\mathcal{D}^{A,\,B}\;.\label{eq:bme}
\end{align}
Since $|f(\zeta)-f(\sigma)|\le1$, we have 
\[
\sum_{\zeta\in\widehat{\mathcal{N}}(\mathcal{S})^{c}:\,\zeta\sim\sigma}\mu_{\beta}(\zeta)[f(\zeta)-f(\sigma)]^{2}\le\sum_{\zeta\in\widehat{\mathcal{N}}(\mathcal{S})^{c}:\,\zeta\sim\sigma}\mu_{\beta}(\zeta)\;,
\]
and moreover by a direct computation, we get\footnote{It is enough to find the order of the number of configurations adjacent
to $\sigma$ with energy $\Gamma+1$, $\Gamma+2$, or $\Gamma+3$.
We omit tedious and elementary verification.} 
\[
\sum_{\zeta\in\widehat{\mathcal{N}}(\mathcal{S})^{c}:\,\zeta\sim\sigma}\mu_{\beta}(\zeta)=\begin{cases}
O(L^{2}e^{-(\Gamma+1)\beta}) & \text{if }\sigma\in\mathcal{R}_{n}^{A,\,B}\;,\\
O(Le^{-(\Gamma+1)\beta})+O(L^{2}e^{-(\Gamma+2)\beta}) & \text{if }\sigma\in\mathcal{C}_{n,\,\scal{o}}^{A,\,B}\;,\\
O(Le^{-(\Gamma+1)\beta})+O(Le^{-(\Gamma+2)\beta})+O(L^{2}e^{-(\Gamma+3)\beta}) & \text{if }\sigma\in\mathcal{C}_{n,\,\scal{e}}^{A,\,B}\;,\\
O(Le^{-(\Gamma+1)\beta})+O(Le^{-(\Gamma+2)\beta})+O(L^{2}e^{-(\Gamma+3)\beta}) & \text{if }\sigma\in\mathcal{D}^{A,\,B}\;.
\end{cases}
\]
Since 
\[
\sum_{n=3}^{L-3}|\mathcal{R}_{n}^{A,\,B}|=O(L^{2})\;\;\;\;\text{and}\;\;\;\;\sum_{n=2}^{L-3}\big(\,|\mathcal{C}_{n,\,\scal{o}}^{A,\,B}|+|\mathcal{C}_{n,\,\scal{e}}^{A,\,B}|\,\big)+|\mathcal{D}^{A,\,B}|=O(L^{4})\;,
\]
we can combine the computations above along with \eqref{eq:bme} to
conclude that (as $L\ll e^{\beta}$)
\[
\sum_{\sigma\in\mathcal{B}^{A,B}\setminus\mathcal{E}^{A,B}}\sum_{\zeta\in\widehat{\mathcal{N}}(\mathcal{S})^{c}}\mu_{\beta}(\sigma)c_{\beta}(\sigma,\,\zeta)[f(\zeta)-f(\sigma)]^{2}=O(L^{5}e^{-(\Gamma+1)\beta})\;.
\]
We can now complete the proof by gathering all the results so far
since
\[
\sum_{\sigma\in\widehat{\mathcal{N}}(\mathcal{S})}\sum_{\zeta\in\widehat{\mathcal{N}}(\mathcal{S})^{c}}\mu_{\beta}(\sigma)c_{\beta}(\sigma,\,\zeta)[f(\zeta)-f(\sigma)]^{2}=e^{-\Gamma\beta}\times O(L^{8}e^{-\beta})=o_{L}(e^{-\Gamma\beta})\;.
\]
\end{proof}
\begin{lem}
\label{l_tf3}We have
\begin{equation}
\sum_{\{\sigma,\,\zeta\}\subseteq\widehat{\mathcal{N}}(\mathcal{S})^{c}}\mu_{\beta}(\sigma)c_{\beta}(\sigma,\,\zeta)[f(\zeta)-f(\sigma)]^{2}=o_{L}(e^{-\Gamma\beta})\;.\label{e_tf3}
\end{equation}
\end{lem}

\begin{proof}
By Proposition \ref{p_count}-(1) and the definition of $f$ on $\widehat{\mathcal{N}}(\mathcal{S})^{c}$,
the left-hand side of \eqref{e_tf3} can be written as 
\begin{equation}
\sum_{i=3}^{3L^{2}-2L}\sum_{\substack{\{\sigma,\,\zeta\}\subseteq\widehat{\mathcal{N}}(\mathcal{S})^{c}:\\
\sigma\sim\zeta,\,\max\{H(\sigma),\,H(\zeta)\}=2L+i,\\
\sum_{a\in A}\|\sigma\|_{a}=L^{2}\,\text{and}\,\sum_{a\in A}\|\zeta\|_{a}=L^{2}-1
}
}\frac{1}{Z_{\beta}}e^{-(2L+i)\beta}\;.\label{e_tf4}
\end{equation}
By Theorem \ref{t_Gibbs1}-(1) and Proposition \ref{p_count}-(2),
the summation for $3\le i<(\sqrt{6}-2)L-1$ is bounded by 
\[
CL\times\sum_{i=3}^{\infty}(CL)^{3i}e^{-(2L+i)\beta}\le Le^{2\beta}e^{-\Gamma\beta}\sum_{i=3}^{\infty}(CL^{3}e^{-\beta})^{i}\le CL^{10}e^{-\beta}e^{-\Gamma\beta}\;,
\]
which equals $o_{L}(e^{-\Gamma\beta})$. We emphasize that\textbf{
this is the location where the condition $L^{10}\ll e^{\beta}$ is
crucially used.}

Next, by Lemma \ref{l_Xi}, there exists a positive integer $\theta$
such that 
\[
|\mathcal{X}_{2L+i}|\le q^{2L+i+1}\sum_{\substack{n_{3},\,n_{4},\,n_{5},\,n_{6}\ge0:\\
3n_{3}+4n_{4}+5n_{5}+6n_{6}=2L+i
}
}{\theta L^{2} \choose n_{3}}{\theta L^{2} \choose n_{4}}{\theta L^{2} \choose n_{5}}{\theta L^{2} \choose n_{6}}\;,
\]
and thus by Theorem \ref{t_Gibbs1}-(1), the summation \eqref{e_tf4}
for $i\ge(\sqrt{6}-2)L-1$ is bounded from above by
\[
2qL^{2}\sum_{j=\lfloor\sqrt{6}L\rfloor}^{3L^{2}}q^{j+1}\sum_{\substack{n_{3},\,n_{4},\,n_{5},\,n_{6}\ge0:\\
3n_{3}+4n_{4}+5n_{5}+6n_{6}=j
}
}{\theta L^{2} \choose n_{3}}{\theta L^{2} \choose n_{4}}{\theta L^{2} \choose n_{5}}{\theta L^{2} \choose n_{6}}e^{-\beta j}\;,
\]
where the factor $2qL^{2}$ comes from the trivial bound on the number
of possible $\zeta$ (resp. $\sigma)$ given $\sigma$ (resp. $\zeta$).
Using ${\alpha \choose \beta}{\gamma \choose \delta}\le{\alpha+\gamma \choose \beta+\delta}$,
we bound this by
\begin{equation}
2q^{2}L^{2}\sum_{j=\lfloor\sqrt{6}L\rfloor}^{3L^{2}}(qe^{-\beta})^{j}\sum_{\substack{n_{3},\,n_{4},\,n_{5},\,n_{6}\ge0:\\
3n_{3}+4n_{4}+5n_{5}+6n_{6}=j
}
}{4\theta L^{2} \choose n_{3}+n_{4}+n_{5}+n_{6}}\;.\label{e_tf3.2}
\end{equation}
Since $n_{3}+n_{4}+n_{5}+n_{6}\le\frac{1}{3}(3n_{3}+4n_{5}+5n_{5}+6n_{6})=\frac{j}{3}\le L^{2}$,
and since $\theta>1$, the last summation is bounded from above by
\[
\sum_{\substack{n_{3},\,n_{4},\,n_{5},\,n_{6}\ge0:\\
3n_{3}+4n_{4}+5n_{5}+6n_{6}=j
}
}{4\theta L^{2} \choose \lfloor\frac{j}{3}\rfloor}\le{4\theta L^{2} \choose \lfloor\frac{j}{3}\rfloor}\times CL^{6}\le CL^{6}(4\theta L^{2})^{j/3}\;
\]
for some positive constant $C$. Hence, \eqref{e_tf3.2} is bounded
from above by
\[
CL^{8}\sum_{j>\sqrt{6}L-1}(qe^{-\beta})^{j}(4\theta L^{2})^{j/3}=CL^{8}\sum_{j>\sqrt{6}L-1}(CL^{2/3}e^{-\beta})^{j}\;.
\]
Since $L^{2/3}e^{-\beta}\ll e^{-\frac{14}{15}\beta}$ by the condition\textbf{
$L^{10}\ll e^{\beta}$}, we can further bound the right-hand side
by \textbf{
\[
CL^{8}(Ce^{-\frac{14}{15}\beta})^{\sqrt{6}L-1}=o_{L}(e^{-\Gamma\beta})
\]
}since $\frac{14}{15}\sqrt{6}>2$.
\end{proof}
Finally, we can now conclude the proof of Proposition \ref{p_tf}.
\begin{proof}[Proof of Proposition \ref{p_tf}]
 The fact that $f\in\mathfrak{C}(\mathcal{S}(A),\,\mathcal{S}(B))$
is immediate from the construction of $f$ on $\mathcal{E}^{A,\,B}$.
The estimate of $D_{\beta}(f)$ follows from the decomposition \eqref{e_Dbetaf}
and Lemmas \ref{l_tf1}, \ref{l_tf2}, and \ref{l_tf3}.
\end{proof}
\begin{rem}
Careful reading of the proof reveals that Lemmas \ref{l_tf1}, \ref{l_tf2},
and \ref{l_tf3} hold under the conditions $L^{2/3}\ll e^{\beta}$
(the optimal one in view of Theorem \ref{t_Gibbs1}), $L^{8}\ll e^{\beta}$,
and $L^{10}\ll e^{\beta}$, respectively. This shows that the sub-optimality
of our result comes essentially from our ignorance on the behavior
of the process $\sigma_{\beta}(\cdot)$ outside $\widehat{\mathcal{N}}(\mathcal{S})$.
\end{rem}

\section{\label{sec10_LB}Lower Bound for Capacities}

The purpose of this section is to establish a suitable test flow to
apply the generalized Thomson principle (Theorem \ref{t_gTP}). This
yields the lower bound for the capacity compensating the upper bound
obtained in the previous section. At the end of the current section,
the proof of Theorem \ref{t_cap} will be finally presented. We remark
that Notation \ref{n_sec910} will be consistently used in the current
section as well.

\subsection{Construction of test flow}

In view of Theorem \ref{t_gTP}, the test flow should approximate
the flow $c\Psi_{h_{\mathcal{S}(A),\mathcal{S}(B)}^{\beta}}$ where
$h_{\mathcal{S}(A),\,\mathcal{S}(B)}^{\beta}$ denotes the equilibrium
potential between $\mathcal{S}(A)$ and $\mathcal{S}(B)$. We provide
this approximation below. Since we already know the approximation
of $h_{\mathcal{S}(A),\,\mathcal{S}(B)}^{\beta}$ from Definition
\ref{d_tf}, the construction of the test flow follows naturally from
it. 
\begin{defn}[Test flow]
\label{d_tfl} Recall the test function $f=f^{A,\,B}$ constructed
in Definition \ref{d_tf}. We define the test flow $\psi=\psi^{A,\,B}$
by (cf. Notation \ref{n_92})

\[
\psi(\sigma,\,\zeta)=\begin{cases}
\mu_{\beta}(\sigma)c_{\beta}(\sigma,\,\zeta)[f(\sigma)-f(\zeta)] & \text{if }\sigma,\,\zeta\in\mathcal{B}^{A,\,B}\;,\\
\frac{\mathfrak{e}_{A}}{Z_{\beta}\mathfrak{c}}e^{-\Gamma\beta}\times[\mathfrak{h}^{A}(\sigma)-\mathfrak{h}^{A}(\zeta)] & \text{if }\sigma,\,\zeta\in\mathcal{E}^{A}\text{ with }\sigma\sim\zeta\;,\\
\frac{\mathfrak{e}_{B}}{Z_{\beta}\mathfrak{c}}e^{-\Gamma\beta}\times[\mathfrak{h}^{B}(\zeta)-\mathfrak{h}^{B}(\sigma)] & \text{if }\sigma,\,\zeta\in\mathcal{E}^{B}\text{ with }\sigma\sim\zeta\;,\\
0 & \text{otherwise}.
\end{cases}
\]
The well-definedness of the definition on $\mathcal{N}(\mathcal{R}_{2}^{A,\,B})$
and on $\mathcal{N}(\mathcal{R}_{L-2}^{A,\,B})$ should be carefully
addressed. It can be checked by noting that, for $\sigma,\,\zeta\in\mathcal{N}(\mathcal{R}_{2}^{A,\,B})$
(resp. $\mathcal{N}(\mathcal{R}_{L-2}^{A,\,B})$), the definitions
of $\psi$ on $\mathcal{B}^{A,\,B}$ and $\mathcal{E}^{A,\,B}$ both
imply that $\psi(\sigma,\,\zeta)=0$ since we have $f(\sigma)=f(\zeta)$
as mentioned in Definition \ref{d_tf} and $\mathfrak{h}^{A}(\sigma)=\mathfrak{h}^{A}(\zeta)$
(resp. $\mathfrak{h}^{B}(\sigma)=\mathfrak{h}^{B}(\zeta)$) as mentioned
in Notation \ref{n_92}.
\end{defn}

In the remainder of the current section, we shall prove the following
proposition. 
\begin{prop}
\label{p_tfl}For the test flow $\psi=\psi^{A,\,B}$ constructed in
the previous definition, it holds that 
\begin{equation}
\frac{1}{\|\psi\|_{\beta}^{2}}\Big[\,\sum_{\sigma\in\mathcal{X}}h_{\mathcal{S}(A),\,\mathcal{S}(B)}^{\beta}(\sigma)(\mathrm{div}\,\psi)(\sigma)\,\Big]^{2}=\frac{1+o_{L}(1)}{q\mathfrak{c}}e^{-\Gamma\beta}\;.\label{e_tfl}
\end{equation}
\end{prop}

The proof of this proposition is divided into two steps. First, we
have to compute the flow norm $\|\psi\|_{\beta}^{2}$. This can be
done by a direct computation with our explicit construction of the
test flow $\psi$ and will be presented in Section \ref{sec102}.
Then, it remains to compute the summation appeared in the left-hand
side \eqref{e_tfl}. To that end, we have to suitably estimate $h_{\mathcal{S}(A),\,\mathcal{S}(B)}^{\beta}(\sigma)$
and then compute the divergence term $(\mathrm{div}\,\psi)(\sigma)$.
This will be done in Section \ref{sec103}. Finally, in Section \ref{sec104},
we shall conclude the proof of Proposition \ref{p_tfl} as well as
the proof of Theorem \ref{t_cap}. 

The main issue in the large volume regime in the construction of the
test function carried out in the previous section is the construction
on $\widehat{\mathcal{N}}(\mathcal{S})^{c}$. However, in the test
flow, we do not encounter this sort of difficulty as we simply assign
zero flow on this remainder set. Instead, an additional difficulty,
compared to the small volume regime, appears in the control of $h_{\mathcal{S}(A),\,\mathcal{S}(B)}^{\beta}(\sigma)$.

\subsection{\label{sec102}Flow norm of $\psi$}
\begin{prop}
\label{p_tfl1}For $L^{2/3}\ll e^{\beta}$, it holds that 
\[
\|\psi\|_{\beta}^{2}=\frac{1+o_{L}(1)}{q\mathfrak{c}}e^{-\Gamma\beta}\;.
\]
\end{prop}

\begin{proof}
The strategy is to compare the flow norm with the Dirichlet form of
$f$. Since $\psi\equiv0$ on $\mathcal{B}^{A,\,B}\cap\mathcal{E}^{A}$
and $\mathcal{B}^{A,\,B}\cap\mathcal{E}^{A}$ as mentioned in Definition
\ref{d_tfl}, we can write
\begin{equation}
\|\psi\|_{\beta}^{2}=\Big[\,\sum_{\{\sigma,\,\zeta\}\subseteq\mathcal{B}^{A,B}}+\sum_{\{\sigma,\,\zeta\}\subseteq\mathcal{E}^{A}}+\sum_{\{\sigma,\,\zeta\}\subseteq\mathcal{E}^{B}}\,\Big]\frac{\psi(\sigma,\,\zeta)^{2}}{\mu_{\beta}(\sigma)c_{\beta}(\sigma,\,\zeta)}\;.\label{e_tfl1.1}
\end{equation}
Let us consider three summations separately. 
\begin{itemize}
\item By the definition of $\psi$ on $\mathcal{B}^{A,\,B}$, we can write
the first summation as 
\begin{equation}
\sum_{\{\sigma,\,\zeta\}\subseteq\mathcal{B}^{A,B}}\mu_{\beta}(\sigma)c_{\beta}(\sigma,\,\zeta)[f^{A,\,B}(\sigma)-f^{A,\,B}(\zeta)]^{2}\;.\label{e_tfl1.2}
\end{equation}
\item By the definition of $\psi$ on $\mathcal{E}^{A}$, the second summation
equals
\[
\sum_{\{\sigma,\,\zeta\}\subseteq\mathcal{E}^{A}}\frac{1}{\mu_{\beta}(\sigma)c_{\beta}(\sigma,\,\zeta)}\times\frac{\mathfrak{e}_{A}^{2}}{Z_{\beta}^{2}\mathfrak{c}^{2}}e^{-2\Gamma\beta}\times[\mathfrak{h}^{A}(\sigma)-\mathfrak{h}^{A}(\zeta)]^{2}\;.
\]
Note from Notation \ref{n_92} that $\{\sigma,\,\zeta\}\subseteq\mathcal{E}^{A}$
with $\mathfrak{h}^{A}(\sigma)\ne\mathfrak{h}^{A}(\zeta)$ implies
$\max\{H(\sigma),\,H(\zeta)\}=\Gamma$. Thus, by \eqref{e_detbal},
we can rewrite the last summation as 
\begin{equation}
\sum_{\{\sigma,\,\zeta\}\subseteq\mathcal{E}^{A}}\mu_{\beta}(\sigma)c_{\beta}(\sigma,\,\zeta)\times\frac{\mathfrak{e}_{A}^{2}}{\mathfrak{c}^{2}}[\mathfrak{h}^{A}(\sigma)-\mathfrak{h}^{A}(\zeta)]^{2}=\sum_{\{\sigma,\,\zeta\}\subseteq\mathcal{E}^{A}}\mu_{\beta}(\sigma)c_{\beta}(\sigma,\,\zeta)[f(\sigma)-f(\zeta)]^{2}\;.\label{e_tfl1.3}
\end{equation}
Similarly, the third summation equals
\begin{equation}
\sum_{\{\sigma,\,\zeta\}\subset\mathcal{E}^{B}}\mu_{\beta}(\sigma)c_{\beta}(\sigma,\,\zeta)[f(\sigma)-f(\zeta)]^{2}\;.\label{e_tfl1.4}
\end{equation}
\end{itemize}
Gathering \eqref{e_tf1.1}, \eqref{e_tf1.2}, \eqref{e_tf1.3}, and
\eqref{e_tfl1.4}, we can conclude that
\[
\|\psi\|_{\beta}^{2}=\sum_{\{\sigma,\,\zeta\}\subseteq\widehat{\mathcal{N}}(\mathcal{S})}\mu_{\beta}(\sigma)c_{\beta}(\sigma,\,\zeta)[f(\sigma)-f(\zeta)]^{2}\;.
\]
The right-hand side is $\frac{1+o_{L}(1)}{q\mathfrak{c}}e^{-\Gamma\beta}$
by Lemma \ref{l_tf1} and the proof is completed.
\end{proof}

\subsection{\label{sec103}Divergence of $\psi$}

Next, we compute the summation appeared in \eqref{e_tfl}. More precisely,
we wish to prove the following proposition in this section. 
\begin{prop}
\label{p_tfl2}Suppose that $L^{10}\ll e^{\beta}$. Then, we have
that 
\begin{equation}
\sum_{\sigma\in\mathcal{X}}h_{\mathcal{S}(A),\,\mathcal{S}(B)}^{\beta}(\sigma)(\mathrm{div}\,\psi)(\sigma)=\frac{1}{q\mathfrak{c}}e^{-\Gamma\beta}+o_{L}(e^{-\Gamma\beta})\;.\label{e_tfl2}
\end{equation}
\end{prop}

The proof is divided into several lemmas. We first look at the divergence
term $(\mathrm{div}\,\psi)(\sigma)$. We deduce that this divergence
is zero at most of the bulk configurations. 
\begin{lem}
\label{l_div1}We have $(\mathrm{div}\,\psi)(\sigma)=0$ if
\begin{enumerate}
\item $\sigma\in\mathcal{D}^{A,\,B}$,
\item $\sigma\in\mathcal{C}_{n,\,\scal{e}}^{a,\,b}$ for some $a\in A$,
$b\in B$ and $n\in\llbracket2,\,L-3\rrbracket$, and 
\item $\sigma\in\mathcal{C}_{n,\,\scal{o}}^{a,\,b}$ with $|\mathfrak{p}^{a,\,b}(\sigma)|\in\llbracket3,\,2L-3\rrbracket$
for some $a\in A$, $b\in B$ and $n\in\llbracket2,\,L-3\rrbracket$.
\end{enumerate}
\end{lem}

\begin{proof}
(1) By the definition of $f$ on $\mathcal{D}^{a,\,b}$ in Definition
\ref{d_tf}, we have $f(\sigma)=f(\zeta)$ for all $\zeta\sim\sigma$
with $\zeta\in\widehat{\mathcal{N}}(\mathcal{S})$. Recalling the
definition of $\psi$ in Definition \ref{d_tfl}, we have $\psi(\sigma,\,\zeta)=0$
for all $\sigma\in\mathcal{D}^{a,\,b}$ and $\zeta\sim\sigma$, and
we are done.\medskip{}

\noindent (2) For $\sigma\in\mathcal{C}_{n,\,\scal{e}}^{a,\,b}$ with
$n\in\llbracket2,\,L-3\rrbracket$, by Lemma \ref{l_even} and \eqref{e_detbal},
we can write 
\[
(\mathrm{div}\,\psi)(\sigma)=\frac{1}{Z_{\beta}}e^{-\Gamma\beta}\times\sum_{\zeta\in\mathcal{C}_{n,\scal{o}}^{a,b}:\,\zeta\sim\sigma}[f(\sigma)-f(\zeta)]\;.
\]
The last summation can be computed as 
\[
\begin{cases}
\frac{\mathfrak{b}}{\mathfrak{c}}\times\big[\,\frac{\frac{5}{2}-\frac{5}{2}}{(5L-3)(L-4)}\,\big]=0 & \text{if }|\mathfrak{p}^{a,\,b}(\sigma)|\in\llbracket4,\,2L-4\rrbracket\;,\\
\frac{\mathfrak{b}}{\mathfrak{c}}\times\big[\,\frac{3-3}{(5L-3)(L-4)}\,\big]=0 & \text{if }|\mathfrak{p}^{a,\,b}(\sigma)|=2\text{ and }\mathfrak{p}^{a,\,b}(\sigma)\text{ is connected},\\
\frac{\mathfrak{b}}{\mathfrak{c}}\times\big[\,\frac{4-2-2}{(5L-3)(L-4)}\,\big]=0 & \text{if }|\mathfrak{p}^{a,\,b}(\sigma)|=2\text{ and }\mathfrak{p}^{a,\,b}(\sigma)\text{ is disconnected},
\end{cases}
\]
and we can similarly handle the case $|\mathfrak{p}^{a,\,b}(\sigma)|=2L-2$.
This proves part (2). \medskip{}

\noindent (3) For $\sigma\in\mathcal{C}_{n,\,\scal{o}}^{a,\,b}$ with
$|\mathfrak{p}^{a,\,b}(\sigma)|\in\llbracket3,\,2L-3\rrbracket$ and
$n\in\llbracket2,\,L-3\rrbracket$, by Lemma \ref{l_odd} and \eqref{e_detbal},
we can write 
\[
(\mathrm{div}\,\psi)(\sigma)=\frac{1}{Z_{\beta}}e^{-\Gamma\beta}\times\sum_{\zeta\in\mathcal{C}_{n,\scal{e}}^{a,b}\cup\mathcal{Q}_{n}^{a,b}:\,\zeta\sim\sigma}[f(\sigma)-f(\zeta)]\;.
\]
For $\zeta\in\mathcal{Q}_{n}^{a,\,b}$, the summation vanishes by
Lemma \ref{l_QQ} and Definition \ref{d_tf}-(2). For $\zeta\in\mathcal{C}_{n,\,\scal{e}}^{a,\,b}$,
the last summation is calculated as
\begin{equation}
\frac{\mathfrak{b}}{\mathfrak{c}}\times\Big[\,\frac{\frac{5}{2}\times4-\frac{5}{2}\times4}{(5L-3)(L-4)}\,\Big]=0\;.\label{e_div1}
\end{equation}
This concludes the proof.
\end{proof}
In the previous lemma, it has been shown that the divergence of $\psi$
is zero on all bulk configurations except in $\mathcal{N}(\mathcal{R}_{n}^{A,\,B})$
with $n\in\llbracket2,\,L-2\rrbracket$. We next show that the divergences
on these sets are canceled out with each other.
\begin{lem}
\label{l_div2}Let $\zeta\in\mathcal{R}_{n}^{A,\,B}$ for some $n\in\llbracket2,\,L-2\rrbracket$.
Then, we have 
\[
\sum_{\sigma\in\mathcal{N}(\zeta)}(\mathrm{div}\,\psi)(\sigma)=0\;.
\]
\end{lem}

\begin{proof}
Let $a\in A$, $b\in B$ and $n\in\llbracket2,\,L-2\rrbracket$ and
then fix $\zeta\in\mathcal{R}_{n}^{a,\,b}$. 

First, by definitions of $f$ and $\psi$, we can readily deduce that
$\psi(\zeta,\,\xi)=0$ for all $\xi\in\mathcal{X}$ and therefore
we immediately have $(\mathrm{div}\,\psi)(\zeta)=0$. Next let $\sigma\in\mathcal{N}(\zeta)\setminus\{\zeta\}$
so that, by Lemma \ref{l_reg}, 
\begin{equation}
\begin{cases}
\sigma\in\mathcal{C}_{n,\,\scal{o}}^{a,\,b}\text{ with }|\mathfrak{p}^{a,\,b}(\sigma)|=1\text{ or}\\
\sigma\in\mathcal{C}_{n-1,\,\scal{o}}^{a,\,b}\text{ with }|\mathfrak{p}^{a,\,b}(\sigma)|=2L-1\;.
\end{cases}\label{epr2L-1}
\end{equation}

\noindent \textbf{(Case 1: $n\in\llbracket3,\,L-3\rrbracket$)} By
\eqref{e_detbal} and explicit definitions of $f$ and $\psi$, we
can check through elementary computations that 
\begin{equation}
(\mathrm{div}\,\psi)(\sigma)=\begin{cases}
\frac{1}{Z_{\beta}}\frac{10\mathfrak{b}e^{-\Gamma\beta}}{\mathfrak{c}(5L-3)(L-4)} & \text{if }\sigma\in\mathcal{C}_{n,\,\scal{o}}^{a,\,b}\text{ with }|\mathfrak{p}^{a,\,b}(\sigma)|=1\;,\\
-\frac{1}{Z_{\beta}}\frac{10\mathfrak{b}e^{-\Gamma\beta}}{\mathfrak{c}(5L-3)(L-4)} & \text{if }\sigma\in\mathcal{C}_{n-1,\,\scal{o}}^{a,\,b}\text{ with }|\mathfrak{p}^{a,\,b}(\sigma)|=2L-1\;.
\end{cases}\label{e_divps}
\end{equation}
Hence, we have 
\[
\sum_{\sigma\in\mathcal{N}(\zeta)}(\mathrm{div}\,\psi)(\sigma)=0+\Big[\,\frac{1}{Z_{\beta}}\frac{10\mathfrak{b}e^{-\Gamma\beta}}{\mathfrak{c}(5L-3)(L-4)}-\frac{1}{Z_{\beta}}\frac{10\mathfrak{b}e^{-\Gamma\beta}}{\mathfrak{c}(5L-3)(L-4)}\,\Big]\times2L=0\;.
\]

\noindent \textbf{(Case 2: $n=2$ or $L-2$)} First, we let $n=2$.
By the same computation above, we can check 
\begin{equation}
(\mathrm{div}\,\psi)(\sigma)=\frac{1}{Z_{\beta}}\frac{10\mathfrak{b}e^{-\Gamma\beta}}{\mathfrak{c}(5L-3)(L-4)}\;\;\;\;\text{if }\sigma\in\mathcal{C}_{2,\,\scal{o}}^{a,\,b}\text{ with }|\mathfrak{p}^{a,\,b}(\sigma)|=1\;.\label{ediv1}
\end{equation}
On the other hand, by the definition of $\psi$ on $\mathcal{E}^{A}$,
we can write
\[
\sum_{\sigma\in\mathcal{C}_{1,\scal{o}}^{a,b}\cap\mathcal{N}(\zeta)}(\mathrm{div}\,\psi)(\sigma)=\sum_{\sigma\in\mathcal{C}_{1,\scal{o}}^{a,b}\cap\mathcal{N}(\zeta)}\sum_{\xi\in\mathcal{O}^{A}:\,\xi\sim\sigma}\psi(\sigma,\,\xi)=\sum_{\sigma\in\mathcal{N}(\zeta)}\sum_{\xi\in\mathcal{O}^{A}:\,\xi\sim\sigma}\psi(\sigma,\,\xi)\;,
\]
where the first equality holds since, for $\sigma\in\mathcal{C}_{1,\,\scal{o}}^{a,\,b}\cap\mathcal{N}(\zeta)$,
we have $\psi(\sigma,\,\xi)=0$ unless $\xi\in\mathcal{O}^{A}$, and
the second equality holds since the configurations in $\mathcal{C}_{2,\,\scal{o}}^{a,\,b}\cup\{\zeta\}$
is not connected with $\mathcal{O}^{A}$. By \eqref{e_detbal}, we
can write 
\begin{align*}
\sum_{\sigma\in\mathcal{N}(\zeta)}\sum_{\xi\in\mathcal{O}^{A}:\,\xi\sim\sigma}\psi(\sigma,\,\xi) & =\sum_{\sigma\in\mathcal{N}(\zeta)}\sum_{\xi\in\mathcal{O}^{A}:\,\xi\sim\sigma}\frac{\mathfrak{e}_{A}}{Z_{\beta}\mathfrak{c}}e^{-\Gamma\beta}[\mathfrak{h}^{A}(\sigma)-\mathfrak{h}^{A}(\xi)]\\
 & =\frac{\mathfrak{e}_{A}}{Z_{\beta}\mathfrak{c}}e^{-\Gamma\beta}\sum_{\xi\in\mathcal{O}^{A}}r^{A}(\zeta,\,\xi)[\mathfrak{h}^{A}(\zeta)-\mathfrak{h}^{A}(\xi)]=-\frac{\mathfrak{e}_{A}}{Z_{\beta}\mathfrak{c}}e^{-\Gamma\beta}\times(L^{A}\mathfrak{h}^{A})(\zeta)\;,
\end{align*}
where the second equality follows from Notation \eqref{n_92} and
the definition of $r^{A}$ (cf. \eqref{e_rA}). By the property of
capacities (e.g. \cite[Lemmas 7.7 and 7.12]{BdH}) and the definition
of $\mathfrak{e}_{A}$ (cf. \eqref{e_eA}), we get 
\begin{equation}
\sum_{\zeta\in\mathcal{R}_{2}^{A,B}}(L^{A}\mathfrak{h}^{A})(\zeta)=|\mathscr{V}^{A}|\mathrm{cap}^{A}\big(\,\mathcal{S}(A),\,\mathcal{R}_{2}^{A,\,B}\,\big)=\frac{1}{\mathfrak{e}_{A}}\;,\label{eLh}
\end{equation}
and therefore by symmetry, we get 
\[
(L^{A}\mathfrak{h}^{A})(\zeta)=\frac{1}{|\mathcal{R}_{2}^{A,\,B}|\mathfrak{e}_{A}}=\frac{1}{3L|A|(q-|A|)\mathfrak{e}_{A}}\;,
\]
where the factor $3$ comes from three possible directions. By gathering
the computations above, we can conclude that 
\begin{equation}
\sum_{\sigma\in\mathcal{C}_{1,\scal{o}}^{a,b}\cap\mathcal{N}(\zeta)}(\mathrm{div}\,\psi)(\sigma)=\frac{1}{Z_{\beta}\mathfrak{c}}e^{-\Gamma\beta}\times\frac{1}{3L|A|(q-|A|)}\;.\label{ediv2}
\end{equation}
By \eqref{ediv1} and \eqref{ediv2}, Theorem \ref{t_Gibbs1}-(1),
and by recalling the definitions \eqref{e_b} of $\mathfrak{b}$ and
$\mathfrak{c}$, we finally get 
\[
\sum_{\sigma\in\mathcal{N}(\zeta)}(\mathrm{div}\,\psi)(\sigma)=\frac{1}{Z_{\beta}}\frac{10\mathfrak{b}e^{-\Gamma\beta}}{\mathfrak{c}(5L-3)(L-4)}\times2L-\frac{1}{Z_{\beta}\mathfrak{c}}e^{-\Gamma\beta}\times\frac{1}{3L|A|(q-|A|)}=0\;.
\]
Since the proof for the case $n=L-2$ is identical, the proof is completed.
\end{proof}
Next, we turn to the divergences of $\psi$ on the edge typical configurations.
\begin{lem}
\label{l_div3}For all $\sigma\in\mathcal{O}^{A}\cup\mathcal{O}^{B}$,
we have $(\mathrm{div}\,\psi)(\sigma)=0$.
\end{lem}

\begin{proof}
We only consider the case $\sigma\in\mathcal{O}^{A}$ since the proof
for $\mathcal{O}^{B}$ is identical. By definition of $\psi$, we
can write 
\begin{align*}
(\mathrm{div}\,\psi)(\sigma)=\sum_{\zeta\in\mathcal{E}^{A}:\,\zeta\sim\sigma}\psi(\sigma,\,\zeta) & =\sum_{\zeta\in\mathcal{E}^{A}:\,\zeta\sim\sigma}\frac{\mathfrak{e}_{A}}{Z_{\beta}\mathfrak{c}}e^{-\Gamma\beta}\times\big[\,\mathfrak{h}^{A}(\sigma)-\mathfrak{h}^{A}(\zeta)\,\big]\\
 & =-\frac{\mathfrak{e}_{A}}{Z_{\beta}\mathfrak{c}}e^{-\Gamma\beta}\times(L^{A}\mathfrak{h}^{A})(\sigma)\;.
\end{align*}
Since $\mathcal{O}^{A}\subseteq\mathcal{E}^{A}\setminus(\mathcal{S}(A)\cup\mathcal{R}_{2}^{A,\,B})$,
we have $(L^{A}\mathfrak{h}^{A})(\sigma)=(L^{A}h_{\mathcal{S}(A),\,\mathcal{R}_{2}^{A,B}}^{A})(\sigma)=0$
by the elementary property of equilibrium potentials. This completes
the proof.
\end{proof}
\begin{lem}
\label{l_div4}For $a\in A$, $b\in B$, and $\sigma\in\mathcal{C}_{1,\,\scal{o}}^{a,\,b}\cup\mathcal{C}_{L-2,\,\scal{o}}^{a,\,b}$
with $|\mathfrak{p}^{a,\,b}(\sigma)|\in\llbracket3,\,2L-3\rrbracket$,
we have $(\mathrm{div}\,\psi)(\sigma)=0$.
\end{lem}

\begin{proof}
By symmetry, we may assume $\sigma\in\mathcal{C}_{1,\,\scal{o}}^{a,\,b}$
and $|\mathfrak{p}^{a,\,b}(\sigma)|\in\llbracket3,\,2L-3\rrbracket$.
Then as $\mathcal{N}(\sigma)=\{\sigma\}$, we may write
\begin{align*}
(\mathrm{div}\,\psi)(\sigma)=\sum_{\zeta\in\mathcal{E}^{A}:\,\zeta\sim\sigma}\psi(\sigma,\,\zeta) & =\sum_{\zeta\in\mathcal{O}^{A}:\,\zeta\sim\sigma}\frac{\mathfrak{e}_{A}}{Z_{\beta}\mathfrak{c}}e^{-\Gamma\beta}\times\big[\,\mathfrak{h}^{A}(\sigma)-\mathfrak{h}^{A}(\zeta)\,\big]\\
 & =-\frac{\mathfrak{e}_{A}}{Z_{\beta}\mathfrak{c}}e^{-\Gamma\beta}\times(L^{A}\mathfrak{h}^{A})(\sigma)\;.
\end{align*}
Since $\sigma\notin\mathcal{S}(A)\cup\mathcal{R}_{2}^{A,\,B}$, we
again have $(L^{A}\mathfrak{h}^{A})(\sigma)=(L^{A}h_{\mathcal{S}(A),\,\mathcal{R}_{2}^{A,B}}^{A})(\sigma)=0$
and the proof is completed.
\end{proof}
\begin{lem}
\label{l_div5}We have $(\mathrm{div}\,\psi)(\sigma)=0$ for all $\sigma\in\mathcal{N}(\zeta)$
with $\zeta\in\mathcal{I}_{\textup{rep}}^{A}\setminus(\mathcal{S}(A)\cup\mathcal{R}_{2}^{A,\,B}\cup\mathcal{C}_{1,\,\scal{o}}^{A,\,B})$.
\end{lem}

\begin{proof}
For all $\xi\in\widehat{\mathcal{N}}(\mathcal{S})$ with $\sigma\sim\xi$,
by Lemma \ref{l_edge3}, both $\sigma$ and $\xi$ have an $a$-cross
for some $a\in A$. Therefore, we have by Lemma \ref{l_edge1} that
$\mathfrak{h}^{A}(\sigma)=\mathfrak{h}^{A}(\xi)=1$ and therefore
we have $\psi(\sigma,\,\xi)=0$. This concludes the proof.
\end{proof}
\begin{lem}
\label{l_div6}We have
\[
\sum_{\sigma\in\mathcal{N}(\mathcal{S}(A))}(\mathrm{div}\,\psi)(\sigma)=\frac{1}{Z_{\beta}\mathfrak{c}}e^{-\Gamma\beta}\;\;\;\;\text{and}\;\;\;\;\sum_{\sigma\in\mathcal{N}(\mathcal{S}(B))}(\mathrm{div}\,\psi)(\sigma)=-\frac{1}{Z_{\beta}\mathfrak{c}}e^{-\Gamma\beta}\;.
\]
\end{lem}

\begin{proof}
We focus only on the first one since the proof for the second one
is identical. As in the previous proof, we can write 
\begin{align}
\sum_{\sigma\in\mathcal{N}(\mathcal{S}(A))}(\mathrm{div}\,\psi)(\sigma) & =\sum_{\sigma\in\mathcal{S}(A)}\sum_{\zeta\in\mathcal{E}^{A}:\,\zeta\sim\sigma}\frac{\mathfrak{e}_{A}}{Z_{\beta}\mathfrak{c}}e^{-\Gamma\beta}\times\big[\,\mathfrak{h}^{A}(\sigma)-\mathfrak{h}^{A}(\zeta)\,\big]\nonumber \\
 & =-\frac{\mathfrak{e}_{A}}{Z_{\beta}\mathfrak{c}}e^{-\Gamma\beta}\times\sum_{\sigma\in\mathcal{S}(A)}(L^{A}\mathfrak{h}^{A})(\sigma)\;.\label{eha2}
\end{align}
By the same reasoning with \eqref{eLh}, we have that 
\[
\sum_{\sigma\in\mathcal{S}(A)}(L^{A}\mathfrak{h}^{A})(\sigma)=-|\mathscr{V}^{A}|\mathrm{cap}^{A}\big(\,\mathcal{S}(A),\,\mathcal{R}_{2}^{A,\,B}\,\big)=-\frac{1}{\mathfrak{e}_{A}}\;,
\]
and injecting this to \eqref{eha2} completes the proof.
\end{proof}
Since we only have the control on the summation of divergences in
the $\mathcal{N}$-neighborhoods of ground states or regular configurations,
we need the following flatness result on the equilibrium potential
$h_{\mathcal{S}(A),\,\mathcal{S}(B)}^{\beta}$ on these $\mathcal{N}$-neighborhoods
to control the summation at the left-hand side of \eqref{e_tfl2}. 
\begin{lem}
\label{l_eqp}There exists $C>0$ such that the following results
hold. 
\begin{enumerate}
\item For $\mathbf{s}\in\mathcal{S}$ and $\sigma\in\mathcal{N}(\mathbf{s})$,
denote by $N_{\sigma}$ the shortest length of $(\Gamma-1)$-paths
connecting $\mathbf{s}$ and $\sigma$. Then, it holds that 
\begin{equation}
\big|\,h_{\mathcal{S}(A),\,\mathcal{S}(B)}^{\beta}(\sigma)-h_{\mathcal{S}(A),\,\mathcal{S}(B)}^{\beta}(\mathbf{s})\,\big|\le CN_{\sigma}e^{-\beta}\;.\label{e_eqp}
\end{equation}
\item For all $n\in\llbracket2,\,L-2\rrbracket$ and $\zeta\in\mathcal{R}_{n}^{A,\,B}$,
it holds that 
\[
\max_{\sigma\in\mathcal{N}(\zeta)}\big|\,h_{\mathcal{S}(A),\,\mathcal{S}(B)}^{\beta}(\sigma)-h_{\mathcal{S}(A),\,\mathcal{S}(B)}^{\beta}(\zeta)\,\big|\le CL^{2}e^{-\beta}\;.
\]
\end{enumerate}
\end{lem}

The proof of this lemma follows from a well-known standard renewal
argument (cf. \cite[Lemma 8.4]{BdH}) along with rough estimate of
capacities based on the Dirichlet--Thomson principles. Moreover,
the proof is identical to \cite[Lemmas 10.4 and 16.5]{KS2}. Thus,
we omit the detail of the proof. The reason why we have the $L^{2}$-term
in the right-hand side of part (2) comes from explicit computation,
i.e., the number of pairs of configurations $(\xi_{1},\,\xi_{2})$
with $\xi_{1}\in\mathcal{N}(\zeta)$, $\xi_{2}\notin\mathcal{N}(\zeta)$,
and $\xi_{1}\sim\xi_{2}$. 

We next control the factor $N_{\sigma}$ appeared in \eqref{e_eqp}.
Note that this quantitative result was not needed in small volume
regime. 
\begin{lem}
\label{l_eqp2}In the notation of Lemma \ref{l_eqp} with $\mathbf{s}=\mathbf{a}$
for some $a\in A$, we have $N_{\sigma}<4L$ if $(\mathrm{div}\,\psi)(\sigma)\ne0$.
\end{lem}

\begin{proof}
By the definition of $\psi$ that for $\sigma\in\mathcal{N}(\mathbf{s})$,
\begin{equation}
(\mathrm{div}\,\psi)(\sigma)=\sum_{\zeta\in\mathcal{O}^{A}:\,\zeta\sim\sigma}\psi(\sigma,\,\zeta)=\sum_{\zeta\in\mathcal{O}^{A}}\frac{\mathfrak{e}_{A}}{Z_{\beta}\mathfrak{c}}e^{-\Gamma\beta}\times[1-\mathfrak{h}^{A}(\zeta)]\;.\label{e_eqp2}
\end{equation}
Therefore, $(\mathrm{div}\,\psi)(\sigma)\ne0$ if and only if there
exists $\zeta\in\mathcal{O}^{A}$ with $\mathfrak{h}^{A}(\zeta)\ne1$.
Therefore, the statement of lemma is a direct consequence of Lemma
\ref{l_edge2}.
\end{proof}
\begin{lem}
\label{l_divsum}There exists $C>0$ such that
\begin{align}
 & \sum_{\sigma\in\mathcal{N}(\zeta)}|(\mathrm{div}\,\psi)(\sigma)|\le CL^{2}e^{-\Gamma\beta}\;\text{\;\;\;for all }\zeta\in\mathcal{R}_{n}^{A,\,B}\text{ with }n\in\llbracket2,\,L-2\rrbracket\;\text{and}\label{e_divsum}\\
 & \sum_{\sigma\in\mathcal{N}(\mathbf{s})}|(\mathrm{div}\,\psi)(\sigma)|\le CL^{9}e^{-\Gamma\beta}\;\;\;\;\text{for all }\text{\ensuremath{\mathbf{s}\in\mathcal{S}}\;.}\label{e_divsum2}
\end{align}
\end{lem}

\begin{proof}
First, suppose that $\zeta\in\mathcal{R}_{n}^{A,\,B}$ with $n\in\llbracket3,\,L-3\rrbracket$.
Then, we have by \eqref{e_divps} that
\[
\sum_{\sigma\in\mathcal{N}(\zeta)}|(\mathrm{div}\,\psi)(\sigma)|=8L\times\frac{1}{Z_{\beta}}\frac{10\mathfrak{b}e^{-\Gamma\beta}}{\mathfrak{c}(5L-3)(L-4)}\le CL^{-1}e^{-\Gamma\beta}\;,
\]
where the factor $8L$ denotes the number of $\sigma\in\mathcal{N}(\zeta)\setminus\{\zeta\}$.
This proves \eqref{e_divsum} in this case.

Next, suppose that $\zeta\in\mathcal{R}_{2}^{A,\,B}$, say $\zeta\in\mathcal{R}_{2}^{a,\,b}$
for some $(a,\,b)\in A\times B$. In this case, as above, we again
have $(\mathrm{div}\,\psi)(\zeta)=0$ and for $\sigma\in\mathcal{N}(\zeta)$
with $\sigma\in\mathcal{C}_{2,\,\scal{o}}^{a,\,b}$ and $|\mathfrak{p}^{a,\,b}(\sigma)|=1$,
\begin{equation}
|(\mathrm{div}\,\psi)(\sigma)|=\frac{1}{Z_{\beta}}\frac{10\mathfrak{b}e^{-\Gamma\beta}}{\mathfrak{c}(5L-3)(L-4)}\le CL^{-2}e^{-\Gamma\beta}\;.\label{eq:pf1}
\end{equation}
Moreover, if $\sigma\in\mathcal{N}(\zeta)$ with $\sigma\in\mathcal{C}_{1,\,\scal{o}}^{a,\,b}$
and $|\mathfrak{p}^{a,\,b}(\sigma)|=2L-1$, then by the definition
of $\psi$, we have
\[
|(\mathrm{div}\,\psi)(\sigma)|\le\sum_{\xi:\,\xi\sim\sigma}|\psi(\sigma,\,\xi)|=\sum_{\xi:\,\xi\sim\sigma}\frac{\mathfrak{e}_{A}}{Z_{\beta}\mathfrak{c}}e^{-\Gamma\beta}\times|\mathfrak{h}^{A}(\sigma)-\mathfrak{h}^{A}(\xi)|\;.
\]
Since number of such $\xi$ is trivially bounded by $2qL^{2}$, we
can bound the right-hand side using Proposition \ref{p_eAest} by
\begin{equation}
2qL^{2}\times CL^{-1}e^{-\Gamma\beta}=2qCLe^{-\Gamma\beta}\;,\label{eq:pf2}
\end{equation}
where we used $|\mathfrak{h}^{A}(\sigma)-\mathfrak{h}^{A}(\xi)|\le1$.
Therefore by \eqref{eq:pf1} and \eqref{eq:pf2}, we have
\[
\sum_{\sigma\in\mathcal{N}(\zeta)}|(\mathrm{div}\,\psi)(\sigma)|\le4L\times CL^{-2}e^{-\Gamma\beta}+4L\times2qCLe^{-\Gamma\beta}=O(L^{2}e^{-\Gamma\beta})\;,
\]
where the two factors $4L$ denote the number of such possible $\sigma$.
This concludes \eqref{e_divsum} in the case $\zeta\in\mathcal{R}_{2}^{A,\,B}$.
The case $\mathcal{R}_{L-2}^{A,\,B}$ can be proved in the same manner.
Thus, we conclude the proof of \eqref{e_divsum}.

Finally, we prove \eqref{e_divsum2}. We may assume $\mathbf{s}=\mathbf{a}$
for some $a\in A$. By the definition of $\psi$, we have
\[
\sum_{\sigma\in\mathcal{N}(\mathbf{a})}|(\mathrm{div}\,\psi)(\sigma)|=\sum_{\sigma\in\mathcal{N}(\mathbf{a})}\sum_{\zeta\in\mathcal{O}^{A}:\,\sigma\sim\zeta}\frac{\mathfrak{e}_{A}}{Z_{\beta}\mathfrak{c}}e^{-\Gamma\beta}\times|\mathfrak{h}^{A}(\sigma)-\mathfrak{h}^{A}(\zeta)|\;.
\]
Since the summand vanishes if $\mathfrak{h}^{A}(\zeta)=\mathfrak{h}^{A}(\sigma)$,
we can bound the right-hand side by
\[
\sum_{\sigma\in\mathcal{N}(\mathbf{a})}\sum_{\zeta\in\mathcal{O}^{A}:\,\sigma\sim\zeta,\,\mathfrak{h}^{A}(\zeta)\ne1}\frac{\mathfrak{e}_{A}}{Z_{\beta}\mathfrak{c}}e^{-\Gamma\beta}\le\sum_{\sigma\in\mathcal{N}(\mathbf{a})}\sum_{\zeta\in\mathcal{O}^{A}:\,\sigma\sim\zeta,\,\mathfrak{h}^{A}(\zeta)\ne1}CL^{-1}e^{-\Gamma\beta}\;,
\]
where the inequality is induced by Proposition \ref{p_eAest} and
Theorem \ref{t_Gibbs1}-(1). By Lemma \ref{l_edge2}, the number of
such $\sigma$ so that the summand does not vanish is $O(L^{8})$,
and for each such $\sigma$, the corresponding $\zeta$ has at most
$2qL^{2}$ choices. Thus, we conclude
\[
\sum_{\sigma\in\mathcal{N}(\mathbf{a})}|(\mathrm{div}\,\psi)(\sigma)|\le O(L^{8})\times2qL^{2}\times CL^{-1}e^{-\Gamma\beta}=O(L^{9}e^{-\Gamma\beta})\;.
\]
This concludes the proof of Lemma \ref{l_divsum}.
\end{proof}
Now, we are ready to prove Proposition \ref{p_tfl2}.
\begin{proof}[Proof of Proposition \ref{p_tfl2}]
 It is clear from the definition of $\psi$ that $\mathrm{div}\,\psi=0$
on $\widehat{\mathcal{N}}(\mathcal{S})^{c}$. Hence, by Lemmas \ref{l_div1},
\ref{l_div3}, \ref{l_div4}, and \ref{l_div5}, we can write the
left-hand side of \eqref{e_tfl2} as
\begin{equation}
\Big[\,\sum_{n=2}^{L-2}\sum_{\zeta\in\mathcal{R}_{n}^{A,B}}\sum_{\sigma\in\mathcal{N}(\zeta)}+\sum_{\zeta\in\mathcal{S}}\sum_{\sigma\in\mathcal{N}(\zeta)}\,\Big]\,h_{\mathcal{S}(A),\,\mathcal{S}(B)}^{\beta}(\sigma)(\mathrm{div}\,\psi)(\sigma)\;.\label{e_tfl2pf}
\end{equation}
By Lemmas \ref{l_div2}, \ref{l_div6}, \ref{l_eqp}, \ref{l_eqp2},
and \ref{l_divsum}, this equals
\[
\frac{1}{Z_{\beta}\mathfrak{c}}e^{-\Gamma\beta}+L^{2}\times O(L^{2}e^{-\beta})\times O(L^{2}e^{-\Gamma\beta})+O(Le^{-\beta})\times O(L^{9}e^{-\Gamma\beta})=\frac{1+o_{L}(1)}{Z_{\beta}\mathfrak{c}}e^{-\Gamma\beta}\;,
\]
since $L^{10}\ll e^{\beta}$, where the factor $L^{2}$ takes the
possibility of selecting a regular configuration in $\mathcal{R}_{n}^{A,\,B}$,
$n\in\llbracket2,\,L-2\rrbracket$ into account. By Theorem \ref{t_Gibbs1}-(1),
the proof is completed. 
\end{proof}

\subsection{\label{sec104}Proof of Theorem \ref{t_cap}}

First, by gathering the previous proposition with Proposition \ref{p_tfl1},
we can conclude the proof of Proposition \ref{p_tfl}.
\begin{proof}[Proof of Proposition \ref{p_tfl}]
By Propositions \ref{p_tfl1} and \ref{p_tfl2}, we have
\[
\|\psi^{A,\,B}\|_{\beta}^{2}=\frac{1+o_{L}(1)}{q\mathfrak{c}}e^{-\Gamma\beta}\;\;\;\;\text{and}\;\;\;\;\sum_{\sigma\in\mathcal{X}}h_{\mathcal{S}(A),\,\mathcal{S}(B)}^{\beta}(\sigma)(\mathrm{div}\,\psi^{A,\,B})(\sigma)=\frac{1+o_{L}(1)}{q\mathfrak{c}}e^{-\Gamma\beta}\;.
\]
Inserting these to the left-hand side of \eqref{e_tfl} completes
the proof.
\end{proof}
Then, we can now complete the proof of the capacity estimate.
\begin{proof}[Proof of Theorem \ref{t_cap}]
By Theorem \ref{t_DP} and Proposition \ref{p_tf}, we get the upper
bound as
\[
\mathrm{Cap}_{\beta}\big(\,\mathcal{S}(A),\,\mathcal{S}(B)\,\big)\le D_{\beta}(f^{A,\,B})=\frac{1+o_{L}(1)}{q\mathfrak{c}}e^{-\Gamma\beta}\;.
\]
On the other hand, by Theorem \ref{t_gTP} and Proposition \ref{p_tfl},
we get the matching lower bound as 
\[
\mathrm{Cap}_{\beta}\big(\,\mathcal{S}(A),\,\mathcal{S}(B)\,\big)\ge\frac{1}{\|\psi^{A,\,B}\|_{\beta}^{2}}\Big[\,\sum_{\sigma\in\mathcal{X}}h_{\mathcal{S}(A),\,\mathcal{S}(B)}^{\beta}(\sigma)(\mathrm{div}\,\psi^{A,\,B})(\sigma)\,\Big]^{2}=\frac{1+o_{L}(1)}{q\mathfrak{c}}e^{-\Gamma\beta}\;.
\]
The proof is completed by combining these upper and lower bounds.
\end{proof}
\begin{acknowledgement*}
SK was supported by NRF-2019-Fostering Core Leaders of the Future
Basic Science Program/Global Ph.D. Fellowship Program and the National
Research Foundation of Korea (NRF) grant funded by the Korean government
(MSIT) (No. 2018R1C1B6006896). IS was supported by the National Research
Foundation of Korea (NRF) grant funded by the Korean government (MSIT)
(No. 2018R1C1B6006896 and No. 2017R1A5A1015626).
\end{acknowledgement*}

\end{document}